\documentclass[a4paper]{amsart}

\usepackage[utf8]{inputenc}
\usepackage[english]{babel}
\usepackage{amsmath,amsthm,amssymb,hyperref,microtype}
\usepackage[percent]{overpic}
\usepackage[all]{xy}
\usepackage{tikz}
\usetikzlibrary{arrows}
\usepackage{multirow}
\usepackage{fancybox}

\usepackage{hhline}

\newcommand{\R}{\mathbb R} 
\newcommand{\Z}{\mathbb Z}
 
\newcommand{\N}{\mathbb N}

\newcommand{\inv}{^{-1}}
\newcommand{\Sing}{\text{Sing}}
\newcommand{\ind}{\text{ind}}
\newcommand{\defi}{\textbf}

\newcommand{\QSone}{Q_{\mathbb S^1}}
\newcommand{\PF}{A} 

\newcommand{\Trepul}{T_{\sigma\inv}}
\newcommand{\Tattr}{T_{\sigma}}
\newcommand{\Pmin}{\mathcal P_{\text{min}}}
\newcommand{\Pmax}{\mathcal P_{\text{max}}}
\newcommand{\ssm}{\smallsetminus}

\newcommand{\upor}[1]{\overset{\leadsto}{#1}}   %

\newcommand{\circlebox}[2]{\rule[-1.2ex]{0pt}{4ex}\tikz[baseline = (text.base)]{\node[thick,rounded corners=3pt, draw=#1, text=#1] (text) {#2};}}

\newtheorem{thm}{Theorem}
\newtheorem{lem}{Lemma}[section]

\newtheorem{prop}[lem]{Proposition}

\newtheorem*{thm*}{Theorem}
\newtheorem*{prop*}{Proposition}
\newtheorem*{thm-main*}{Theorem~\ref{thm:main}}

\theoremstyle{definition} 
\newtheorem{defn}[lem]{Definition}
\newtheorem*{defn*}{Definition} 
\newtheorem{rem}[lem]{Remark}
\newtheorem*{rem*}{Remark} 
\newtheorem{example}[lem]{Example}

\newtheorem{quest}[lem]{Question}
\numberwithin{equation}{section} 
\newtheorem{algo}[lem]{Algorithm}

\begin{document}

\title{Tree substitutions and {R}auzy fractals}

\author{Thierry Coulbois}
\author{Milton Minervino}

\address{Institut de mathématiques de Marseille (i2m), 39 rue F.~Joliot-Curie, 13013 Marseille, France}

\email{thierry.coulbois@univ-amu.fr, milton.minervino@univ-amu.fr}

\date{\today}



\thanks{The second author was founded by Labex Archimède - Université d'Aix-Marseille}

\begin{abstract}
  We work with attracting subshifts generated by substitutions which
  are also irreducible parageometric automorphisms of free groups. For
  such a dynamical system, we construct a tree substitution to
  approximate the repelling real tree of the automorphism. We produce
  images of this tree inside the Rauzy fractal when the substitution
  is irreducible Pisot. We describe the contour of this tree
  and compute an interval exchange transformation of the circle
  covering the original substitution.
\end{abstract}

\maketitle

\section{Introduction}

\subsection{Results}

The main objects of our work are attracting subshifts generated by a
substitution $\sigma$ over a finite alphabet $A$. These are also
called substitutive subshifts  and are defined as the action of the
shift $S$ on the set $X_\sigma$ of bi-infinite words all of whose
finite factors are factors of some iterations of the substitution on
some letter.  The word ``attracting" reflects that under iterations of
the substitution any letter $a \in A$ converges to the attracting
subshift $X_\sigma$.

Our motivation is to find a suitable geometrical interpretation for
these subshifts.

Rauzy fractals provide such a geometric interpretation. If $\sigma$ is
irreducible Pisot, there exists a continuous map $\varphi:X_\sigma\to
E_c$ from the attracting subshift $X_\sigma$ into the contracting
space $E_c$ of the abelianization matrix $M_\sigma$. The image
$\mathcal R_\sigma=\varphi(X_\sigma)$ is called the Rauzy fractal, it
is a compact subset equipped with a domain exchange and the contracting 
action of the matrix $M_\sigma$, which are the push-forward of the shift 
map and of the substitution respectively. This celebrated construction
goes back to the seminal work of Rauzy~\cite{rauzy} and was later
generalized by Arnoux and Ito~\cite{arnoux-ito}.

Another geometric interpretation arises when $\sigma$ is also a fully
irreducible (abbreviated as iwip) automorphism of the free group $F_A$
on the alphabet $A$. Following the work of Bestvina and Handel,
Gaboriau, Jaeger, Levitt and, Lustig~\cite{gjll} described the repelling
tree of such an iwip automorphism. This is a real tree $\Trepul$ with
an action of the free group by isometries and a homothety $H$. The
homothety is contracting by a factor $\frac{1}{\lambda_{\sigma\inv}}$,
where $\lambda_{\sigma\inv}$ is the expansion factor of the inverse
automorphism of $\sigma$. Later a continuous map $Q:X_\sigma\to
\overline{T}_{\sigma^{-1}}$, where $\overline{T}_{\sigma^{-1}}$ is the
metric completion of the repelling tree, was defined~\cite{ll-north-south,chl4}. The image $\Omega_A=Q(X_\sigma)$ is
a compact subset and the shift map is pushed-forward through $Q$ to
the actions of the elements of $A^{\pm 1}$, while the action of
$\sigma$ is pushed-forward to the homothety.

It was remarked~\cite{kl-diagonal-closure,chr} that these two
geometric interpretations are one above the other: when $\sigma$ is both
irreducible Pisot and iwip, there exists a continuous equivariant map
$\psi:\Omega_A\to\mathcal R_\sigma$ that makes the diagram commute:
\[
\xymatrix{
X_\sigma\ar[dr]_\varphi\ar[r]^Q & \Omega_A\ar[d]^\psi \\ & \mathcal
  R }
\]
When $\sigma$ is a parageometric iwip automorphism, $\Omega_A$ is
a connected subset of $\overline{T}_{\sigma^{-1}}$: a compact
$\R$-tree, called the compact heart. The first aim of this paper is to provide approximations of
the compact heart and to draw them inside the Rauzy fractal.

The approximation of the repelling tree (or rather of its subtree
$\Omega_A$) is achieved by a tree
substitution. Jullian \cite{jullian-coeur} introduced this version of
graph-directed self-similar systems (in the sense of~\cite{mw}) and
described some examples, in particular in the case of the Tribonacci
automorphism $\sigma:a\mapsto ab, b\mapsto ac,c\mapsto a$. We
generalize this construction: we construct an abstract
tree substitution which converges to the repelling tree of the
automorphism.

\begin{thm}[Proposition~\ref{prop:tree-substitution-converge} and Algorithms~\ref{algo:singular-points} and \ref{algo:gluing-and-tau}]\label{thm:abstract-tree}
Let $\sigma$ be a primitive substitution on a finite alphabet $A$ and
a parageometric iwip automorphism of the free group $F_A$. Then there
exists a tree substitution $\tau$ such that the iterations of $\tau$
on the initial tree $W$ renormalized by the ratio of the contracting
homothety converge to the repelling tree of the automorphism $\sigma$:
\[
\lim_{n\to\infty}\bigg(\frac 1{\lambda_{\sigma\inv}}\bigg)^n\tau^n(W)=\Omega_A\subseteq\overline{T}_{\sigma^{-1}}.
\]
We provide an algorithm to construct this tree substitution.
\end{thm}

The key of our algorithm is to describe the vertices of the tree $W$
by singular bi-infinite words.  The singular bi-infinite words of the
attracting shift are also known in the literature as asymptotic pairs,
proximal pairs or special infinite words
\cite{cassaigne-speciaux, bd-complete-invariant,bd-proximality}. In
the train-track dialect finding singular words amounts to compute
periodic Nielsen paths for the automorphism $\sigma$. They exactly
correspond to pairs of distinct bi-infinite words with the same
image by $Q$.

Turning back to the Pisot hypothesis, using the map $\varphi$ we can
regard the tree substitution inside the contracting space. More
precisely we can use as vertices the images by $\varphi$ in the Rauzy
fractal of the singular words.  The surjectivity of the map $\varphi$
implies that the iterates under the tree substitution of the set of
singular points fills the Rauzy fractal. If we connect them, for
instance with line segments, we get a
picture of a tree reminding of a Peano curve inside the Rauzy
fractal. 

However, there is no guarantee that the tree substitution yields a
tree in the contracting space. Indeed, the map $\varphi$ is not injective
and distinct singular words can be misleadingly identified. This is
the case in all our examples. Nevertheless, we provide a way to overcome
this difficulty by using coverings of the tree substitution. A
covering consists in extending the alphabet to a bigger one $\tilde
A$, defining a forgetful map $f:\tilde A\to A$ and allowing different
prototiles $W_{\tilde a}$ for different $\tilde a\in f\inv(a)$ but
keeping the same self-similar structure. In practice we prune the tree
substitution, i.e. we erase some branches of the tree which are
dead-end and which create loops when iterating the embedded tree
substitution. This results in extending the definition of the tree
substitution to a larger number of initial subtrees.

Summing up the above discussion we get the following theorem.

\begin{thm}[Sections~\ref{sec:tree-substitution-in-Rauzy} and \ref{sec:pruning-cheating}]\label{thm:embedded-tree}
Let $\sigma$ be an irreducible Pisot substitution on a finite alphabet
$A$ and a parageometric iwip automorphism of the free group
$F_A$. Then, the tree substitution $\tau$ can be realized inside the
contracting space $E_c$ of $\sigma$ and the renormalized iterated
images $M_\sigma^n\tau^n(W)$ converge to the Rauzy fractal
$\mathcal{R}_\sigma$.

In all our examples we provide a covering tree
substitution which yields trees inside the contracting space.
\end{thm}

We will show some pictures of these trees in the contracting space
like in Figure~\ref{fig:peano-tree-tribo-thirteen}.  Such pictures
where already obtained by Arnoux~\cite{A88} (for the Tribonacci
substitution) and Bressaud and Jullian~\cite{BJ12} (as well as some other
examples); in \cite{BDJP,DPV} the tree associated with the Tribonacci substitution is
obtained by connecting adjacent cubes of a thickened version of the
stepped approximation of the contracting plane.

\begin{figure}[h]
  \includegraphics[width=\textwidth, height=.7\textwidth]{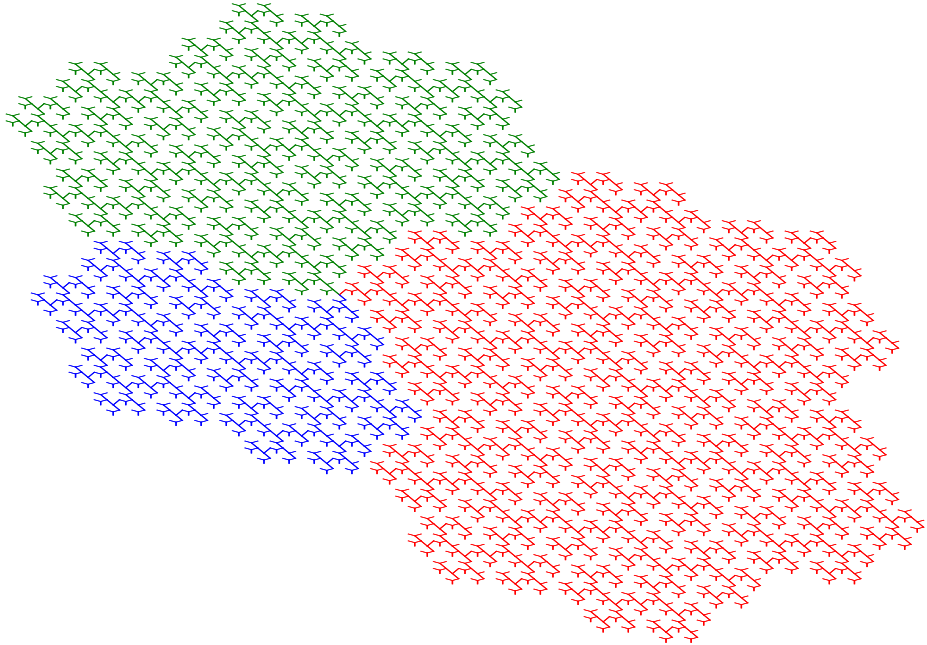}
  \caption{\label{fig:peano-tree-tribo-thirteen} Twelfth iterate of the
    tree substitution for the Tribonacci automorphism $\sigma:a\mapsto
    ab, b\mapsto ac,c\mapsto a$ inside the Rauzy fractal.}
\end{figure}

The last geometric interpretation we have in mind are interval
exchange transformations. We are looking for an interpretation of the
original attracting shift $X_\sigma$ as an interval exchange
transformation.

We take inspiration mainly from the Arnoux-Yoccoz interval
exchange~\cite{arnoux-yoccoz}. Arnoux and Yoccoz studied a certain
exchange of six intervals on the unit circle. They observed that this
exchange is conjugate to its first return map into a subinterval,
that is, it is self-induced. Moreover, it is measurably conjugate to
the attracting subshift defined by the Tribonacci substitution
$\sigma:a\mapsto ab$, $b\mapsto ac$, $c\mapsto a$. At the time, this
transformation provided one of the first non-trivial examples of
pseudo-Anosov automorphism with a cubic dilatation factor, namely the
dominant root of the polynomial $x^3-x^2-x-1$. Arnoux-Rauzy words
\cite{arnoux-rauzy} are another example of codings of six interval
exchanges.

For a substitution which is a parageometric iwip automorphism,
Bressaud and Jullian~\cite{BJ12} proved that the contour of the
compact heart $\Omega_A$ always provides a self-induced interval
exchange of the circle.

We provide an algorithm to compute the contour substitution. Using our
tree substitution $\tau$, we described the repelling tree $\Trepul$
and its compact heart $\Omega_A$ by the iterated images
$\tau^n(W)$. Once the cyclic orders at branch points of the finite
tree $W$ and $\tau(W)$ are chosen and satisfy the compatibility
conditions
\ref{condition:simplicial-isomorphism}-\ref{condition:Vershik-compatible},
we get a contour substitution $\chi$ and we prove that it is primitive
with same dominant eigenvalue as $\sigma$. It turns out that $\chi$
has a dual substitution $\chi^*$ and we get a continuous map
$\QSone:X_{\chi^*}\to\mathbb S^1$ from the attracting shift of
$\chi^*$ to the circle. The shift map and $\chi^*$ are pushed forward
by $\QSone$ to a piecewise exchange of the circle and to a piecewise
homothety respectively. The dual substitution $\chi^*$ defined on the
bigger alphabet $\tilde A$ covers $\sigma$: there exists a forgetful
map $f:\tilde A\to A$ such that $f\circ\chi^*=\sigma\circ f$. The
attracting shift $X_{\chi^*}$ factors onto $X_\sigma$.

\begin{thm}[Sections~\ref{sec:contour} and \ref{sec:vershik-map}]\label{thm:contour-iet}
Let $\sigma$ be a primitive substitution on a finite alphabet $A$
and a parageometric iwip automorphism of the free group
$F_A$.
Let $\tau$ be the tree substitution constructed in
Theorem~\ref{thm:abstract-tree}. Then, there exist cyclic orders on
the initial trees $W$ and $\tau(W)$ which define a contour
substitution $\chi$ and a dual contour substitution $\chi^*$ such that:
\begin{enumerate}
\item the substitutions $\chi$ and $\chi^*$ are primitive and the
  iterated images of the contour of $W$ under $\chi$ converge to the
  contour of the compact heart $\Omega_A$ of the repelling tree;
\item there exists a continuous map $\QSone: X_{\chi^*}\to\mathbb S^1$
  which pushes forward the action of the shift on $X_{\chi^*}$ to an
  interval exchange on the contour circle of $\Omega_A$ induced by the
  action of the elements of $A^{\pm 1}$ on the repelling tree $\Trepul$. 
  The action of $\chi^*$ is pushed forward to a piecewise contracting
  homothety of ratio $\frac{1}{\lambda_\sigma}$.
\end{enumerate}
We provide an algorithm to compute $\chi$ and $\chi^*$.
\end{thm}

We also refer to the work of Sirvent~\cite{Sir1,Sir2} on self-similar
interval exchange of the circle associated to some particular
substitutions. The trees we construct are dual to the geodesic
laminations obtained in his works.

It is conjectured that each irreducible Pisot
substitutive subshift has pure discrete spectrum: this is known as the
Pisot conjecture (see \cite{ABBLS15} for a survey). Pure discreteness
of the spectrum can be proved geometrically by showing that the Rauzy
fractal associated with an irreducible Pisot substitution tiles
periodically the contracting space. In this case, the substitutive
subshift is measurably conjugate to a translation on a torus. We
expect that our constructions of the tree substitution and of the
contour interval exchange could shed new light and open new techniques
to attack the Pisot conjecture, at least when the substitution is a
parageometric iwip automorphism.

\subsection{Techniques and self-similarity}

All the objects we consider, that is, attracting shifts, repelling
trees and their compact hearts, Rauzy fractals, tree substitutions,
contour interval exchanges of the circle, are governed by
self-similarity. This self-similarity is best described using the
prefix-suffix automaton of the
substitution. Mossé~\cite{mosse-reconnaissabilite} proved that any
bi-infinite word in the attracting subshift of a primitive
substitution can be uniquely desubstituted: this allows to define a
map $\Gamma:X_\sigma\to \mathcal{P}$ from the attracting shift to the
set $\mathcal P$ of infinite desubstitution paths.  For a finite path
$\gamma$ in the prefix-suffix automaton we consider the cylinder
$[\gamma]$ of the bi-infinite words in $X_\sigma$ with desubstitution
paths ending by $\gamma$. Then $X_\sigma$ decomposes as the disjoint
union of cylinders $[\gamma]$, for all paths $\gamma$ of length $n$ in
the prefix-suffix automaton. This self-similarity is passed through
$Q$ to the repelling tree and through $\varphi$ to the Rauzy fractal:
\[ \Omega_A=\bigcup_\gamma p(\gamma)\inv H^n(\Omega_a),\quad \mathcal
  R_\sigma= \bigcup_\gamma M_\sigma^{n}\mathcal R_\sigma(a) +
  \pi_c(\ell(p(\gamma))) \]
where $a$ is the starting letter of $\gamma$, $p(\gamma)$ is its
prefix, $\ell$ is the abelianization map, $\pi_c$ the projection to
the contracting space and $H$ is the contracting homothety.
If $\sigma$ is a parageometric iwip automorphism the tiles
$\Omega_\gamma$ are compact trees with at most one point in common,
while with the irreducible Pisot hypothesis the tiles $\mathcal
R_\gamma$ are disjoint in measure, if the strong coincidence
condition holds.

The tree substitution of Theorem~\ref{thm:abstract-tree} reflects the 
self-similarity in the repelling tree. Indeed, iterates of the 
substitution $\tau$ on an initial patch $W$ are made of  tiles
$W_\gamma$, each of them isomorphic to one of the finitely many
 prototiles $(W_a)_{a\in A}$:
\[
\tau^n(W)=\bigcup_\gamma W_\gamma.
\]

The key of our algorithm to construct the tree substitution is to
describe how to glue together the abstract trees $W_\gamma$ to mimic
the self-similarity of $\Omega_A$.  These gluing instructions, or
adjacency relations, between the tiles $W_\gamma$, are governed by
pairs of singular bi-infinite words which are exactly the pairs with the same
$Q$-image in the repelling tree.
From the works of Queffelec~\cite{quef-book} and, Holton and
Zamboni~\cite{hz-directed} there are finitely many singular words and
their desubstitution paths are eventually periodic. 

The Rauzy fractal is also the limit
of projected and renormalized iterates of the dual substitution
$E_1^*(\sigma)$ acting on codimension one faces of the unit
hypercube. We denote such a face based at $x$ and orthogonal to the
basis vector $e_i$ by $(x,i)^*$. The dual substitution is defined by
\[
E_1^*(\sigma)^n\left(\bigcup_{i\in
  A}(0,i)^*\right)=\bigcup_\gamma \left(M_\sigma^{-n}\ell(p(\gamma)),a\right)^*. 
\]
This is once again the same self-similarity and, we also use the dual
substitution to draw our tree substitution in the contracting space.

Regarding the construction of the contour map and the piecewise
exchange on the circle, we consider infinite paths in the
prefix-suffix automaton of $\chi$ and $\chi^*$. In this way, the shift
map on the attracting shift $X_{\chi^*}$ is translated through
$\Gamma$ to what is called the Vershik map. 
Pushing forward by $\QSone$ the Vershik map gives the piecewise rotation 
on the circle.

We implemented our algorithms in Sage~\cite{sagemath}.

\subsection{Perspectives}

At the end of this work several questions remain open.

Our construction only works for parageometric iwip automorphisms but
we expect that (with some more technicalities) the repelling tree of
any iwip substitution can be described by a tree substitution. Recall
that from the botany obtained with Hilion~\cite{ch-b} the attracting
tree of a non-parageometric iwip is of Levitt type: the limit set
$\Omega_A$ and the cylinders $\Omega_\gamma$ are no longer trees but
rather Cantor sets inside the repelling tree. Thus, the adjacency
relation between tiles is more complicated to describe.

Our work raises the question of which substitutions can be covered by
an interval exchange transformation. This is the purpose of the (dual)
contour substitution and we ask if for any substitution there exists a
bigger alphabet $\tilde A$, a forgetful map $f:\tilde A\to A$ and
a substitution $\tilde\sigma$ such that $f\circ\tilde\sigma=\sigma\circ
f$ and such that $\tilde\sigma$ codes an interval exchange
transformation.

The trees we get are also given by adjacency relations of tiles inside
the repelling tree. Through the map $Q$ these adjacency relations are
preserved inside the Rauzy fractal. But is it true that they are
satisfied for the dual substitution or some variation of it?

Our original goal was to draw trees in the contracting space. Is it
true that, for any irreducible Pisot substitution which is an iwip 
parageometric automorphism, there exists a tree substitution which can be 
realized injectively inside the contracting space (that is, are the iterated images $\tau^n(W)$ visualized as trees therein)?

The map $Q$ implies that the Rauzy fractal of a parageometric Pisot
substitution is arcwise connected. We ask if furthermore it is always
disk-like.

Finally, we have in mind the Pisot conjecture. Do our constructions
give informations on the spectrum of the original substitution?

\section{Preliminaries and notations}

\subsection{Substitutions, attracting shift and prefix-suffix automaton}\label{sec:subst-attr-shift}

Let $A$ be a finite alphabet, $A^*$ be the free monoid on $A$, where
the operation is the concatenation of words.  The full bi-infinite
shift is the space $A^\Z$ of bi-infinite words. For $w\in A^\Z$, we
denote by $w_i$ the letter at position $i$ and by $w_{[i,j]}$ the
factor $w_iw_{i+1}\cdots w_j$ between positions $i$ and $j$. We put
$w_0$ right after a dot. For example
\[ w=\cdots aaaa \cdot bbbb\cdots,\quad w_{[-2,3]}=aabbbb \]

It comes equipped with the shift operator
$S:(w_n)_{n\in\mathbb{Z}}\mapsto(w_{n+1})_{n\in\mathbb{Z}}$ and with a
topology given by clopen cylinders: for a finite word $u$, and for
$i\in\Z$ we define the cylinder $[u]_i$ as the subset of bi-infinite
words which read $u$ at position $i$:
\[ [u]_i =\{w\in A^\Z : w_{i}\cdots w_{i+|u|-1} =u\}\]
where $|u|$ denotes the length of the word $u$.  We extend this
notation by letting
\[
[w]=[w]_0\quad\text{ and }\quad [u\cdot v]=\{w\in A^\Z : w_{-|u|}\cdots w_{-1}w_0\cdots w_{|v|-1} = uv\}.\]
Let $\sigma$ be a substitution on $A$, that is, a free monoid
endomorphism $\sigma:A^*\to A^*$ which is completely determined by its
restriction $\sigma:A\to A^*$. 

\begin{example}\label{exsub}
Here is a list of substitutions which we will use as examples:
\[
\begin{array}[t]{rcl}a&\mapsto&ab\\b&\mapsto&ac\\c&\mapsto&a\end{array}, 
\quad\begin{array}[t]{rcl}a&\mapsto&ac\\b&\mapsto&ab\\c&\mapsto&b\end{array},
\quad\begin{array}[t]{rcl}a&\mapsto&abc\\b&\mapsto&bcabc\\c&\mapsto&cbcabc\end{array}
\]
The leftmost substitution is known as the Tribonacci substitution.
\end{example}

The \defi{language} of $\sigma$ is defined by
\[ \mathcal{L}_\sigma = \{w\in A^* : w \text{ is a factor of } \sigma^n(i) \text{ for some } i \in A, n\in \mathbb{N} \}. \]
The \defi{attracting shift} $X_\sigma$ of a substitution $\sigma$ is the
set of bi-infinite words in $A^\Z$ whose factors are in $\mathcal{L}_\sigma$. This is
a closed (indeed compact) shift-invariant subset of $A^\Z$. 
We can define the action of $\sigma$ on $A^\Z$ by letting
\[
\sigma(\cdots a_{-2}a_{-1}\cdot a_0a_1a_2\cdots)=\cdots\sigma(a_{-2})\sigma(a_{-1})\cdot \sigma(a_0)\sigma(a_1)\sigma(a_2)\cdots
\]
The attracting shift is invariant by the action of $\sigma$.

The \defi{prefix-suffix automaton} of $\sigma$ has the alphabet $A$ as set of
states and a transition $a\stackrel{p,s}{\longleftarrow}b$ whenever $\sigma(b)=pas$ for $(p,a,s)\in A^*\times A\times A^*$. We let $\mathcal{P}$ be the set of infinite paths $\gamma$ in the prefix-suffix automaton.

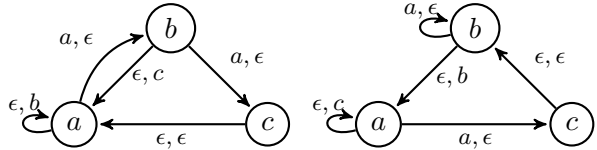
\begin{figure}[h]
\centering
\begin{tikzpicture}[->,>=stealth',shorten >=1pt,auto,node distance=1.8cm,
                    thick,main node/.style={circle,draw,font=\sffamily\bfseries}]

  \node[main node] (1) {\large $a$};
  \node[main node] (2) [above right of=1] {\large $b$};
  \node[main node] (3) [below right of=2] {\large $c$};
  
  \path[every node/.style={font=\sffamily\small}]
    (1) edge [bend left] node [above left] {$a,\epsilon$} (2)    
        edge [loop left] node [above] {$\epsilon,b$} (1)
    (2) edge [] node[right] {$\epsilon,c$} (1)
    (2) edge [] node[above right] {$a,\epsilon$} (3)
    (3) edge [] node[below] {$\epsilon,\epsilon$} (1);
\end{tikzpicture}\,
\begin{tikzpicture}[->,>=stealth',shorten >=1pt,auto,node distance=1.8cm,
                    thick,main node/.style={circle,draw,font=\sffamily\bfseries}]

  \node[main node] (1) {\large $a$};
  \node[main node] (2) [above right of=1] {\large $b$};
  \node[main node] (3) [below right of=2] {\large $c$};
  
  \path[every node/.style={font=\sffamily\small}]
    (1) edge [] node [below] {$a,\epsilon$} (3)    
    (1) edge [loop left] node [above] {$\epsilon,c$} (1)
    (2) edge [loop left] node [above] {$a,\epsilon$} (2)
    (2) edge [] node[right] {$\epsilon,b$} (1)
    (3) edge [] node[above right] {$\epsilon,\epsilon$} (2);
\end{tikzpicture}\,
\caption{The prefix-suffix automatons of the first two substitutions of Example~\ref{exsub}.}
\end{figure}

Mossé~\cite{mosse-reconnaissabilite} proved that every primitive substitution satisfies a property of bilateral recognizability, i.e.~every bi-infinite word $w\in X_\sigma$ can be ``desubstituted'' uniquely: there exist a unique bi-infinite word $w'$ and a unique integer $k$ with  $0\le k\le |\sigma(w_0')|$ such that $w=S^k(\sigma(w'))$. We denote by $\theta:X_\sigma\to X_\sigma$, $w\mapsto w'$ this map.
We remark that $k=|p|$ where $w_0\stackrel{p,s}{\longleftarrow}w_0'$ is an edge of the prefix-suffix automaton.

The \defi{desubstitution map} $\Gamma: X_\sigma\to \mathcal{P}$ associates to $w\in X_\sigma$ the infinite path $\gamma\in\mathcal{P}$ defined by applying recursively $\theta$ to $w$:
\[
\gamma=a_0\stackrel{p_0,s_0}{\longleftarrow}a_1\stackrel{p_1,s_1}{\longleftarrow}a_2\stackrel{p_2,s_2}{\longleftarrow}\cdots\quad (\text{with } a_0 = w_0).
\] 
We call $\gamma$ the \defi{infinite desubstitution path}, or the \defi{prefix-suffix expansion}, of $w$.

\begin{prop}[\cite{cs-pref-suff,hz-directed}]
The map $\Gamma$ is continuous, onto and one-to-one except on the orbit of periodic points of $\sigma$. 
Furthermore $\Gamma(X_\sigma^\mathrm{per}) = \mathcal{P}_\mathrm{min}$ and $\Gamma(S^{-1}X_\sigma^\mathrm{per}) = \mathcal{P}_\mathrm{max}$, where $\mathcal{P}_\mathrm{min}$ and $\mathcal{P}_\mathrm{max}$ are the infinite desubstitution paths with empty prefixes and empty suffixes respectively.
\end{prop}

The last proposition implies that, if $w\in X_\sigma$ is such that its prefix-suffix expansion $\gamma$ does not have $p_i=\epsilon$ or $s_i=\epsilon$ for all $i\ge i_0$ for some $i_0\in\mathbb{N}$, then we can express $w$ as  
\[
w = \lim_{n\to\infty} \sigma^{n}(p_n)\sigma^{n-1}(p_{n-1})\cdots\sigma(p_1)p_0\cdot w_0s_0\sigma(s_1)\cdots\sigma^{n-1}(s_{n-1})\sigma^{n}(s_{n}).
\]

\begin{example}\label{ex:tribo-fix-points}Let $\sigma$ be the Tribonacci substitution. We consider the three periodic points
\begin{align*}
w_a&=\lim_{n\to\infty}\sigma^{3n}(a)\cdot\sigma^n(a)=\cdots abacab abacaba\cdot abacabaabacabab\cdots \\
w_b&=\lim_{n\to\infty}\sigma^{3n}(b)\cdot\sigma^n(a)= \cdots abacaba abacab\cdot abacabaabacabab\cdots \\
w_c&=\lim_{n\to\infty}\sigma^{3n}(c)\cdot\sigma^n(a)= \cdots abacaba abac\cdot abacabaabacabab\cdots
\end{align*}
The infinite desubstitution path of $w_a$, $w_b$ and $w_c$ is
$\gamma_0$ which we describe as the unique periodic path of the graph
\begin{center}
\begin{tikzpicture}[->,>=stealth',shorten >=1pt,auto,node distance=2cm,
                    thick,main node/.style={circle,draw,font=\sffamily\bfseries}]

  \node[main node] (1) {\large $a$};
  
  \path[every node/.style={font=\sffamily\small}]
    (1) edge [loop right] node [right] {$\epsilon,b$} (1);
\end{tikzpicture}\end{center}
while those of $S^{-1}(w_a)$, $S^{-1}(w_b)$ and $S^{-1}(w_c)$ are
$\gamma_a$, $\gamma_b$ and $\gamma_c$ which are the unique infinite
paths of the graph
\begin{center}
\begin{tikzpicture}[->,>=stealth',shorten >=1pt,auto,node distance=1.8cm,
                    thick,main node/.style={circle,draw,font=\sffamily\bfseries}]

  \node[main node] (1) {\large $a$};
  \node[main node] (2) [above right of=1] {\large $b$};
  \node[main node] (3) [below right of=2] {\large $c$};
  
  \path[every node/.style={font=\sffamily\small}]
    (1) edge [] node [above left] {$a,\epsilon$} (2)    
    (2) edge [] node[above right] {$a,\epsilon$} (3)
    (3) edge [] node[below] {$\epsilon,\epsilon$} (1);
\end{tikzpicture}
\end{center}
ending respectively at $a$, $b$ or $c$.
\end{example}

For a finite path $\gamma$ in the prefix-suffix automaton
\[
\gamma=a_0\stackrel{p_0,s_0}{\longleftarrow}a_1\stackrel{p_1,s_1}{\longleftarrow}a_2\stackrel{p_2,s_2}{\longleftarrow}\cdots\stackrel{p_{n-1},s_{n-1}}{\longleftarrow}a_n
\]
its length is denoted by $|\gamma|$ and equals the number $n$ of edges.
We say that $a_0$ is the end and $a_n$ is the beginning of the path.

The \defi{cylinder} $[\gamma]$ is the set of bi-infinite words $w \in X_\sigma$ with infinite desubstitution path ending with $\gamma$. Remark that all words $w$ in $[\gamma]$ have indexed factor:
\[
\sigma^{n-1}(p_{n-1})\cdots\sigma(p_1)p_0\cdot a_0s_0\sigma(s_1)\cdots\sigma^{n-1}(s_{n-1}).
\]
It is convenient to introduce the notation
\[ p(\gamma) := \sigma^{n-1}(p_{n-1})\cdots\sigma(p_1)p_0 \]
to denote what we call the \defi{prefix of the path} $\gamma$. 
Given two finite paths 
\[
\gamma = a_0\stackrel{p_0,s_0}{\longleftarrow}\cdots\stackrel{p_{n-1},s_{n-1}}{\longleftarrow}a_n, \quad
\gamma' = a_n\stackrel{p_n,s_n}{\longleftarrow}\cdots\stackrel{p_{m-1},s_{m-1}}{\longleftarrow}a_{m},
\]
we denote their concatenation by $\gamma\gamma'$ and we say that
$\gamma'$ \defi{extends} $\gamma$, $\gamma$ is the \defi{head} and
$\gamma'$ the \defi{tail}. In Section~\ref{sec:not-renorm-sing} we
will use the notation $B(\gamma)$ to \defi{behead} the prefix-suffix
expansion $\gamma$ of its heading edge, thus with the above notation:
\[
B^n(\gamma\gamma')=\gamma'.
\]

\begin{prop}\label{prop:cylinders-autosimilar}
For any path $\gamma\gamma'$ in the prefix-suffix automaton with $n=|\gamma|$ we have
\[
[{\gamma\gamma'}]=
 S^{|p_0|}\, \sigma\, S^{|p_1|}\,\sigma\,\cdots\, S^{|p_{n-1}|}\,\sigma([{\gamma'}]).
\]
Furthermore we have the following prefix-suffix decomposition of cylinders: for any path
$\gamma$ in the prefix-suffix automaton and every $n\in\N^*$
\[
[\gamma]=\bigsqcup_{\gamma': |\gamma'|=n}[{\gamma\gamma'}],\quad \text{ and in
  particular }\quad X_\sigma=\bigsqcup_{\gamma': |\gamma'|=n}[{\gamma'}],
\] 
where the disjoint union is taken over all paths $\gamma'$ in the
prefix-suffix automaton extending $\gamma$ (or ending at any state in
case $\gamma$ is the empty path) of length $n$.
\end{prop} 

\subsection{Singular Words}\label{sec:singular-leaves}

Two bi-infinite words $Z$ and $Z'$ in the attracting shift $X_\sigma$ \defi{share a half} if either $Z_{[0;+\infty)}=Z'_{[0;+\infty)}$ or $Z_{(-\infty;-1]}=Z'_{(-\infty;-1]}$. 
We say that $Z$ and $Z'$ are \defi{singular} if either they share their right half and $Z_{-1}\neq Z'_{-1}$ or they share their left half and $Z_0\neq Z'_0$. 

In the dialect of word combinatorics, if $Z$ and $Z'$ share their right half, i.e.~$Z_{[0;+\infty)}=Z'_{[0;+\infty)}$, and their letters $Z_{-1}$ and $Z'_{-1}$ are distinct then the common half $Z_{[0;+\infty)}$ is a \defi{left-special word} (and similarly for \defi{right-special} left-infinite words). 

In the dialect of dynamical system, \defi{asymptotic pairs} are equivalence classes of words $Z,Z'$ such that for $\ell$ big enough $S^{\ell}Z$ and $S^{\ell}Z'$ share their right half.  

In the case of the attracting shift of a primitive substitution
Queffelec proved that there are only finitely many singular words, see
also~\cite{bdh-asymptotic-orbits}.

\begin{prop}[\cite{quef-book}]\label{prop:finite-index}
Let $\sigma$ be a primitive substitution, then there are finitely many singular bi-infinite words in the attracting shift $X_\sigma$. \qed
\end{prop}

\begin{example}\label{ex:tribo-right-left-special}
For the Tribonacci substitution $\sigma$ the three infinite words $w_a$, $w_b$ and $w_c$ of Example~\ref{ex:tribo-fix-points} are the only infinite words in the attracting shift with a left-special right half (and they share this right half). The three infinite words
\begin{align*}
w'_a &= \lim_{n\to\infty} \widetilde{\sigma}^{3n}(a)\cdot\widetilde{\sigma}^{3n}(a) = \cdots bacabaabacba\cdot abacababacaba\cdots\\
w'_b &= \lim_{n\to\infty} \widetilde{\sigma}^{3n}(a)\cdot\widetilde{\sigma}^{3n}(b) = \cdots bacabaabacba\cdot bacabaabacaba\cdots
\\
w'_c &= \lim_{n\to\infty} \widetilde{\sigma}^{3n}(a)\cdot\widetilde{\sigma}^{3n}(c) = \cdots bacabaabacba\cdot cabaabacaba\cdots
\end{align*}
where $\widetilde{\sigma}:a\mapsto ba$, $b\mapsto ca$, $c\mapsto a$ is the flipped Tribonacci substitution, are the only infinite words in $X_\sigma$ with a right-special left half. They come from the periodic Nielsen path (of period 3)
\[
\begin{tikzpicture}
\draw (0,0) to [out=0,in=180] node[near end,above,sloped] {$ab$} (.7,.4);
\draw (0,0) to [out=0,in=180] node[near end,right] {$\;ba$} (.7,0);
\draw (0,0) to [out=0,in=180] node[near end,below,sloped] {$ca$} (.7,-.4);
\end{tikzpicture}
\]
Indeed, applying $\sigma$ three times to these three words we get that they coincide on the prefix $abacaba$ and after that each of them starts again with $ab$, $ba$ and $ca$ respectively. This process is responsible for the fact that $w_a'$, $w_b'$ and $w_c'$ are generated by $i_{abacaba}\circ\sigma^3 = \widetilde{\sigma}^3$, where $i_{w}(u)=w^{-1}uw$ denotes the inner automorphism.
\end{example}

Using the work of Holton and Zamboni~\cite{hz-directed} we get:

\begin{prop}[{\cite[Theorem~5.1]{hz-directed}}]\label{prop:singular-preperiodic}
Let $Z$ be a singular bi-infinite word in the attracting shift $X_\sigma$ of a primitive substitution. Then, the infinite de-substitution path $\gamma$ of $Z$ is eventually periodic.
\end{prop}

\begin{example}\label{ex:tribo-periodic-desubstitution}
For the Tribonacci substitution, we know from Example~\ref{ex:tribo-fix-points} the common infinite desubstitution path of $w_a$, $w_b$ and $w_c$: 
\[
\gamma_0=a\stackrel{\epsilon,b}{\longleftarrow}a
\stackrel{\epsilon,b}{\longleftarrow}a
\stackrel{\epsilon,b}{\longleftarrow}a\cdots.
\]
The infinite desubstitution paths of $w'_a$, $w'_b$ and $w'_b$ from Example~\ref{ex:tribo-right-left-special} are
\[
\gamma_a=a\stackrel{\epsilon,b}{\longleftarrow}\big(a
\stackrel{\epsilon,b}{\longleftarrow}a
\stackrel{\epsilon,c}{\longleftarrow}b
\stackrel{a,\epsilon}{\longleftarrow}\big)\big(a
\stackrel{\epsilon,b}{\longleftarrow}a
\stackrel{\epsilon,c}{\longleftarrow}b
\stackrel{a,\epsilon}{\longleftarrow}\big)\cdots,
\]
\[
\gamma_b=b
\stackrel{a,\epsilon}{\longleftarrow}a
\stackrel{\epsilon,b}{\longleftarrow}\big(a
\stackrel{\epsilon,b}{\longleftarrow}a
\stackrel{\epsilon,c}{\longleftarrow}b
\stackrel{a,\epsilon}{\longleftarrow}\big)\big(a
\stackrel{\epsilon,b}{\longleftarrow}a
\stackrel{\epsilon,c}{\longleftarrow}b
\stackrel{a,\epsilon}{\longleftarrow}\big)\cdots\text{ and,}
\]
\[
\gamma_c=c
\stackrel{a,\epsilon}{\longleftarrow}b
\stackrel{a,\epsilon}{\longleftarrow}a
\stackrel{\epsilon,b}{\longleftarrow}\big(a
\stackrel{\epsilon,b}{\longleftarrow}a
\stackrel{\epsilon,c}{\longleftarrow}b
\stackrel{a,\epsilon}{\longleftarrow}\big)\big(a
\stackrel{\epsilon,b}{\longleftarrow}a
\stackrel{\epsilon,c}{\longleftarrow}b
\stackrel{a,\epsilon}{\longleftarrow}\big)\cdots
\]
\end{example}

\subsection{Repelling trees}\label{sec:preliminary-T}

The monoid $A^*$ is embedded inside the free group $F_A$ of reduced words in $A^{\pm 1}$. The (Gromov) boundary $\partial F_A$ of $F_A$ consists of infinite reduced words in $A^{\pm 1}$. The free group $F_A$ acts continuously by left multiplication on its boundary. 

A bi-infinite word $Z$ in $A^\Z$ can be identified with the pair $(X,Y) \in \partial^2 F_A = (\partial F_A \times \partial F_A)\setminus \Delta$, where $\Delta$ denotes the diagonal, such that $Z=X\inv\cdot Y$. Then the action of $F_A$ on $\partial F_A$ induces a partial action on $A^\Z$.
In particular, the shift map $S$ is defined on each point $(X,Y) \in \partial^2 F_A$ by 
\[ S(X,Y) = (Y_0^{-1} X, Y_0^{-1} Y). \]

In all this paper we assume that the substitution $\sigma$ is \defi{iwip} (\defi{irreducible with irreducible powers}), that is, it induces an automorphism whose powers fix no proper free factor of $F_A$. An important object to study the dynamics of $\sigma$ is the {\bf attracting tree} $\Tattr$. The tree $\Tattr$ is an $\R$-tree (geodesic and $0$-hyperbolic metric space) with a minimal (there
is no proper $F_A$-invariant subtree), very-small (see \cite{cl-verysmall}) action of $F_A$ by isometries.

For completeness we recall a concrete construction of $\Tattr$ \cite{gjll}.

First, any isometry of a real tree without fixed points acts as a translation along an axis. This defines the translation length of an isometry. For instance, the element $u\in F_A$ acts on the Cayley tree $T_A$ of the free group $F_A$ by a translation of length $\|u\|_A$ where $\|u\|_A$ stands for the length of the cyclically reduced part of $u$.  The length function (of a real tree $T$ with an action of the free group by isometries) maps any element $u \in F_A$ to the translation length $\|u\|_T$ of the action of $u$ on $T$. A minimal action is completely determined by its translation length function \cite{lyndon,chiswell-length-function,chiswell-lambda-tree}. 

Second, any automorphism $\sigma$ of $F_A$ has an expansion  factor $\lambda_\sigma$ which is its maximal exponential growth rate:
\[
\lambda_\sigma=\max_{u\in F_A}\lim_{n\to\infty}\sqrt[n]{\|\sigma^n(u)\|_A}.
\]
Alternatively, and more concretely, the expansion factor of an iwip automorphism is the dominant eigenvalue of the matrix of a train-track representative for $\sigma$. We will not deal with the train-track machinery in this paper and we let the reader learn from Bestvina and Handel work \cite{bh-traintrack}.

Finally, the translation length function for the attracting tree $T_{\sigma}$ is given by
\[
\|u\|_{T_{\sigma}}=\lim_{n\to\infty}\frac{\|\sigma^{n}(u)\|_A}{{\lambda_{\sigma}^n}}.
\]
Note that $\lambda_\sigma$ is the only real number such that this length function $\|\cdot\|_{\Tattr}$ is finite for any $u\in F_A$ and non-zero for at least one $u$.

However, passing from the translation length function to the tree $\Tattr$ is not straightforward. A shortcut for this construction is to consider the Cayley tree $T_A$ of the free group $F_A$ as a metric space by realizing the edges as isometric copies of the real unit segment $[0,1]$ and by extending equivariantly and continuously the automorphism $\sigma$ to edges. We will call this extension the topological realization of $\sigma$. We can then define 
\[
d_\infty(x,y)=\lim_{n\to\infty}\frac{d(\sigma^{n}(x),\sigma^{n}(y))}{\lambda_\sigma^n}
\]
between two points $x,y$ in $T_A$ (all points, not only vertices). The function $d_\infty$ is a pseudo-distance on $T_A$ and the attracting tree is obtained by identifying points at distance $0$. Once again, this construction might be better understood and more concretely handled by starting with a train-track representative rather than the Cayley tree. 

From this last construction we observe that the topological realization of $\sigma$ is a homothety of ratio $\lambda_\sigma$ for the (pseudo-)distance $d_\infty$:
\[
d_\infty(\sigma(x),\sigma(y))=\lambda_\sigma d_\infty(x,y).
\]

Unfortunately for the reader, in this paper we are interested in the {\bf repelling tree} rather than the attracting tree. Indeed we deal with the duality between the attracting shift and the repelling tree. We sum up the above discussion (and replace the attracting tree by the repelling one).

\begin{prop}[\cite{gjll}]\label{thm:gjll}
Every iwip automorphism $\sigma$ has a repelling tree $T_{\sigma\inv}$ which is a real tree with a minimal action of $F_A$ by isometries and a contracting homothety $H$ of ratio $\frac{1}{\lambda_{\sigma\inv}}$ such that for all point $P$ in $\Trepul$ and all element $u$ in $F_A$:
\[
H(uP)=\sigma(u)H(P).
\]
\end{prop}

Indeed the duality between the attracting shift and the repelling tree is emphasized by the existence of the {\bf map $Q$}. Let $\overline{T}_{\sigma\inv}$ denote the metric completion of the tree $T_{\sigma\inv}$.

\begin{prop}
There exists a unique continuous map $Q: X_\sigma\to \overline{T}_{\sigma\inv}$ such that, for any bi-infinite word $Z$ in the attracting shift $X_\sigma$ with first letter $Z_0$,
\[
Q(S(Z))= Z_0^{-1} Q(Z). 
\]
Moreover, using the contracting homothety $H$ on $\Trepul$, for
any bi-infinite word $Z\in X_\sigma$:
\[
Q(\sigma(Z))=H(Q(Z)). 
\]
\end{prop}
The image $Q(X_\sigma)$ is a compact subset of $\overline{T}_{\sigma\inv}$ called the \defi{compact limit set} $\Omega_A$ of $\Trepul$ \cite{chl4}. 

We remark that, as a topological space, the limit set $\Omega_A$ can also be obtained as a quotient of the attracting shift. Indeed, let $\sim$ be the equivalence relation on $X_\sigma$ which is the transitive closure of the ``share a half'' relation (defined in Section \ref{sec:singular-leaves}), then
\[
\Omega_A=X_\sigma/\sim.
\]
The above equality is a consequence of the following statement.

\begin{prop}[{\cite[Corollary~1.3]{chr}},{\cite[Theorem~2]{kl-diagonal-closure}}]\label{prop:Q-injective}
Let $\sigma$ be a substitution which induces an iwip automorphism.
Let $Z,Z'$ be two bi-infinite words in the attracting shift $X_\sigma$.
If $Z$ and $Z'$ share a half then $Q(Z)=Q(Z')$. 
Conversely, if $Q(Z)=Q(Z')$, then there exists a finite sequence $Z=Z_0$, $Z_1,\ldots,Z_{n-1}$, $Z_n=Z'$ of bi-infinite words in the attracting shift $X_\sigma$
such that for each $i$, $Z_i$ and $Z_{i+1}$ share a half.
\end{prop}

We also consider the convex hull $K_A$ of $\Omega_A$, which is called the \defi{compact heart} of $\Trepul$. In this paper we focus on the situation where the limit set is convex: $\Omega_A=K_A$. An iwip automorphism in such a situation is called \defi{parageometric}.

\subsection{Index}

For the subshift $X_\sigma$ we considered the transitive closure of the ``share a half'' equivalence relation $\sim$. For each equivalence class $[Z]$ we define the {\bf index} as the number of possible letters around the origin minus two: 
\[
\ind([Z])=|\{Z'_{-1}\ |\ Z'\in [Z]\}|+|\{Z'_{0}\ |\ Z'\in [Z]\}|-2.
\]
Using Proposition~\ref{prop:finite-index} this index is finite and there are finitely many equivalence classes with positive index. Thus we define
\[
\ind(X_\sigma)=\sum_{[Z]}\ind([Z]).
\]

\begin{prop}[\cite{gjll}]For any substitution $\sigma$ which induces an iwip automorphism the index of the attracting subshift is bounded above by $2N-2$:
\[\ind(X_\sigma)\leq 2N-2.
\]
Moreover, this index is maximal if and only if $\sigma$ is parageometric.
\end{prop}

\begin{example} For Tribonacci substitution there are two equivalence classes with positive index, both contain three bi-infinite words: $\{w_a,w_b,w_c\}$ (see Example~\ref{ex:tribo-fix-points}) and $\{w'_a,w'_b,w'_c\}$ (see Example~\ref{ex:tribo-right-left-special}), thus the index is $2+2=4=2\times 3-2$: the Tribonacci automorphism is parageometric.
\end{example}

\subsection{Cylinders in the repelling tree}\label{sec:cylind-repell-tree}

For a finite word $w\in A^*$ we push forward the cylinder $[w]$ to get a compact subset $\Omega_w=Q([w])$ of the compact limit set. If $\gamma$ is a finite path in the prefix-suffix automaton we similarly push forward $[\gamma]$ to get $\Omega_\gamma=Q([\gamma])$.

Pushing forward Proposition~\ref{prop:cylinders-autosimilar}, we get
\begin{prop}\label{prop:autosimilar-tree}
Let $\sigma$ be a substitution which induces an iwip automorphism.
For any finite path $\gamma\gamma'$ in the prefix-suffix automaton we have
\[\Omega_{\gamma\gamma'}=p_0^{-1}\sigma(p_1^{-1})\cdots\sigma^{n-1}(p_{n-1}^{-1})H^n(\Omega_{\gamma'})
=p(\gamma)\inv H^n(\Omega_{\gamma'})\]
where
\[
\gamma
=a_0\stackrel{p_0,s_0}{\longleftarrow}a_1\stackrel{p_1,s_1}{\longleftarrow}a_2
\cdots\stackrel{p_{n-1},s_{n-1}}{\longleftarrow}a_n.
\]
Moreover, for any path $\gamma$ in the prefix-suffix automaton and any $n\in\N^*$
\[
\Omega_\gamma=\bigcup_{\gamma', |\gamma'|=n}\Omega_{\gamma\gamma'}\quad \text{in
  particular for } \gamma=\epsilon,\quad \Omega_A=\bigcup_{\gamma',
  |\gamma'|=n}\Omega_{\gamma'}
\] 
where the union is taken over all paths $\gamma'$ in the prefix-suffix automaton extending $\gamma$ (in case $\gamma$ is the empty path, ending at any state) of length $n$. Furthermore, if $\sigma$ is parageometric then each $\Omega_\gamma$ is connected.
\end{prop}
\begin{proof}
The result follows easily applying the map $Q$ to both sides of the statements of Proposition~\ref{prop:cylinders-autosimilar} recalling that $Q(S(Z))=Z_0^ {-1}Q(Z)$ and $Q(\sigma(Z)) = H(Q(Z))$. For a proof of the connectedness of $\Omega_\gamma$ we refer to \cite{chl4}.
\end{proof}

The formula in the previous Proposition can be used to define (or at least better understand) the map $Q$. For a bi-infinite word $Z\in X_\sigma$, with prefix-suffix development $\gamma=a_0\stackrel{p_0,s_0}{\longleftarrow}a_1\stackrel{p_1,s_1}{\longleftarrow}\cdots$, for any point $P\in \Trepul$:
\[
Q(Z)=\lim_{n\to\infty}{p_0}\inv H({p_1}\inv H(\cdots (H({p_n}\inv P))\cdots )),
\]
indeed, as $H$ is a contracting homothety, the terms inside the limit form a Cauchy sequence.
Abusing of notations, we use the above formula to consider the  map $Q:\mathcal{P}\to\Trepul$: for an infinite desubstitution path $\gamma$ and a bi-infinite word $Z$ such that $\Gamma(Z)=\gamma$, we have 
\[
Q(\gamma)=Q(\Gamma(Z))=Q(Z).
\]
For an infinite desubstitution path $\gamma=\gamma'\gamma''$ (with $\gamma'$ a finite desubstitution path),
\[
Q(\gamma'\gamma'')=p(\gamma')\inv H^{|\gamma'|}(Q(\gamma'')).
\]
In particular for words with eventually periodic prefix-suffix expansions we get the following statement.

\begin{prop}\label{prop:Pexpansion}
Let $Z\in X_\sigma$ with eventually periodic prefix-suffix expansion $\gamma=\alpha\beta^\infty$, where
\[
\alpha=a_0\stackrel{p_0,s_0}{\longleftarrow}a_1\stackrel{p_1,s_1}{\longleftarrow}
\cdots\stackrel{p_{m-1},s_{m-1}}{\longleftarrow}a_m \quad \text{and},
\]
\[
\beta=a_m\stackrel{p_{m},s_{m}}{\longleftarrow}a_{m+1}\stackrel{p_{m+1},s_{m+1}}{\longleftarrow}
\cdots a_{m+n-1}\stackrel{p_{m+n-1},s_{m+n-1}}{\longleftarrow}a_{m}.
\] 
Then we have
\[
P=Q(Z)=Q(\gamma)=p(\alpha)\inv\sigma^m(p(\beta)\inv)\sigma^n(p(\alpha))H^n(P).
\]
\end{prop}

\begin{proof}
By Proposition~\ref{prop:autosimilar-tree} we have 
\begin{equation}\label{eq:P}
P = p(\alpha)^{-1}H^{|\alpha|}(P_\beta),
\end{equation}
where $P_\beta = Q(Z')$ and $Z'\in X_\sigma$ has infinite desubstitution path $\beta^\infty$. Now applying Theorem~\ref{thm:gjll} we obtain
\[ P = p(\alpha)^{-1}\sigma^{|\alpha|}(p(\beta^{-1})) H^{|\alpha|+
|\beta|}(P_\beta) \] 
and by \eqref{eq:P} we finally get the result.
\end{proof}

\subsection{Rauzy fractals}\label{sec:preliminary-rauzy}

Let $\ell:A^*\to\Z^{A}$, $w\mapsto (|w|_a)_{a\in A}$ be the abelianization map and $M_\sigma$ be the abelianization matrix of $\sigma$ acting on $\Z^{A}$ or $\R^{A}$. Then $\ell(\sigma(w))=M_\sigma \ell(w)$ holds for any $w\in A^*$.

We assume that the substitution $\sigma$ is \defi{irreducible Pisot}, that is, the characteristic polynomial of $M_\sigma$ is the minimal polynomial of a Pisot number $\lambda_\sigma$, i.e.~a real number strictly bigger than $1$ whose conjugates are strictly smaller than one in modulus. Then we have the $M_\sigma$-invariant decomposition
\[
\R^{A}=\R\,\vec u \oplus E_c
\]
where $\vec u$ is the eigenvector associated to $\lambda_\sigma$ and, $E_c\cong \R^{|A|-1}$ is the contracting hyperplane spanned by the eigenvectors associated with the Galois conjugates of $\lambda_\sigma$ (except itself). Let $\pi_c$ be the projection of $\R^{A}$ to $E_c$ along $\vec u$.

For $Z\in X_\sigma$ with prefix-suffix decomposition
$\gamma=a_0\stackrel{p_0,s_0}{\longleftarrow}a_1\stackrel{p_1,s_1}{\longleftarrow}a_2
\cdots$ define
\begin{equation}
\varphi:X_\sigma\to E_c,\quad
\varphi(Z)= \sum_{i=0}^\infty \pi_c(M^i_\sigma\, \ell(p_i)),
\end{equation}
where the infinite sum converges since $M_\sigma$ is a contraction in $E_c$. Remark that the map $\varphi$ is continuous and only depends on the prefix-suffix expansion of $Z$ and thus, abusing again of notations, we define $\varphi(\gamma)=\varphi(\Gamma(Z))=\varphi(Z)$ and we regard also the map $\varphi:\mathcal P\to E_c$. 

The Pisot property allows to represent the attracting shift $X_\sigma$
geometrically as a compact domain with fractal boundary, called Rauzy
fractal in honor of G.~Rauzy who first defined it for the Tribonacci
substitution \cite{rauzy}.

\begin{defn}
The {\bf Rauzy fractal} $\mathcal{R}_\sigma$ is the compact set
$\varphi(X_\sigma)\subset E_c$.  Since $X_\sigma$ is the union of the
cylinders $[a]$, for $a\in A$, the Rauzy fractal is decomposed into
subpieces $\mathcal{R}_\sigma(a) = \varphi([a])$.
\end{defn}

We state now some properties of these fractals. For more details we refer to \cite{cant-chapter7}.

\begin{prop}\label{prop:rauzyprop}
Let $\sigma$ be a Pisot substitution. Then the following properties hold:
\begin{itemize}
\item $\mathcal{R}_\sigma$ is the closure of its interior.
\item $\partial\mathcal{R}_\sigma(a)$ has measure zero, for each $a\in A$.
\item The Rauzy fractal obeys to the set equation
\[  \mathcal{R}_\sigma(a) = \bigcup_{a\stackrel{p,s}{\leftarrow}b} M_\sigma\mathcal{R}_\sigma(b) + \pi_c(\ell(p)), \quad \text{for }a\in A. \]
Furthermore the union is measure disjoint.

\end{itemize}
\end{prop}

The substitution $\sigma$ satisfies the \defi{strong coincidence condition} if $\forall\,(a,b)\in A^2$ there exist $n\in\mathbb{N}$ and $i\in A$ such that $\sigma^n(a)=p_1 i s_1$ and $\sigma^n(b)=p_2 i s_2$ and the prefixes $p_1$ and $p_2$ share the same abelianization.
In this case the subtiles $\mathcal{R}_\sigma(a)$ are pairwise disjoint in measure and we can define the \defi{domain exchange}
\[ E: \mathcal{R}_\sigma \to \mathcal{R}_\sigma,\quad z \mapsto z+\pi_c(\ell(a)),\quad\text{if }z\in\mathcal{R}_{\sigma}(a).  \]

From the definition of the map $\varphi$ we get the following properties \cite{cs}:
\[
\varphi(S(Z))=\varphi(Z)+\pi_c(\ell(Z_0))\quad \text{and}\quad \varphi(\sigma(Z))=M_\sigma\,\varphi(Z).
\]
Thus the shift and the action of $\sigma$ correspond respectively to the domain exchange and to the contraction $M_{\sigma}$ on the Rauzy fractal.

We emphasize that if we have an ultimately periodic prefix-suffix decomposition
$\gamma=\alpha\beta^\infty$ (see notations at the end of Section~\ref{sec:preliminary-T}) then we get the following formula:
\[
\varphi(Z)=\pi_c(\ell(p(\alpha)))+\pi_c(M_\sigma^m(M_\sigma^n-Id)\inv \ell(p(\beta))).
\]

\begin{figure}[h]
\centering
\includegraphics[scale=.35]{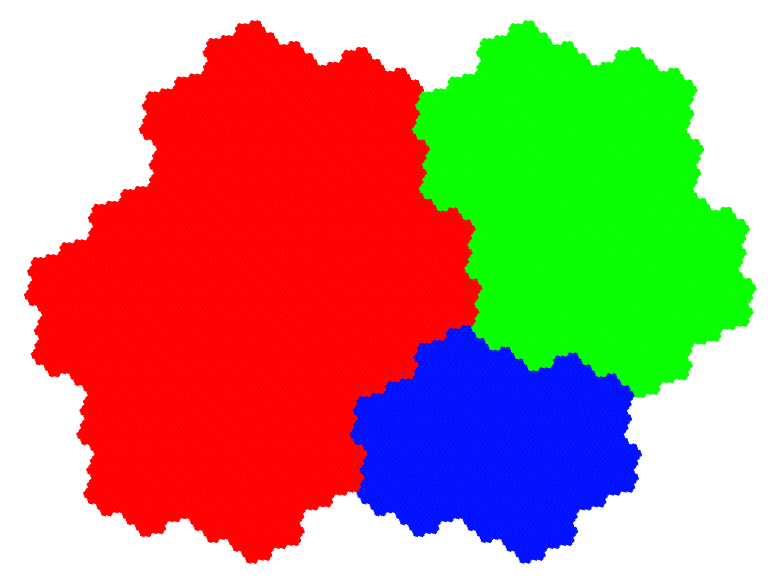}
\caption{The Rauzy fractal of the Tribonacci substitution.}
\end{figure}

Rauzy fractals are an important geometrical tool to understand the
dynamics of the attracting shift $X_\sigma$. It is conjectured that
every $(X_\sigma,S)$ generated by an irreducible Pisot substitution has
pure discrete spectrum, or equivalently is metrically conjugate to a
translation on a compact Abelian group. This famous problem is known
as {\bf Pisot conjecture} (see \cite{ABBLS15} for a recent
survey). One way to attack this conjecture consists in showing that
the Rauzy fractal tiles periodically the contracting space
$E_c$. Indeed, if this is the case, then the Rauzy fractal is a
fundamental domain of the torus $\mathbb{T}^{A-1}$ and the domain
exchange turns into a toral translation. Words in the attracting shift
$X_\sigma$ are then codings of this toral translation with respect to
the partition $\{\mathcal{R}(i): i\in A\}$.

In this work we do not tackle this important problem but we believe
that the tree substitutions approach will give a new particular
insight to it.

\subsection{The global picture}\label{sec:global-picture}

The map $\varphi$ of the previous section factors through the map $Q$ of Section~\ref{sec:preliminary-T}, and we can state

\begin{prop}\label{prop:global-picture}
For an irreducible Pisot substitution $\sigma$ which is an iwip
automorphism of the free group $F_A$, let $X_\sigma$ be the attracting
shift of $\sigma$, $\Trepul$ the repelling tree and,
$\mathcal{R}_\sigma$ the Rauzy fractal, as defined in the previous
sections. Then there exists a continuous map $\psi$ such that the
following diagram commutes
\[
\xymatrix{
   A^\Z\supseteq X_\sigma \ar[r]^\Gamma \ar []+DR;[dr]+UL^Q \ar@/_/ []+DR-<5pt,0pt>;[ddr]+UL_\varphi
&\mathcal{P} \ar[d]+UL+<10pt,0pt>_Q \ar@/^4pc/[dd]+UL+<15pt,0pt>^\varphi\\
&\Omega_A\subseteq K_A\subset \overline{T}_{\sigma^{-1}} \ar []+DL+<10pt,0pt>;[d]+UL+<8pt,0pt>^\psi \\
& \mathcal{R}_\sigma\subset E_c &
}
\]
Moreover the action of the shift on $X_\sigma$ is pushed forward to
the action of a system of partial isometries of $\Omega_A$ and to the
domain exchange on $\mathcal{R}_\sigma$, while the action of $\sigma$
on $X_\sigma$ is pushed forward to the contracting homothety $H$ of
the attracting tree and the matrix $M_\sigma$ on the contracting
hyperplane: $\forall Z\in X_\sigma$,
\begin{align*}
Q(SZ)&=Z_0^{-1}Q(Z), & \varphi(SZ)&=\varphi(Z)+\pi_c(Z_0), \\
Q(\sigma(Z))&=H(Q(Z)), & \varphi(\sigma(Z))&=M_\sigma\varphi(Z).
\end{align*}
\end{prop}

The aim of this paper is to study $\psi$ and more specifically to draw approximations of $\psi(\Omega_A)$ inside the Rauzy fractal.

\subsection{Renormalization}\label{sec:renormalization}

Each in its own fashion, the attracting shift, the repelling tree and
the Rauzy fractal are self-similar. This is expressed in
Propositions~\ref{prop:cylinders-autosimilar},
\ref{prop:autosimilar-tree} and \ref{prop:rauzyprop}. In the repelling
tree and in the Rauzy fractal, the self-similarity is here expressed
after renormalization (i.e.~after applying the contracting homothety
$H$ or the matrix $M_\sigma$ on the contracting space). In order to
draw the Rauzy fractal and the repelling tree using geometric
substitutions (tree substitutions or dual substitutions which we will
define in the next sections), it is convenient not to
renormalize. Thus we introduce the following notations.  For a finite
path $\gamma$ of length $n$ in the prefix-suffix automaton we let
\begin{itemize}
\item $\widetilde{\Omega}_\gamma = H^{-n}(\Omega_\gamma)$ in the repelling tree $\Trepul$; 
\item $\widetilde{\mathcal R}_\gamma=M_\sigma^{-n}(\mathcal R_\gamma)$ in the contracting hyperplane.
\end{itemize}

In this setting, tiles $\widetilde{\mathcal R}_\gamma$ of the Rauzy fractal are translates one from another as soon as the paths $\gamma$ begins with the same letter, and similarly for the tiles $\widetilde{\Omega}_\gamma$ of the repelling tree:

\begin{prop}\label{prop:translated-tiles}
Let $\gamma=a_0\stackrel{p_0,s_0}{\longleftarrow}
\cdots\stackrel{p_{n-1},s_{n-1}}{\longleftarrow}a_n$ be a finite path in the prefix-suffix automaton and recall our notation $p(\gamma)=\sigma^{n-1}(p_{n-1})\cdots\sigma(p_1)p_0$. Then 
\[
\widetilde{\Omega}_\gamma = \sigma^{-n}(p(\gamma)\inv) \Omega_{a_n} \quad (\text{and } \Omega_{a_n}=\widetilde{\Omega}_{a_n} ) \]
\[
\widetilde{\mathcal{R}}_\gamma=\mathcal{R}_{a_n}+M_\sigma^{-n}\pi_c(\ell(p(\gamma))) \quad (\text{and } \mathcal{R}_{a_n}=\widetilde{\mathcal{R}}_{a_n}=\mathcal{R}(a_n) ).
\]
\end{prop}

\subsection{Dual substitution}\label{sec:dual-substitution}

For each $(x,a)\in\mathbb{Z}^{A}\times A$ we let $(x,a)^*$ be the face of the hypercube based at $x$ orthogonal to $\ell(a)$; precisely
\[ (x,a)^* = \Big\{ x + \sum_{i\neq a} t_i e_i : t_i\in [0,1] \Big\}.  \]
The \defi{dual substitution} $E^*_1(\sigma)$ is defined by
\[
E^*_1(\sigma)(x,a)^*=\bigcup_{a\stackrel{p,s}{\longleftarrow}b}(M^{-1}_\sigma(x+\ell(p)),b)^*
\]
where the union is taken over all edges $a\stackrel{p,s}{\longleftarrow}b$ ending at $a$ in the prefix-suffix automaton.

Iterating the dual substitution on the hyperface $(0,a)^*$ and renormalizing through the contracting action of $M_\sigma$ we converge in the Hausdorff metric towards the $a$-th subtile of the Rauzy fractal:
\[\lim_{n\to\infty}M_\sigma^n\, \pi_c\big(E^*_1(\sigma)^n(0,a)^*\big)=\mathcal R_\sigma(a) .\]

\begin{example}
The dual Tribonacci substitution is defined by
\begin{align*}
E^*_1(\sigma): (x,1)^* &\mapsto (M_\sigma^{-1}x,1)^* \cup (M_\sigma^{-1}x,2)^* \cup (M_\sigma^{-1}x,3)^* \\
(x,2)^* &\mapsto (M_\sigma^{-1}(x+e_1),1)^* \\ 
(x,3)^* &\mapsto (M_\sigma^{-1}(x+e_1),2)^*
\end{align*}

\begin{figure}[h]
\centering
\begin{tikzpicture}
[x={(-0.433013cm,-0.250000cm)}, y={(0.433013cm,-0.250000cm)},
z={(0.000000cm,0.500000cm)}]
\definecolor{facecolor}{rgb}{0.000,1.000,0.000}
\fill[fill=facecolor, draw=black, shift={(0,0,0)}]
(0, 0, 0) -- (0, 0, 1) -- (1, 0, 1) -- (1, 0, 0) -- cycle;
\definecolor{facecolor}{rgb}{1.000,0.000,0.000}
\fill[fill=facecolor, draw=black, shift={(0,0,0)}]
(0, 0, 0) -- (0, 1, 0) -- (0, 1, 1) -- (0, 0, 1) -- cycle;
\definecolor{facecolor}{rgb}{0.000,0.000,1.000}
\fill[fill=facecolor, draw=black, shift={(0,0,0)}]
(0, 0, 0) -- (1, 0, 0) -- (1, 1, 0) -- (0, 1, 0) -- cycle;
\node[circle,fill=black,draw=black,minimum size=1.5mm,inner sep=0pt] at
(0,0,0) {};
\end{tikzpicture}
\quad
\begin{tikzpicture}
[x={(-0.433013cm,-0.250000cm)}, y={(0.433013cm,-0.250000cm)},
z={(0.000000cm,0.500000cm)}]
\definecolor{facecolor}{rgb}{0.000,0.000,1.000}
\fill[fill=facecolor, draw=black, shift={(0,0,1)}]
(0, 0, 0) -- (0, 0, 1) -- (1, 0, 1) -- (1, 0, 0) -- cycle;
\definecolor{facecolor}{rgb}{1.000,0.000,0.000}
\fill[fill=facecolor, draw=black, shift={(0,0,0)}]
(0, 0, 0) -- (0, 0, 1) -- (1, 0, 1) -- (1, 0, 0) -- cycle;
\fill[fill=facecolor, draw=black, shift={(0,0,0)}]
(0, 0, 0) -- (0, 1, 0) -- (0, 1, 1) -- (0, 0, 1) -- cycle;
\definecolor{facecolor}{rgb}{0.000,1.000,0.000}
\fill[fill=facecolor, draw=black, shift={(0,0,1)}]
(0, 0, 0) -- (0, 1, 0) -- (0, 1, 1) -- (0, 0, 1) -- cycle;
\definecolor{facecolor}{rgb}{1.000,0.000,0.000}
\fill[fill=facecolor, draw=black, shift={(0,0,0)}]
(0, 0, 0) -- (1, 0, 0) -- (1, 1, 0) -- (0, 1, 0) -- cycle;
\node[circle,fill=black,draw=black,minimum size=1.5mm,inner sep=0pt] at
(0,0,0) {};
\end{tikzpicture}
\quad
\begin{tikzpicture}
[x={(-0.433013cm,-0.250000cm)}, y={(0.433013cm,-0.250000cm)},
z={(0.000000cm,0.500000cm)}]
\definecolor{facecolor}{rgb}{1.000,0.000,0.000}
\fill[fill=facecolor, draw=black, shift={(0,0,1)}]
(0, 0, 0) -- (0, 0, 1) -- (1, 0, 1) -- (1, 0, 0) -- cycle;
\fill[fill=facecolor, draw=black, shift={(0,0,1)}]
(0, 0, 0) -- (0, 1, 0) -- (0, 1, 1) -- (0, 0, 1) -- cycle;
\definecolor{facecolor}{rgb}{0.000,0.000,1.000}
\fill[fill=facecolor, draw=black, shift={(0,1,0)}]
(0, 0, 0) -- (0, 1, 0) -- (0, 1, 1) -- (0, 0, 1) -- cycle;
\definecolor{facecolor}{rgb}{0.000,1.000,0.000}
\fill[fill=facecolor, draw=black, shift={(0,1,-1)}]
(0, 0, 0) -- (1, 0, 0) -- (1, 1, 0) -- (0, 1, 0) -- cycle;
\fill[fill=facecolor, draw=black, shift={(0,1,-1)}]
(0, 0, 0) -- (0, 0, 1) -- (1, 0, 1) -- (1, 0, 0) -- cycle;
\definecolor{facecolor}{rgb}{1.000,0.000,0.000}
\fill[fill=facecolor, draw=black, shift={(0,0,0)}]
(0, 0, 0) -- (0, 0, 1) -- (1, 0, 1) -- (1, 0, 0) -- cycle;
\definecolor{facecolor}{rgb}{0.000,1.000,0.000}
\fill[fill=facecolor, draw=black, shift={(0,1,-1)}]
(0, 0, 0) -- (0, 1, 0) -- (0, 1, 1) -- (0, 0, 1) -- cycle;
\definecolor{facecolor}{rgb}{1.000,0.000,0.000}
\fill[fill=facecolor, draw=black, shift={(0,0,0)}]
(0, 0, 0) -- (1, 0, 0) -- (1, 1, 0) -- (0, 1, 0) -- cycle;
\fill[fill=facecolor, draw=black, shift={(0,0,0)}]
(0, 0, 0) -- (0, 1, 0) -- (0, 1, 1) -- (0, 0, 1) -- cycle;
\node[circle,fill=black,draw=black,minimum size=1.5mm,inner sep=0pt] at
(0,0,0) {};
\end{tikzpicture}
\quad
\begin{tikzpicture}
[x={(-0.433013cm,-0.250000cm)}, y={(0.433013cm,-0.250000cm)},
z={(0.000000cm,0.500000cm)}]
\definecolor{facecolor}{rgb}{1.000,0.000,0.000}
\fill[fill=facecolor, draw=black, shift={(0,0,1)}]
(0, 0, 0) -- (0, 0, 1) -- (1, 0, 1) -- (1, 0, 0) -- cycle;
\fill[fill=facecolor, draw=black, shift={(0,0,0)}]
(0, 0, 0) -- (0, 1, 0) -- (0, 1, 1) -- (0, 0, 1) -- cycle;
\fill[fill=facecolor, draw=black, shift={(0,0,1)}]
(0, 0, 0) -- (0, 1, 0) -- (0, 1, 1) -- (0, 0, 1) -- cycle;
\fill[fill=facecolor, draw=black, shift={(0,1,0)}]
(0, 0, 0) -- (0, 1, 0) -- (0, 1, 1) -- (0, 0, 1) -- cycle;
\definecolor{facecolor}{rgb}{0.000,0.000,1.000}
\fill[fill=facecolor, draw=black, shift={(1,0,-1)}]
(0, 0, 0) -- (1, 0, 0) -- (1, 1, 0) -- (0, 1, 0) -- cycle;
\definecolor{facecolor}{rgb}{0.000,1.000,0.000}
\fill[fill=facecolor, draw=black, shift={(1,-1,1)}]
(0, 0, 0) -- (0, 1, 0) -- (0, 1, 1) -- (0, 0, 1) -- cycle;
\fill[fill=facecolor, draw=black, shift={(1,-1,1)}]
(0, 0, 0) -- (0, 0, 1) -- (1, 0, 1) -- (1, 0, 0) -- cycle;
\definecolor{facecolor}{rgb}{1.000,0.000,0.000}
\fill[fill=facecolor, draw=black, shift={(0,0,0)}]
(0, 0, 0) -- (0, 0, 1) -- (1, 0, 1) -- (1, 0, 0) -- cycle;
\fill[fill=facecolor, draw=black, shift={(0,1,-1)}]
(0, 0, 0) -- (1, 0, 0) -- (1, 1, 0) -- (0, 1, 0) -- cycle;
\fill[fill=facecolor, draw=black, shift={(0,1,-1)}]
(0, 0, 0) -- (0, 0, 1) -- (1, 0, 1) -- (1, 0, 0) -- cycle;
\definecolor{facecolor}{rgb}{0.000,1.000,0.000}
\fill[fill=facecolor, draw=black, shift={(1,-1,0)}]
(0, 0, 0) -- (0, 1, 0) -- (0, 1, 1) -- (0, 0, 1) -- cycle;
\definecolor{facecolor}{rgb}{1.000,0.000,0.000}
\fill[fill=facecolor, draw=black, shift={(0,1,-1)}]
(0, 0, 0) -- (0, 1, 0) -- (0, 1, 1) -- (0, 0, 1) -- cycle;
\fill[fill=facecolor, draw=black, shift={(0,0,0)}]
(0, 0, 0) -- (1, 0, 0) -- (1, 1, 0) -- (0, 1, 0) -- cycle;
\definecolor{facecolor}{rgb}{0.000,1.000,0.000}
\fill[fill=facecolor, draw=black, shift={(1,-1,0)}]
(0, 0, 0) -- (1, 0, 0) -- (1, 1, 0) -- (0, 1, 0) -- cycle;
\fill[fill=facecolor, draw=black, shift={(1,-1,0)}]
(0, 0, 0) -- (0, 0, 1) -- (1, 0, 1) -- (1, 0, 0) -- cycle;
\definecolor{facecolor}{rgb}{0.000,0.000,1.000}
\fill[fill=facecolor, draw=black, shift={(1,0,-1)}]
(0, 0, 0) -- (0, 0, 1) -- (1, 0, 1) -- (1, 0, 0) -- cycle;
\fill[fill=facecolor, draw=black, shift={(1,0,-1)}]
(0, 0, 0) -- (0, 1, 0) -- (0, 1, 1) -- (0, 0, 1) -- cycle;
\node[circle,fill=black,draw=black,minimum size=1.5mm,inner sep=0pt] at
(0,0,0) {};
\end{tikzpicture}
\quad
\begin{tikzpicture}
[x={(-0.433013cm,-0.250000cm)}, y={(0.433013cm,-0.250000cm)},
z={(0.000000cm,0.500000cm)}]
\definecolor{facecolor}{rgb}{0.000,0.000,1.000}
\fill[fill=facecolor, draw=black, shift={(0,-1,2)}]
(0, 0, 0) -- (1, 0, 0) -- (1, 1, 0) -- (0, 1, 0) -- cycle;
\definecolor{facecolor}{rgb}{0.000,1.000,0.000}
\fill[fill=facecolor, draw=black, shift={(-1,0,2)}]
(0, 0, 0) -- (0, 1, 0) -- (0, 1, 1) -- (0, 0, 1) -- cycle;
\definecolor{facecolor}{rgb}{1.000,0.000,0.000}
\fill[fill=facecolor, draw=black, shift={(1,-1,1)}]
(0, 0, 0) -- (0, 1, 0) -- (0, 1, 1) -- (0, 0, 1) -- cycle;
\fill[fill=facecolor, draw=black, shift={(0,1,-1)}]
(0, 0, 0) -- (0, 0, 1) -- (1, 0, 1) -- (1, 0, 0) -- cycle;
\fill[fill=facecolor, draw=black, shift={(1,-1,0)}]
(0, 0, 0) -- (0, 1, 0) -- (0, 1, 1) -- (0, 0, 1) -- cycle;
\definecolor{facecolor}{rgb}{0.000,1.000,0.000}
\fill[fill=facecolor, draw=black, shift={(-1,0,3)}]
(0, 0, 0) -- (0, 1, 0) -- (0, 1, 1) -- (0, 0, 1) -- cycle;
\definecolor{facecolor}{rgb}{0.000,0.000,1.000}
\fill[fill=facecolor, draw=black, shift={(0,-1,2)}]
(0, 0, 0) -- (0, 0, 1) -- (1, 0, 1) -- (1, 0, 0) -- cycle;
\definecolor{facecolor}{rgb}{0.000,1.000,0.000}
\fill[fill=facecolor, draw=black, shift={(-1,1,1)}]
(0, 0, 0) -- (0, 1, 0) -- (0, 1, 1) -- (0, 0, 1) -- cycle;
\definecolor{facecolor}{rgb}{1.000,0.000,0.000}
\fill[fill=facecolor, draw=black, shift={(0,1,-1)}]
(0, 0, 0) -- (1, 0, 0) -- (1, 1, 0) -- (0, 1, 0) -- cycle;
\definecolor{facecolor}{rgb}{0.000,1.000,0.000}
\fill[fill=facecolor, draw=black, shift={(-1,0,2)}]
(0, 0, 0) -- (0, 0, 1) -- (1, 0, 1) -- (1, 0, 0) -- cycle;
\definecolor{facecolor}{rgb}{1.000,0.000,0.000}
\fill[fill=facecolor, draw=black, shift={(1,-1,0)}]
(0, 0, 0) -- (0, 0, 1) -- (1, 0, 1) -- (1, 0, 0) -- cycle;
\definecolor{facecolor}{rgb}{0.000,1.000,0.000}
\fill[fill=facecolor, draw=black, shift={(-1,0,3)}]
(0, 0, 0) -- (0, 0, 1) -- (1, 0, 1) -- (1, 0, 0) -- cycle;
\fill[fill=facecolor, draw=black, shift={(-1,0,2)}]
(0, 0, 0) -- (1, 0, 0) -- (1, 1, 0) -- (0, 1, 0) -- cycle;
\definecolor{facecolor}{rgb}{1.000,0.000,0.000}
\fill[fill=facecolor, draw=black, shift={(0,0,1)}]
(0, 0, 0) -- (0, 1, 0) -- (0, 1, 1) -- (0, 0, 1) -- cycle;
\fill[fill=facecolor, draw=black, shift={(0,1,0)}]
(0, 0, 0) -- (0, 1, 0) -- (0, 1, 1) -- (0, 0, 1) -- cycle;
\definecolor{facecolor}{rgb}{0.000,1.000,0.000}
\fill[fill=facecolor, draw=black, shift={(-1,1,1)}]
(0, 0, 0) -- (1, 0, 0) -- (1, 1, 0) -- (0, 1, 0) -- cycle;
\definecolor{facecolor}{rgb}{1.000,0.000,0.000}
\fill[fill=facecolor, draw=black, shift={(1,0,-1)}]
(0, 0, 0) -- (1, 0, 0) -- (1, 1, 0) -- (0, 1, 0) -- cycle;
\definecolor{facecolor}{rgb}{0.000,0.000,1.000}
\fill[fill=facecolor, draw=black, shift={(0,-1,2)}]
(0, 0, 0) -- (0, 1, 0) -- (0, 1, 1) -- (0, 0, 1) -- cycle;
\fill[fill=facecolor, draw=black, shift={(0,-1,3)}]
(0, 0, 0) -- (0, 1, 0) -- (0, 1, 1) -- (0, 0, 1) -- cycle;
\definecolor{facecolor}{rgb}{1.000,0.000,0.000}
\fill[fill=facecolor, draw=black, shift={(1,-1,0)}]
(0, 0, 0) -- (1, 0, 0) -- (1, 1, 0) -- (0, 1, 0) -- cycle;
\definecolor{facecolor}{rgb}{0.000,1.000,0.000}
\fill[fill=facecolor, draw=black, shift={(-1,1,2)}]
(0, 0, 0) -- (0, 1, 0) -- (0, 1, 1) -- (0, 0, 1) -- cycle;
\definecolor{facecolor}{rgb}{1.000,0.000,0.000}
\fill[fill=facecolor, draw=black, shift={(0,0,1)}]
(0, 0, 0) -- (0, 0, 1) -- (1, 0, 1) -- (1, 0, 0) -- cycle;
\fill[fill=facecolor, draw=black, shift={(0,0,0)}]
(0, 0, 0) -- (0, 1, 0) -- (0, 1, 1) -- (0, 0, 1) -- cycle;
\fill[fill=facecolor, draw=black, shift={(1,-1,1)}]
(0, 0, 0) -- (0, 0, 1) -- (1, 0, 1) -- (1, 0, 0) -- cycle;
\fill[fill=facecolor, draw=black, shift={(0,0,0)}]
(0, 0, 0) -- (0, 0, 1) -- (1, 0, 1) -- (1, 0, 0) -- cycle;
\fill[fill=facecolor, draw=black, shift={(1,0,-1)}]
(0, 0, 0) -- (0, 0, 1) -- (1, 0, 1) -- (1, 0, 0) -- cycle;
\definecolor{facecolor}{rgb}{0.000,0.000,1.000}
\fill[fill=facecolor, draw=black, shift={(0,-1,3)}]
(0, 0, 0) -- (0, 0, 1) -- (1, 0, 1) -- (1, 0, 0) -- cycle;
\definecolor{facecolor}{rgb}{1.000,0.000,0.000}
\fill[fill=facecolor, draw=black, shift={(0,1,-1)}]
(0, 0, 0) -- (0, 1, 0) -- (0, 1, 1) -- (0, 0, 1) -- cycle;
\fill[fill=facecolor, draw=black, shift={(0,0,0)}]
(0, 0, 0) -- (1, 0, 0) -- (1, 1, 0) -- (0, 1, 0) -- cycle;
\fill[fill=facecolor, draw=black, shift={(1,0,-1)}]
(0, 0, 0) -- (0, 1, 0) -- (0, 1, 1) -- (0, 0, 1) -- cycle;
\definecolor{facecolor}{rgb}{0.000,1.000,0.000}
\fill[fill=facecolor, draw=black, shift={(-1,1,1)}]
(0, 0, 0) -- (0, 0, 1) -- (1, 0, 1) -- (1, 0, 0) -- cycle;
\node[circle,fill=black,draw=black,minimum size=1.5mm,inner sep=0pt] at
(0,0,0) {};
\end{tikzpicture}
\caption{The first steps of the dual substitution for Tribonacci on $\{(0,1)^*,(0,2)^*,(0,3)^*\}$.}
\end{figure}
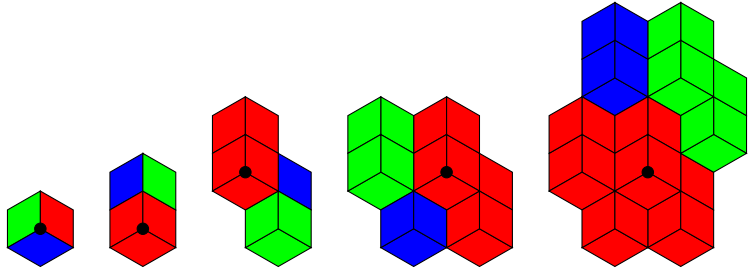
\end{example}

From the definition, the $n$-th iterate of the dual substitution yields the patch
\[
E_1^*(\sigma)^n\left(\bigcup_{a\in A}(0,a)^*\right)=\bigcup_{\gamma} M_\sigma^{-n}\ell(p(\gamma))+(0,a_n)^*,
\]
where $\gamma$ ranges over all paths of length $n$ in the
prefix-suffix automaton and, $a_n$ is the starting letter of $\gamma$.

We focus on the contracting space $E_c$ and we allow more freedom by
picking arbitrary prototiles $P_a\subseteq E_c$. The $n$-th iterate of
the dual substitution yields the patch
\[
P_n=\bigcup_{\gamma} M_\sigma^{-n}\pi_c(\ell(p(\gamma))) + P_{a_n},
\]
which decomposes as the union of the patches $P_n(a)$ obtained by
restricting the union to those paths $\gamma$ of length $n$ ending at
the letter $a$. For each such path $\gamma$ in the prefix-suffix automaton, starting at the letter $a_n$ we denote by $P_\gamma$ the tile
\[
P_\gamma=M_\sigma^{-n}\pi_c(\ell(p(\gamma))) + P_{a_n}.
\]

If we start with the prototiles $P_a=\pi_c((0,a)^*)$ we get the usual
picture of the iterates of the dual substitution with the property
that, for each $n$, $\bigcup_{a\in A}P_n(a)$ tiles periodically the contracting space $E_c$ under the action of the lattice 
\[
\Lambda=\sum_{i,j\in A, i\neq j} M_\sigma^{-n}\pi_c(e_j-e_i)\Z.
\]

If we start with the prototiles $P_a=\mathcal R(a)$ given by the Rauzy
fractal, we get the usual self-similar decomposition of the Rauzy
fractal. If $\sigma$ is an iwip parageometric automorphism, using the
map $\varphi$ of Proposition~\ref{prop:global-picture}, for two
distinct paths $\gamma$ and $\gamma'$ of length $n$ in the
prefix-suffix automaton, we get that if in the repelling tree $\Omega_\gamma\cap\Omega_{\gamma'}\neq\emptyset$ then $P_\gamma\cap P_{\gamma'}\neq\emptyset$.

\begin{quest}\label{quest:properties-dual-substitution}
Let $\sigma$ be an irreducible Pisot substitution and an iwip automorphism. Let $E_c$ be the contracting space and, recall the notations $P_a$, $P_n$ and $P_\gamma$ as above.
Do there exist prototiles $P_a\subseteq E_c$ with the following properties?
\begin{enumerate}
  \item\label{item:shape-quest-properties-dual-substitution} for each $n$ and $a\in A$, the patch $P_n(a)$ is (a) connected, (b) simply connected, (c) disk-like or, (d) a tree;
  \item\label{item:adjacency-quest-properties-dual-substitution} for two distinct finite paths $\gamma$ and $\gamma'$ of length
    $n$ in the prefix-suffix automaton, if in the repelling tree
    $\Omega_\gamma\cap\Omega_{\gamma'}\neq\emptyset$ then
    $P_\gamma\cap P_{\gamma'}\neq\emptyset$;
  \item\label{item:tiling-quest-properties-dual-substitution} the patch $P_0$ tiles the plane under the action of the lattice $\Lambda$.
  \end{enumerate}
\end{quest}

From the above discussion, the prototiles $P_a=\pi_c((0,a)^*)$
satisfies
conditions~\ref{item:shape-quest-properties-dual-substitution}~(a)-(b)-(c) and \ref{item:tiling-quest-properties-dual-substitution}.
The polygonal prototiles given by $P_a=\pi_c(E_1^*(\sigma)^k((0,a)^*))$ are another choice which satisfies condition~\ref{item:tiling-quest-properties-dual-substitution}. 
Whereas fractal prototiles $P_a=\mathcal R(a)$ satisfy
conditions~\ref{item:shape-quest-properties-dual-substitution}(a) and
\ref{item:adjacency-quest-properties-dual-substitution}.

\section{Tree substitutions}

\subsection{Singular points}\label{sec:singular-points}

The next lemma is a consequence of Moss\'e's recognizability
results~\cite{mosse-puissances,mosse-reconnaissabilite}.

\begin{lem}\label{lem:mosse}
Let $\sigma$ be a primitive substitution which admits a non-periodic
fixed point $u$. There exists an integer $L > 0$ such that if
$w_{[-L,L]} = w'_{[-L,L]}$ for some $w,w'\in X_\sigma$, then the
prefix-suffix expansions of $w$ and $w'$ have the same final edge.
\end{lem}

Combining Propositions~\ref{prop:finite-index} and \ref{prop:Q-injective} we get the following result.

\begin{prop}\label{prop:finite-singular}
  Let $\sigma$ be a primitive substitution and an iwip
  automorphism. Then, there exist finitely many pairs $(Z,Z')$ of 
  distinct bi-infinite words in the attractive shift $X_\sigma$ such that
  $Q(Z)=Q(Z')$, and such that the prefix-suffix expansions of $Z$ and 
  $Z'$ do not have a common final edge.
\end{prop}
\begin{proof}
  From Proposition~\ref{prop:Q-injective} there exists a finite
  sequence $Z=Z_0$, $Z_1,\ldots,$ $Z_{n-1},Z_n=Z'$ in $X_\sigma$ such
  that $Z_i$ and $Z_{i+1}$ share a half. Assume that this sequence is
  the shortest one. We discuss two cases. If $n>1$, then (up to a
  symmetric argument) $Z$ and $Z_1$ share their left halves and $Z_1$
  and $Z_2$ share their right halves. Thus there exist $i,j\geq 0$
  such that the left half of $S^i Z_1$ is right-special and the right
  half of $S^{-j}Z_1$ is left-special. By
  Proposition~\ref{prop:finite-index}, there are finitely many
  right-special and left-special words, each of them with finitely
  many extensions. We conclude that there are finitely many possible
  choices for $Z_1$, $i$ and $j$ and thus for $Z$. By symmetry there
  are finitely many possible choices for $Z'$.

  We now study the second case, that is to say when $n=1$. Again, up
  to a symmetric argument we assume that $Z$ and $Z'$ share their left
  halves. As above, there exists $i\geq 0$ such that the left half of
  $S^i Z$ is right-special. Then the indexed finite words $Z_{[-i,i]}$
  and $Z'_{[-i,i]}$ are equal. By Lemma~\ref{lem:mosse} there are at
  most finitely many such words $Z$ and $Z'$ whose prefix-suffix
  expansions have a different final edge, precisely those for which
  $i<L$. Since there are finitely many right-special words we
  conclude that there are finitely many such pairs $Z$ and $Z'$.
\end{proof}

A pair of infinite desubstitution paths $(\gamma,\gamma')$ is
\textbf{singular} if $Q(\gamma)=Q(\gamma')$ and $\gamma$ and $\gamma'$
have different final edges. An infinite desubstitution path $\gamma$ is
\textbf{singular} if it belongs to at least one singular pair. The
above proposition states that there are finitely many singular pairs
of infinite desubstitution paths. From the proof we get that for
each singular infinite desubstitution path $\gamma=\Gamma(Z)$, $Z\in X_\sigma$ is a bi-infinite word in the shift orbit of a singular word. A direct consequence of Proposition~\ref{prop:singular-preperiodic} is the following.

\begin{prop}\label{prop:finite-singular-desubstitutions}
  Let $\sigma$ be a primitive substitution that induces an iwip
  automorphism. Then, there exist finitely many singular infinite
  desubstitution paths. Each of them is eventually periodic.
\end{prop}

From the self-similar decomposition of the attracting shift and of the
limit set (Propositions~\ref{prop:cylinders-autosimilar} and
\ref{prop:autosimilar-tree}), we get the next Proposition.

\begin{prop}\label{prop:singular-periodic}
  Let $\sigma$ be a primitive substitution which induces an iwip
  para-geometric automorphism. Let $\gamma$ and $\gamma'$ be two
  distinct finite paths of the same length in the prefix-suffix
  automaton such that the subtrees $\Omega_\gamma$ and
  $\Omega_{\gamma'}$ have the point $P$ in common. Then there exist
  bi-infinite singular words $Z\in [\gamma]$ and $Z'\in [{\gamma'}]$
  such that $Q(Z)=Q(Z')=P$. 
  The infinite desubstitution paths of $Z$ and $Z'$ are eventually
  periodic and satisfy
  \[
    \Gamma(Z)=\alpha\beta\delta,\ \Gamma(Z')=\alpha\beta'\delta',
  \]
  with $\alpha$ the common head of $\gamma$ and $\gamma'$:
  $\gamma=\alpha\beta$, $\gamma'=\alpha\beta'$ and where
  $(\beta\delta,\beta'\delta')$ is a singular pair of infinite
  desubstitution paths.  
\end{prop}

Such an intersection point $P$ is called a \defi{singular point} of $\Omega_\gamma$.

\subsection{Not renormalized singular points}\label{sec:not-renorm-sing}

In this Section we state a finiteness result for singular points. We
need to compare singular points of different tiles $\Omega_\gamma$
for different finite paths $\gamma$ in the prefix-suffix
automaton. This is the purpose of using not renormalized tiles and
not renormalized singular points.

Recall from Section~\ref{sec:renormalization} that for a finite path
$\gamma$ in the prefix-suffix automaton beginning at state $a_n\in A$:
\[
\widetilde{\Omega}_\gamma=\sigma^{-n}(p(\gamma)\inv)\Omega_{a_n}.
\]
Observe that for any finite path $\gamma$ in the prefix-suffix
automaton beginning at state $a$, for any distinct path $\gamma'$ with
the same length as $\gamma$ such that $ {\Omega}_\gamma
\cap {\Omega}_{\gamma'} = \{P\}$, using the notations of Proposition~\ref{prop:singular-periodic},
\[
P=Q(\alpha\beta\delta)=Q(\alpha\beta'\delta'),
\]
and $(\beta\delta,\beta'\delta')$ is a singular pair of infinite
desubstitution paths.

In the not renormalized form, 
\[
  \widetilde{\Omega}_\gamma \cap \widetilde{\Omega}_{\gamma'} =
  \{H^{-n}(P)\},
\]
and, using the above formula and Section~\ref{sec:cylind-repell-tree},
\[
P':=\sigma^{-n}(p(\gamma))H^{-n}(P)=H^{-n}(p(\gamma)Q(\gamma\delta))=Q(\delta) \in \Omega_a.
\]
We call $P'$ a \textbf{singular point} in $\Omega_a$.
For a letter $a\in A$, we consider the set of all such singular points $P'$:
\[
\Sing(\Omega_a)=\big\{P'\in\Omega_a\ |\ \exists\gamma,\gamma',\ |\gamma|=|\gamma'|=n,\ \gamma(n)=a,
\]
\[ \hspace{2cm}
{\Omega}_\gamma\cap {\Omega}_{\gamma'}=\{P\},\ P'=\sigma^{-n}(p(\gamma))H^{-n}(P)\big\}.
\]

\begin{prop}\label{prop:finite-singular-points} 
For every $a\in A$, the set $\Sing(\Omega_a)$ in the subtree $\Omega_a$ is finite.
\end{prop}
\begin{proof}
  From the above discussion and notations, given $\gamma,\gamma'$ as
  in Proposition~\ref{prop:singular-periodic}, there exist finite paths $\alpha,\beta,\beta'$
  in the prefix-suffix automaton and infinite desubstitution paths
  $\delta, \delta'$ such that
\[
  \gamma=\alpha\beta,\ \gamma'=\alpha\beta',\
  P=Q(\alpha\beta\delta)=Q(\alpha\beta'\delta'),\ P'=Q(\delta)
\]
and, $(\beta\delta,\beta'\delta')$ is a singular pair. Note that if
$\beta$ and $\beta'$ ends at the same letter, then $\alpha$ and thus
$\beta$ and $\beta'$ must contain at least one edge.

Recall from Proposition~\ref{prop:finite-singular-desubstitutions}
that there are finitely many such singular pairs and each singular
infinite desubstitution path is eventually periodic. Thus there are at most
finitely many tails $\delta$, which proves the finiteness of
$\Sing(\Omega_a)$.
\end{proof}

The proof of Proposition~\ref{prop:finite-singular-points} gives an explicit method to compute the set of singular points $\Sing(\Omega_a)$ for any $a\in A$. This is the first step of the algorithm for constructing the tree substitution associated with $\sigma$:

\begin{algo}\label{algo:singular-points}
  \begin{enumerate}
  \item Compute the pairs of singular bi-infinite words in the
    attracting shift $X_\sigma$. This is a classical computation of
    special infinite words for substitution specialists or,
    alternatively, it amounts to compute periodic Nielsen paths which
    has been implemented~\cite{coulbois-sage,cl-long-turns}.
  \item Compute the pairs of singular prefix-suffix expansions. From
    Proposition~\ref{prop:finite-singular-desubstitutions} this is a
    finite set and its computation is detailed in the proof of
    Proposition~\ref{prop:finite-singular}. For each pair of singular
    bi-infinite words $(Z,Z')$, for $i=0,1,\ldots$ and
    $i=-1,-2,\ldots$, compute the prefix-suffix expansions $\gamma_i$
    and $\gamma'_i$ of the shifted pairs $(S^iZ,S^iZ')$. Then
    $(\gamma_i,\gamma'_i)$ is a pair of singular prefix-suffix
    expansions as long as they end with different edges in the
    prefix-suffix expansion. From Lemma~\ref{lem:mosse}, the index $i$ is
    bounded above by Mossé's constant $L$.
  \item For each singular prefix-suffix expansion (which is an
    eventually periodic path), compute the finitely many infinite
    paths $\delta$ such that $\gamma=\beta\delta$ as in
    Proposition~\ref{prop:singular-periodic} and in the proof of
    Proposition~\ref{prop:finite-singular-points}. The paths $\delta$
    ending in $a$ are in one-to-one correspondance with the singular
    points in $\Sing(\Omega_a)$.
  \end{enumerate}
\end{algo}

\begin{example}\label{ex:tribo-singular-points}
Let $\sigma$ be the Tribonacci substitution and recall the infinite singular words $w_a$, $w_b$, $w_c$, $w'_a$, $w'_b$ and $w'_c$ with their infinite desubstitution paths $\gamma_0$, $\gamma_a$, $\gamma_b$ and $\gamma_c$ from Examples~\ref{ex:tribo-fix-points} and \ref{ex:tribo-periodic-desubstitution}. Each of these infinite desubstitution paths is eventually periodic: 
\[
\gamma_0=\alpha\cdot\alpha\cdot\alpha\cdots, \quad
\gamma_a=\big(a\stackrel{\epsilon,b}{\longleftarrow})\cdot\beta\cdot\beta\cdot\beta\cdots, \]
\[
\gamma_b=\big(b\stackrel{a,\epsilon}{\longleftarrow} a\stackrel{\epsilon,b}{\longleftarrow})\cdot\beta\cdot\beta\cdot\beta\cdots,\quad \gamma_c=\big(c
\stackrel{a,\epsilon}{\longleftarrow}b \stackrel{a,\epsilon}{\longleftarrow} a\stackrel{\epsilon,b}{\longleftarrow}\big)\cdot\beta\cdot\beta\cdot\beta\cdots,
\]
where $\alpha$ is $\big(a\stackrel{\epsilon,b}{\longleftarrow}\big)$
and $\beta$ is $\big(a \stackrel{\epsilon,b}{\longleftarrow}a
\stackrel{\epsilon,c}{\longleftarrow}b\stackrel{a,\epsilon}{\longleftarrow}\big)$.
In the repelling tree $\Trepul$ we consider the
point \[P=Q(w'_a)=Q(w'_b)=Q(w'_c).\]
From Propositions~\ref{prop:finite-singular} and
\ref{prop:finite-singular-points} we also have to consider the shifts
$w''_a=S\inv w'_a$, $w''_b=S\inv w'_b$, $w''_c=S\inv w'_c$, with their
infinite desubstitution paths $\gamma'_a$, $\gamma'_b$ and, $\gamma'_c$:
\[
\gamma'_a=\big(a\stackrel{\epsilon,\epsilon}{\longleftarrow}c
\stackrel{a,\epsilon}{\longleftarrow}b\stackrel{a,\epsilon}{\longleftarrow}a\stackrel{\epsilon,b}{\longleftarrow}\big)\cdot\beta\cdot\beta\cdot\beta\cdots, 
\]
\[
\gamma'_b=\big(
a\stackrel{\epsilon,b}{\longleftarrow}
a\stackrel{\epsilon,b}{\longleftarrow}\big)
\cdot\beta\cdot\beta\cdot\beta\cdots,
\]
\[
\gamma'_c=\big(
a\stackrel{\epsilon,c}{\longleftarrow}
b\stackrel{a,\epsilon}{\longleftarrow}a
\stackrel{\epsilon,b}{\longleftarrow}\big)
\cdot\beta\cdot\beta\cdot\beta\cdots,
\]
and their common image $a P=Q(w''_a)=Q(w''_b)=Q(w''_b)$ in the repelling tree.

From the proof of Proposition~\ref{prop:finite-singular-points}, the
singular points are given by the finitely many infinite tails $\delta$ of
$\gamma'_a$, $\gamma'_b$ and $\gamma'_c$ which are obtained by
beheading up to $6$ edges (recall from Section~\ref{sec:subst-attr-shift} the notation $B(\gamma)$ to behead a prefix-suffix expansion).
The singular prefix-suffix expansions
$\gamma_a=B(\gamma_b)=B^2(\gamma_c)=B^3(\gamma'_a)=\ldots$,
$B(\gamma_a)$ and $B^2(\gamma_a)$ give three singular points in
$\Sing(\Omega_a)$, while $\gamma_b=B(\gamma_c)=\ldots$ and
$B^3(\gamma_a)$ give the two singular points in $\Sing(\Omega_b)$;
finally, $\gamma_c$ gives the unique singular point in
$\Sing(\Omega_c)$.

We use the formulas in
Section~\ref{sec:not-renorm-sing} to identify
$F_A$-orbits and we get:
\[
\Sing(\Omega_a)=\{P, b\inv P, c\inv P\},\ \Sing(\Omega_b)=\{P, a\inv P\},\ \Sing(\Omega_c)=\{P\}.
\]
\end{example}

\subsection{Tree substitution inside $\Trepul$}\label{sec:tree-substitution-inside}

Using Proposition~\ref{prop:finite-singular-points} we can define the tree substitution $\sigma_T$.

Let $Y_a$ be the subtree of $\Omega_a$ spanned by the finite set of points $\Sing(\Omega_a)$ of Proposition~\ref{prop:finite-singular-points}. 
We consider the tree substitution $\sigma_T$ which maps each of the finite tree $Y_a$ to the tree
\begin{equation}\label{eq:sigmaT}
\sigma_T(Y_a)=\bigcup_{a\stackrel{p,s}{\longleftarrow}b}\sigma^{-1}(p^{-1}) Y_b, 
\end{equation}
where the union is taken over all edges $a\stackrel{p,s}{\longleftarrow}b$ in the prefix-suffix automaton (compare with the definition of dual substitution in Section~\ref{sec:dual-substitution}). This union is to be understood as a union of subtrees in the repelling tree $\Trepul$. 

For any word $u \in F_A$, the tree substitution extends to translates
$u\,Y_a$ of the finite tree $Y_a$ as
\[
\sigma_T(u\,Y_a)=\bigcup_{a\stackrel{p,s}{\longleftarrow}b}\sigma\inv(u p\inv)\, Y_b.
\]

\begin{prop}\label{prop:connected-tree-substitution}
With the above hypotheses and notations, for any $a\in A$ and any $n>0$, the iterated image 
\[
{\sigma_T}^n(Y_a)=\bigcup_{\gamma=a\stackrel{p_0,s_0}{\longleftarrow}
\cdots \stackrel{p_{n-1},s_{n-1}}{\longleftarrow}b}\sigma^{-n}(p(\gamma)\inv) Y_b
\]
is connected. Furthermore, after renormalization, $H^{n}({\sigma_T}^n(Y_a))$ is a finite subtree of $\Omega_a$ which converges towards $\Omega_a$:
\[
\overline{\bigcup_{n\to+\infty}H^{n}({\sigma_T}^n(Y_a))}=\Omega_a.
\]
Finally, $Y_A=\bigcup_{a\in A}Y_a$ is a connected subtree of $\Omega_A$ and, its iterated images ${\sigma_T}^n(Y_A)$ are connected subtrees of $\Omega_A$ which converge towards $\Omega_A$.
\end{prop}
\begin{proof}
First we show that the right hand side of the first equality is a connected subtree. Recall from Proposition~\ref{prop:autosimilar-tree} that each subtree $\Omega_a$ is connected and decomposes as the union of the subtrees $\Omega_\gamma$, with $\gamma$ ranging over the paths of length $n$ in the prefix-suffix automaton ending at $a$. Two such subtrees $\Omega_\gamma$ and $\Omega_{\gamma'}$, with $\gamma$ and $\gamma'$ starting respectively at state $b$ and $b'$, have at most a point $P$ in common, and this point corresponds to a singular point in $\Sing(\Omega_b)$ and $\Sing(\Omega_{b'})$: 
\[
p(\gamma)H^{-n}(P)\in\Sing(\Omega_b)\quad\text{ and }\quad
p(\gamma')H^{-n}(P)\in\Sing(\Omega_{b'}).
\]
Thus the finite trees $\sigma^{-n}(p(\gamma)\inv) Y_b$ and $\sigma^{-n}(p(\gamma')\inv) Y_{b'}$ have the point $H^{-n}(P)$ in common and the union 
\[
\bigcup_{\gamma=a\stackrel{p_0,s_0}{\longleftarrow}
\cdots \stackrel{p_{n-1},s_{n-1}}{\longleftarrow}b}\sigma^{-n}(p(\gamma)^{-1}) Y_b
\]
is connected.

The convergence of the iterated images follows from the decomposition in Proposition~\ref{prop:autosimilar-tree}. Indeed, each of the subtrees $\Omega_\gamma$ contains a point of the corresponding $H^n(\sigma^{-n}(p(\gamma)\inv) Y_b)$ and the diameter of $\Omega_\gamma$ is $(\frac{1}{\lambda_{\sigma\inv}})^n$ times the diameter of $\Omega_b$ which goes to $0$.
We let the reader extend the proof to the union $Y_A$ of the elementary trees $Y_a$.
\end{proof}

The above approximation of $\Omega_A$ by finite subtrees uses the
action of the free group on the repelling tree. This action is more
difficult to handle than translations in the plane, that we used for
the dual substitution in Section~\ref{sec:dual-substitution}.

It is much easier to describe how to glue together subtrees at a
common vertex instead of considering the action of the free group. And
fortunately we know that these gluings occur only at singular
points. In the next Section we will leave the fact that the trees
$Y_a$ are subtrees of the repelling tree and construct an abstract
tree substitution.

\subsection{Combinatorial Tree Substitutions}

\subsubsection{Definition}\label{sec:definition-tree-substitution}

Let us now use the previous Section to define a purely combinatorial tree substitution. Our aim is to describe how to approximate the limit set $\Omega_A$ from finite data.
We will use the notion of tree substitutions developed by Jullian~\cite{jullian-coeur} with minor adaptations.

A tree substitution starts with a finite collection of finite
simplicial trees $(W_a)_{a\in A}$, each of which has a finite set
of special vertices $V_a$ used to glue together other trees. Note that 
neither vertices in $V_a$ need to be leaves (i.e.~extremal points) nor 
all leaves need to be in $V_a$. All trees we consider are constructed 
as finite disjoint unions of copies of the trees $W_a$ by identifying 
some vertices:
\begin{equation}\label{eq:W}
W=\bigsqcup_{i=1}^n W_{b_i}/\sim,
\end{equation}
where $b_i\in A$ and $\sim$ is an equivalence relation which identifies some points of $V_{b_i}$ with some points of $V_{b_j}$ for some $i,j\in\{1,\ldots,n\}$.

The tree substitution $\tau$ replaces the tree $W_a$ by a tree $\tau(W_a)$ which is a union of copies of the $W_b$ with some vertices identified as in \eqref{eq:W}. The map $\tau$ also sends each gluing vertex in $V_a$ to some gluing vertex of some $V_{b}$.

The tree substitution extends to any tree as above:
\[
\tau(W)=\bigsqcup_{i=1}^n \tau(W_{b_i})/\sim
\]
where the vertices $v$ and $v'$ of $\tau(W_{b_i})$ and $\tau(W_{b_j})$ are identified if and only if there exists $u$ in $V_{b_i}$ and $u'$ in $V_{b_j}$ such that $u\sim u'$ in $W$, $\tau(u)=v$ and $\tau(u')=v'$.
From the construction it is obvious that the iterated images of a tree $W$ are trees.

\subsubsection{Combinatorial tree substitutions for parageometric automorphisms}\label{sec:combinatorial-tree-substitution-parageometric}

Now let $\sigma$ be a primitive substitution over the alphabet $A$ and a parageometric iwip automorphism of $F_A$. For each $a\in A$ consider the subtree $Y_a$ of the repelling tree $\Trepul$. We consider $Y_a$ as a simplicial tree by forgetting its metric structure and we let the set of vertices $V_a$ be the set of singular points $\Sing(\Omega_a)$.

In this framework, the tree substitution is defined by 
\begin{equation}\label{eq:tau_Ya}
\tau(Y_a)=\bigsqcup_{a\stackrel{p,s}{\longleftarrow}b}Y_b/\sim
\end{equation}
where the disjoint union is taken over all edges
$a\stackrel{p,s}{\longleftarrow}b$ in the prefix-suffix automaton and
where, for two such edges ${a\stackrel{p,s}{\longleftarrow}b}$ and
${a\stackrel{q,t}{\longleftarrow}b'}$, we identify the singular point
$P\in Y_b$ with the singular point $P'\in Y_{b'}$ if $\sigma^{-1}(p\inv)
P=\sigma^{-1}(q\inv) P'$ in the repelling tree $\Trepul$. Note that the simplicial
tree $\tau(Y_a)$ is homeomorphic to the subtree $\sigma_T(Y_a)$ of $\Trepul$ defined in \eqref{eq:sigmaT}.

Furthermore, a singular point $P\in\Sing(\Omega_a)\subseteq Y_a$ with
prefix-suffix expansion
$\gamma=a\stackrel{p,s}{\longleftarrow}b\longleftarrow\cdots$ is
mapped to $\sigma_T(P)=H\inv(P)$ which lies in
$\sigma\inv(p\inv)Y_b$. From
Section~\ref{sec:not-renorm-sing}, we know that the tail $B(\gamma)$
of $\gamma$ also gives a singular point, and thus in our new setting
the point $\sigma_T(P)$ is the point $H^{-1}(P) = \sigma^{-1}(p^{-1})
Q(B(\gamma))\in\Sing(\Omega_b)\subseteq Y_b$ in the
corresponding copy of $Y_b$.

This defines the map from the set of gluing points $V_a$ in $Y_a$ to the set of gluing points $V_b$ in $Y_b$.

The simplicial trees $\tau^n(Y_a)$ are obtained from their metric
counterpart described in Section~\ref{sec:tree-substitution-inside} by
forgetting the metric. To state a convergence result we use the
renormalization by the contracting homothety $H$. This renormalization
allows much more flexibility.

\medskip

So far we have considered the tree substitution inside the repelling
tree $T_{\sigma^{-1}}$ and the abstract combinatorial tree
substitution where we forget the metric structure of the tree keeping
only the simplicial one. We want to show that what is really important
in these constructions is the set of gluing points, that is, the
singular points. Indeed, no matter which simplicial or metric
structure we give to the tree, we will converge (in the
Gromov-Hausdorff sense) iterating the tree substitution and
renormalizing to the compact limit set $\Omega_A$.

To this purpose let $W_a$ be any simplicial tree with a finite set of gluing vertices in bijection with the set $V_a$ of gluing vertices of $Y_a$, for each $a\in A$. Let $\tau$ be the tree substitution in \eqref{eq:tau_Ya} defined on the $W_a$:
\[
\tau(W_a)=\bigsqcup_{a\stackrel{p,s}{\longleftarrow}b}W_b/\sim
\]
where the disjoint union is taken over all edges $a\stackrel{p,s}{\longleftarrow}b$ in the prefix-suffix automaton and where for two such edges ${a\stackrel{p,s}{\longleftarrow}b}$ and ${a\stackrel{q,t}{\longleftarrow}b'}$ we identify the gluing points $P$ of $W_b$ and $P'$ of $W_{b'}$ if the corresponding points of $Y_b$ and $Y_{b'}$ are identified. Moreover $\tau$ maps the gluing points of $W_a$ to the gluing points of the $W_b$ exactly as $\sigma_T$ maps the gluing points of $Y_a$ to the gluing points of the $Y_b$.

Let us fix any metric on each of the $W_a$ by fixing the lengths of edges. We get that $\tau^n(W_a)$ is also a metric tree and we renormalize this metric by contraction by $(\lambda_{\sigma\inv})^{-n}$. We denote this metric tree by $(\lambda_{\sigma\inv})^{-n}\tau^n(W_a)$.

\begin{prop}\label{prop:tree-substitution-converge}
Let $\sigma$ be a primitive substitution and a parageometric iwip
automorphism. Let $\tau$ be any tree substitution such that the gluing
vertices $V_a$ of each prototile $W_a$ are in bijection with the
singular points in $\Sing(\Omega_a)$ and such that the gluing
instructions mimic the tree substitution inside the repelling
tree. Then, the renormalized iterated images
$(\lambda_{\sigma\inv})^{-n}\tau^n(W_a)$ converge in the
    Gromov-Hausdorff topology towards the compact real tree
    $\Omega_a$.

Moreover, the renormalized iterated images of $W=\bigsqcup_{a\in A}
W_a/\sim$ converge towards the compact limit set $\Omega_A$.
\end{prop}
\begin{proof}
We first define the matrix of distances for the tree substitution $\tau$. Let $u$ and $v$ be distinct gluing points in $W_a$. After one step of the tree substitution those points are mapped to gluing points $u'$ and $v'$ in $\tau(W_a)$. The segment $[u';v']$ in $\tau(W_a)$ crosses some of the tiles $W_b$ and, for each such tile $W_b$, it crosses a segment between two gluing points of $W_b$. Thus there exist $n$ tiles $W_{b_0},\ldots,W_{b_n}$ and distinct gluing points $x_i,y_i$ in $W_{b_i}$ with $x_0=u'$, $y_i\sim x_{i+1}$, for $i=0,\ldots,n-1$, and $y_n=v'$. As we are in a tree, the distances add up and we get that 
\[
d(\tau(u),\tau(v))=\sum_{i=0}^n d(x_i,y_i).
\]
Repeating this observation for all pairs of distinct gluing points in each $W_a$, we get a \defi{matrix of distances} $M$ with non-negative integer entries, such that if $\vec d=(d_{W_a}(u,v)))_{a\in A, u\neq v\in V_a}$ and $\vec d'=(d_{\tau(W_a)}(\tau(u),\tau(v)))_{a\in A, u\neq v\in V_a}$:
\[\vec d'=M\vec d.\]
Recall that the set of gluing points $V_a$ is the set $\Sing(\Omega_a)\subset Y_a$. Moreover the abstract
tree substitution $\tau$ mimics the tree substitution $\sigma_T$ inside
$\Trepul$. Thus, if we start with the vector $\vec
d_\sigma=(d_{\Trepul}(u,v))_{a\in A, u\neq v\in \Sing(\Omega_a)}$ of
distances between pairs of distinct gluing points in $Y_a$, using the
homothety $H$ we get:
\[\vec d_\sigma=\frac{1}{\lambda_{\sigma\inv}}M\vec d_\sigma.\]
This proves that the matrix of distances $M$ has for dominant
eigenvalue $\lambda_{\sigma\inv}>1$ associated with the positive
eigenvector $\vec d_\sigma$.

It is known that the repelling tree of an iwip automorphism
$\Trepul$ is indecomposable~\cite{ch-b}, this implies that for any two
non-degenerate arcs $I$ and $J$ in the limit set $\Omega_A$, there exists $u\in F_A$
such that $K=I\cap uJ$ is a non-degenerate arc. Let $P$ be a point in
the interior of $K$ with exactly one pre-image by $Q$ (both in the
attracting shift $X_\sigma$ and in the set of infinite
prefix-suffix expansion paths $\mathcal P$) and such that $P$ is not a
branch point in $\Trepul$. Note that such a point $P$ exists as both
branch points and singular points are countable. Let
$Z\in X_\sigma$ and $\gamma\in\mathcal P$ be the preimages:
$Q(Z)=Q(\gamma)=P$. We assume furthermore that both the prefix and
suffix of $\gamma$ are unbounded words (this is again always the case
outside a countable subset of the limit set $\Omega_A$). It is a
property of the limit set~\cite{chl4} (and indeed of the compact
heart) that, as $P$ and $uP$ are points in $\Omega_A$, $uZ$ is in the
attracting shift $X_\sigma$, see also Proposition~\ref{prop:global-picture}. In our
situation, as $\sigma$ is a substitution, $u$ is either a negative
word  and $u\inv$ is a prefix of the right half $Z_{[0;+\infty)}$ or $u$ is a
positive word and a suffix of the left half
$Z_{(-\infty;-1]}$.

We pick $\gamma'$ a prefix of $\gamma$, such that either $u$ is a
positive word and a suffix of $p(\gamma')$ or $u$ is a negative word
and $u\inv$ is a prefix of $a_0s(\gamma')$ with $a_0$ the ending
letter of $\gamma$. We get that there exists a path $\gamma''$ in the
prefix-suffix automaton of the same length as $\gamma'$ such that
\[u\Omega_{\gamma'}=\Omega_{\gamma''}.\]
Indeed, $\gamma''$ is obtained from $\gamma'$ by performing the Vershik
map (see Section~\ref{sec:vershik-map}) $|u|$ times. We remark that $\gamma'$ and $\gamma''$ start at the same letter
$a_n$.

Starting with any two arcs $I=[x_1;y_1]$ and $J=[x_2;y_2]$ joining
distinct gluing points in $W_{b_0}$ and $W_{c_0}$ we found an arc
$[x_0;y_0]$ between two distinct gluing points in $W_{a_n}$, such that
$H^n([x_0;y_0])=I\cap \Omega_{\gamma'}$ is a non-degenerate sub-arc of
$I$ and $u\inv H^n([x_0;y_0])=J\cap \Omega_{\gamma''}$ is a
non-degenerate sub-arc of $J$. This proves that the matrix $M^n$ has
positive coefficients in the columns for $d(x_1,y_1)$ and $d(x_2,y_2)$
at the line of $d(x_0,y_0)$.

This does not prove that the matrix $M$ is primitive (and we do not
claim or expect it is) but this is enough to prove that this matrix is
uniformly contracting for the Hilbert distance (see e.g.~\cite{yocc-2005}
\cite[Section~26]{viana-survey}) in the positive cone of the projective
space. And thus, no matter which distances we chose on the trees $W_a$
(and no matter which simplicial trees $W_a$ we chose, the only assumption
is that $W_a$ has the same gluing points $V_a$ as $Y_a$), for any two
gluing points $u,v$ in $V_a$ we have:
\[\lim_{n\to\infty}(\frac{1}{\lambda_{\sigma\inv}})^n d_{\tau^n(W_a)}(\tau^n(u),\tau^n(v))=d_{\Trepul}(u,v).\]
This extends to gluing points $u$ and $v$ in $\tau^m(W_a)$:
\[\lim_{n\to\infty}(\frac{1}{\lambda_{\sigma\inv}})^n d_{\tau^{m+n}(W_a)}(\tau^n(u),\tau^n(v))=d_{\Trepul}(u,v).\]
We get that the sets of gluing points in $\tau^m(W_a)$, as a metric
space, converge in the Gromov-Hausdorff topology to the limit set
$\Omega_a$.  We let the reader extend the proof to $W$.
\end{proof}

We complete now the algorithm to construct the tree substitution associated with $\sigma$.

\begin{algo}\label{algo:gluing-and-tau}
  Recall that the sets of singular points $V_i = \Sing(\Omega_i)$, for all $i\in A$, have been
  computed in Algorithm~\ref{algo:singular-points}. We now proceed to
  compute the gluing instructions and the map $\tau$.
  \begin{enumerate}\setcounter{enumi}{3}
  \item From the set of pairs of singular prefix-suffix expansions we
    get the gluings between singular points in $V_i$ and $V_j$ with
    $i\neq j$. From these gluing instructions we build the initial patch $W=\bigsqcup W_i/\sim$.
  \item Each gluing point $P\in V_a$ is given by its prefix-suffix
    expansion $\gamma$. It is mapped by $\tau$ to a gluing point
    $\tau(P) \in V_b\subseteq W_b$, where $b$ comes
    from the last edge $e_1(\gamma)=a\stackrel{p,s}{\longleftarrow} b$
    of $\gamma$. The point $\tau(P)\in V_b$ is given by
	\[ \tau(P) = \sigma^{-1}(p^{-1})P' \]    
    and $P'$ has prefix-suffix expansion $B(\gamma)$.
  \item Let $a\stackrel{p,s}{\longleftarrow} b$ and
    $a\stackrel{p',s'}{\longleftarrow} b'$ be two distinct edges of the
    prefix-suffix automaton ending at the same letter $a$.  The points
    $P\in V_b$ and $P'\in V_{b'}$ are identified in $\tau(W_a)$ if
    they are given by $\gamma$ and $\gamma'$ such that
    $((a\stackrel{p,s}{\longleftarrow}
    b)\cdot\gamma,(a\stackrel{p',s'}{\longleftarrow} b')\cdot\gamma')$ is a pair
    of singular prefix-suffix expansions.
  \end{enumerate}
\end{algo}

\begin{example}\label{ex:tribo-tree-substitution}
In Example~\ref{ex:tribo-singular-points} we computed the singular
points for the Tribonacci substitution.
According to \eqref{eq:sigmaT}, the Tribonacci tree substitution inside the repelling tree is 
\begin{align*}
\sigma_T: Y_a &\mapsto Y_a \cup Y_b \cup Y_c\\
Y_b &\mapsto c\inv Y_a \\
Y_c &\mapsto c\inv Y_b 
\end{align*}

Alternatively, we consider the abstract simplicial prototiles $W_a$,
$W_b$ and $W_c$ which are the convex hulls of the sets of gluing
points $V_a$, $V_b$ and $V_c$ in one-to-one correspondence with the
sets of singular points.  We now compute how to glue together the tiles of $W$ and $\tau(W)$ and describe the images by the map $\tau$ of the singular points.

The set $V_a$ consists of three gluing points numbered by $1$, $2$ and
$3$ (in red in Figure~\ref{fig:tribo-tree-substitution}). For example, the point $3$ comes from the singular prefix-suffix expansion
$\gamma_a$, it is mapped by $\tau$ to the singular point corresponding
to the beheaded prefix-suffix expansion $B(\gamma_a)$ in the copy of
$W_a$ coming from the ending edge $e_1(\gamma)=a\stackrel{\epsilon,b}{\longleftarrow}a$.  We repeat this computation for all the singular points. For each of them we give the singular prefix-suffix expansion $\gamma$, its image by the map $Q$ in the repelling tree, the singular prefix-suffix expansion of the image by $\tau$ in the copy of the prototile given by the heading edge $e_1(\gamma)$ of $\gamma$ and the image by the tree substitution $\tau$.
\[
\begin{array}{|c||c|c|c||c|c||c||}
\cline{2-7}
\multicolumn{1}{c||}{}&\multicolumn{3}{c||}{V_a}&\multicolumn{2}{c||}{V_b}&\multicolumn{1}{c||}{V_c}\\
\cline{2-7}
\multicolumn{1}{c||}{}&\circlebox{red}{1}&\circlebox{red}{2}&\circlebox{red}{3}&\circlebox{green}{1}&\circlebox{green}{2}&\circlebox{blue}{1}\\
\hline
\gamma&B(\gamma_a)&B^2(\gamma_a)&\gamma_a&B^3(\gamma_a)&\gamma_b&\gamma_c\\
\hline
Q(\gamma)&c\inv P&b\inv P&P&a\inv P&P&P\\
\hline
\hline
e_1(\gamma)&a\stackrel{\epsilon,b}{\longleftarrow}a&a\stackrel{\epsilon,c}{\longleftarrow}b&a\stackrel{\epsilon,b}{\longleftarrow}a&b\stackrel{a,\epsilon}{\longleftarrow}a&b\stackrel{a,\epsilon}{\longleftarrow}a&c\stackrel{a,\epsilon}{\longleftarrow}b\\
\hline
B(\gamma)&B^2(\gamma_a)&B^3(\gamma_a)&B(\gamma_a)&B(\gamma_a)&\gamma_a&\gamma_b\\
\hline
\tau (Q(\gamma))&b\inv P&a\inv P&c\inv P&c^{-2}P&c\inv P&c\inv P\\
\hline
\end{array}
\]
Finally the gluings between tiles are given by the pairs of singular prefix-suffix expansions. Both the tiles $W_a$, $W_b$ and, $W_c$ of the patch $W$ and of the patch $\tau(W_a)$ have the point $P=Q(\gamma_a)=Q(\gamma_b)=Q(\gamma_c)$ in common.

The tree substitution is described in Figure~\ref{fig:tribo-tree-substitution}.

\begin{figure}[h]
\centering
\begin{overpic}[scale=.7]{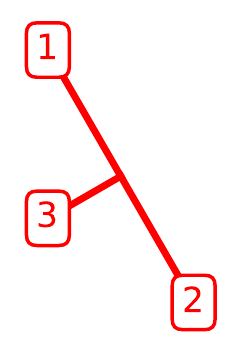}
\end{overpic}
\raisebox{.9cm}{$\mapsto$}
\begin{overpic}[scale=.7]{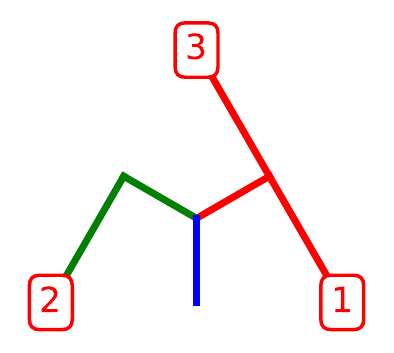}
\end{overpic}\hspace{2cm}
\begin{overpic}[scale=.7]{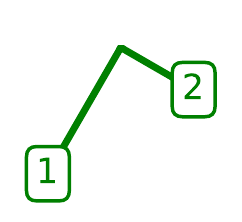}
\end{overpic}
\raisebox{.7cm}{$\mapsto$}
\begin{overpic}[scale=.7]{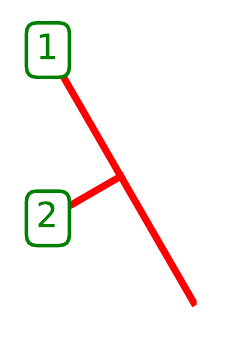}
\end{overpic}
\end{figure}

\begin{figure}[h]
\centering
\begin{overpic}[scale=.7]{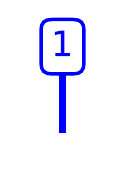}
\end{overpic}
\raisebox{.7cm}{$\mapsto$}
\begin{overpic}[scale=.6]{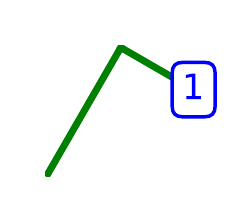}
\end{overpic}
\caption{Images of $W_a$ (red), $W_b$ (green) and $W_c$ (blue) by the Tribonacci tree substitution.\label{fig:tribo-tree-substitution}}
\end{figure}

\begin{figure}[h]
\centering
\includegraphics[scale=.6]{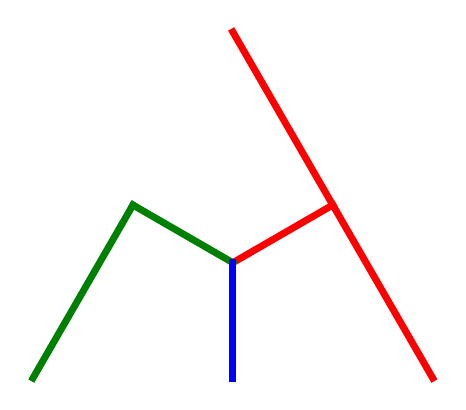}
\hspace{2cm}
\includegraphics[scale=.6]{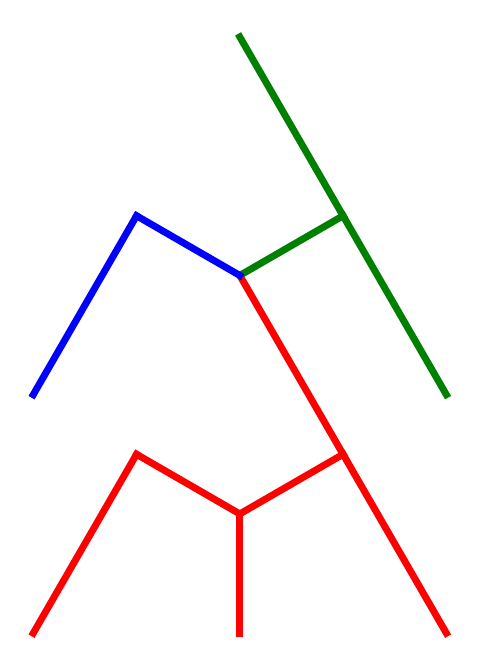}
\caption{Initial tree $W$ and its image $\tau(W)$ for the Tribonacci tree substitution.\label{fig:tribo-tree-0-and-1}}
\end{figure}

\end{example}

It may be frustrating to the reader to let too much freedom for the
choice of the simplicial trees $W_i$, for $i\in A$. It is however 
possible to discover the simplicial structure of the $Y_i$ by iterating the abstract tree substitution $\tau$. Then, it is possible to give to each $W_i$ the same simplicial structure as $Y_i$.

\begin{rem}\label{rem:simplicial-struct-prototiles}
  For $n$ big enough, the iterated images of the gluing points
  $\tau^n(V_a)$ lie in distinct copies of the $W_b$. For such an $n$,
  the simplicial tree spanned by $\tau^n(V_a)$ inside $\tau^n(W_a)$ is
  isomorphic to the simplicial structure of $Y_a$.

  Moreover, if $W_a$ and $Y_a$ have isomorphic simplicial structures,
  then all iterates $\tau^n(W)$ and $\sigma_T^n(Y)$ have isomorphic
  simplicial structures.
\end{rem}
\begin{proof}
  Recall that, for a singular point $P\in\Sing(\Omega_a)$ the tree
  substitution inside the repelling tree comes from the expanding
  homothety: $\sigma_T(P)=H\inv (P)$. As the prototiles $Y_a$ have
  bounded diameter, after finitely many iterations, the images
  $\sigma_T^n(P)$ of the fintely many singular points in
  $\Sing(\Omega_a)$ lie in distinct tiles $\Omega_\gamma$. As the
  abstract tree substitution $\tau$ mimics $\sigma_T$ we get that the
  points in $\tau^n(V_a)$ lie in distinct tiles $W_b$. Finally, as the
  tiles $Y_b$, $\Omega_\gamma$ and $W_b$ are trees, we conclude that
  the simplicial trees spanned by $\sigma_T^n(\Sing(\Omega_a))$ in the
  repelling tree $\Trepul$ and, by $\tau^n(V_a)$ in $\tau^n(W_a)$ are
  isomorphic.
\end{proof}

Another frustration may come from the lack of knowledge of the branch
points, in particular of their prefix-suffix expansions and precise
locations.

\begin{rem}\label{rem:branch-point-desubstitution}
  The prefix-suffix expansions of the branch points inside $Y_a$ are
  eventually periodic. They can be computed using the abstract tree
  substitution $\tau$.
\end{rem}
\begin{proof}
   The first startegy to describe the branch points is to study the
   inverse automorphism $\sigma\inv$.
   The duality between the repelling tree and the attracting shift
   implies that the branch points in $\Trepul$ are described by
   singular bi-infinite words of the repelling shift
   \cite{ch-a}. However, the inverse automorphism $\sigma\inv$ needs
   not be a substitution and the repelling shift is rather the
   repelling lamination which can be studied using a train-track
   representative of $\sigma\inv$, but this is beyond the scope of
   this paper.

   An alternative strategy is to give $W_a$ the same simplicial
   structure as $Y_a$ through
   Remark~\ref{rem:simplicial-struct-prototiles}. 

   As $Y_a$ is the convex hull of the singular points, each branch
   point $P$ is the center of a tripod of three singular points
   $P_1,P_2,P_3$. The tree substitution inside the repelling tree
   $\Trepul$ is defined by the contracting homothety and using the
   $\sigma_T(P_i)=H\inv(P_i)$ and thus, $H\inv (P)$ is the center $P'$
   of the tripod $\sigma_T(P_1)$, $\sigma_T(P_2)$ and,
   $\sigma_T(P_3)$. This center $P'$ lies in a translate of $Y_b$ for an
   edge $a\stackrel{p,s}{\longleftarrow}b$ of the prefix-suffix
   automaton. This edge is the ending edge of the prefix-suffix
   expansion of $P$. Recall that the translates of the $Y_c$ in
   $\sigma_T(Y_a)$ intersect in singular points, thus, $P'$ is a
   singular point or a branch point of the translate of $Y_b$. As
   before, $P'$ is the center of a tripod (possibly degenerate) of
   singular points in $Y_b$. Iterating this construction, we get the
   prefix-suffix expansion of $P$.

   As there are finitely many tripods in the finitely many prototiles
   $Y_b$, the above construction yields an eventually periodic
   prefix-suffix expansion for all branch points of $Y_a$.

   Finally, as the abstract trees $\tau^n(W_a)$ have the same
   simplicial structure as $\sigma_T^n(Y_a)$, the above construction
   can be achieved using the abstract tree substitution $\tau$ rather
   than the tree substitution $\sigma_T$ inside $\Trepul$. The
   prefix-suffix expansions of branch points can be computed from the
   abstract tree substitution $\tau$.
\end{proof}

Finally we remark that we can add extra gluing points to the abstract
tree substitution.

\begin{rem}\label{rem:comb-tree-subst-adding-gluing-vertices} 
 Let $P$ be a point in the limit set $\Omega_A$ with eventually
 periodic prefix-suffix expansion $\gamma$.  The set of tails
 $\{B^n(\gamma)\ | n\in\N\}$ of $\gamma$ is finite. The finitely many
 points $Q(B^n(\gamma))$ can be added to the sets of gluing points
 $V_{a_n}$. The map $\tau$ is then extended by mapping the gluing
 point $Q(B^n(\gamma))$ of $V_{a_n}\subseteq W_{a_n}$ to the gluing
 point $Q(B^{n+1}(\gamma))\in V_{a_{n+1}}\subseteq W_{a_{n+1}}$ of the
 copy of $W_{a_{n+1}}$ corresponding to the $n$-th edge
 $a_n\stackrel{p_n,s_n}{\longleftarrow}a_{n+1}$ of $\gamma$.
\end{rem}

\subsection{Tree Substitution inside the Rauzy fractal}\label{sec:tree-substitution-in-Rauzy}

We are now ready to use the two previous Sections to get a tree
substitution inside the Rauzy fractal.

Using the map $\psi$ of Proposition~\ref{prop:global-picture},
we embed the finite trees $Y_a$ of Section \ref{sec:tree-substitution-inside} inside the tiles $\mathcal{R}(a)$ of the Rauzy fractal. The tree substitution inside the Rauzy fractal $\sigma_R$ is defined by
\[
\sigma_R(\psi(Y_a))=\bigcup_{a\stackrel{p,s}{\longleftarrow}b} \psi(Y_b) + \pi_c(M_\sigma^{-1}\ell(p))
\]
where the union is taken over all edges
$a\stackrel{p,s}{\longleftarrow}b$ in the prefix-suffix automaton.
This version of the tree substitution inside the Rauzy fractal is quite
ineffective as we expect the embedding of an arc of the tree $Y_a$
into the contracting plane to be a Peano curve. We can rather focus on
the finite set of singular points.

For each $a\in A$, the finite set of points $\Sing(\Omega_a)$ is
mapped by $\psi$ to a finite set of points in the tile
$\mathcal{R}(a)$ of the Rauzy fractal. Recall that these points have
eventually periodic prefix-suffix expansions and that we gave an
explicit formula in Section~\ref{sec:preliminary-rauzy} to compute
their images in the contracting space.

For prefix-suffix expansions $\gamma$, $\gamma'$ and for the point $P$
as in Proposition~\ref{prop:singular-periodic}, $\psi(P)$ is a common
point of $\widetilde{\mathcal{R}_\gamma}$ and
$\widetilde{\mathcal{R}_{\gamma'}}$.

We can now consider trees spanned by the images of the singular points embedded inside the contracting space. This gives a tree substitution inside the Rauzy fractal. Remark that, as the subtile $\mathcal R(a)$ is the image by $Q$ of the compact tree $\Omega_a$, it is arcwise connected and we can choose such a linking arc between two singularities of $Y_a$ inside $\mathcal R(a)$.

The question that now arises is whether iterating this tree substitution inside the contracting space (or inside the Rauzy fractal) creates loops.

\begin{quest}Given a primitive parageometric iwip substitution, does there always exist a choice of the linking arcs between the images of the singular points in the contracting space such that the iteration of the tree substitution inside the contracting space does not create loops? Moreover, may the linking arcs be chosen inside the corresponding tiles of the Rauzy fractal without creating loops?
\end{quest}

The answer to this question is negative in general, as illustrated by the
case of the Tribonacci substitution. 
In Example~\ref{ex:tribo-tree-substitution}, after two
iterations the tree substitution embedded in the contracting plane
$E_c$ creates loops, see Figure~\ref{fig:tribo_2}.
And this is not because of a bad choice of the linking arcs: the loops
are created by the non-injectivity of the projection map $\psi$ on
singular points. It is not anymore visible that we are
working with trees. However this is only an artifact of the
representation inside the contracting plane, since the tree substitution was initially defined as an abstract tree substitution.

\begin{figure}[h]
\centering
\begin{overpic}[scale=.6]{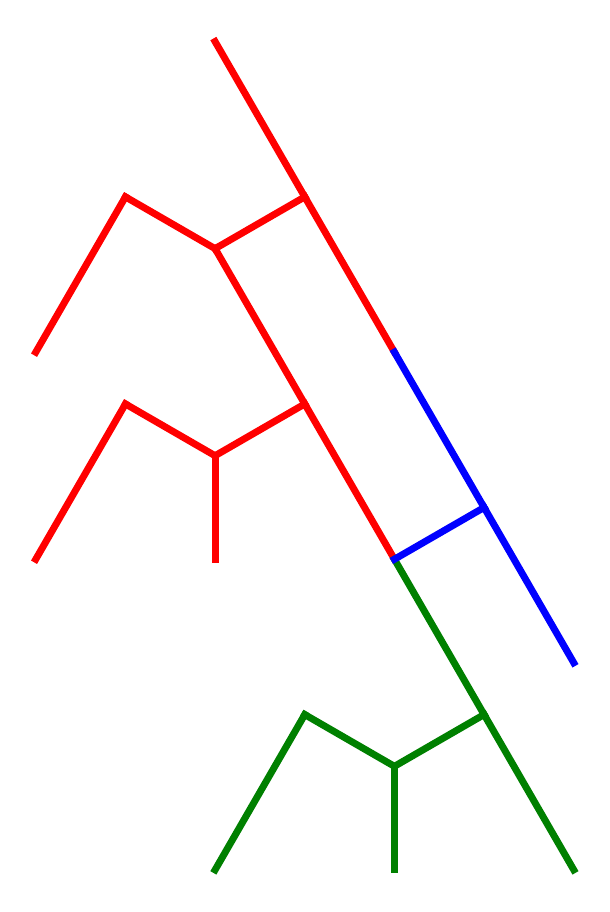}
\put (40,60) {$\bigcirc$ \parbox{5cm}{\scriptsize These two edges are not adjacent in the abstract tree substitution.}}
\end{overpic}
\caption{\label{fig:tribo_2} Second iterate of the Tribonacci tree substitution. The embedding into the contracting
  plane of $\tau^2(W)$ creates a loop.}
\end{figure}

We propose to overcome this difficulty by pruning the tree, as described in the next Section.

\subsection{Pruning and covering}\label{sec:pruning-cheating}

As we observed in Figure~\ref{fig:tribo_2} and
Example~\ref{ex:tribo-tree-substitution}, the tree substitution we
constructed in Sections~\ref{sec:combinatorial-tree-substitution-parageometric} and
\ref{sec:tree-substitution-in-Rauzy} may fail to produce a tree inside
the contracting space because of non-injectivity of the map $\psi$. We
propose in this section a pruning of the tree that will produce for
the Tribonacci substitution a tree inside the contracting plane.

We remark that in the tree substitution some branches are dead-end
which will never grow up: they are leaves of the tree for all steps
$n$. Thus we prune those branches. This will result into a new tree
substitution defined on more pieces: we have a new alphabet $A'$ with
a forgetful map $f:A'\to A$, and for each $a\in A'$ the prototile
$W_a$ is a subtile of the original prototile $W_{f(a)}$. Indeed, the
set of gluing points $V_a$ is a subset of the set of gluing points
$V_{f(a)}$.

As we kept at least one point for each tile, the proof of
Proposition~\ref{prop:tree-substitution-converge} applies and, we get
that the limit tree (after renormalization) is again $\Omega_A$.

\begin{example}\label{ex:tribo_pruned}
  For the Tribonacci tree substitution of
  Example~\ref{ex:tribo-tree-substitution}, pruning the dead-end
  vertices gives the pruned tree substitution described in
  Figure~\ref{fig:tribo-pruned-thierry}.

  It turns out that the choice of the linking arcs is inside the Rauzy
  fractal and that we now have a tree substitution inside the Rauzy
  fractal that does not create loops.  More spectacularly, this tree
  substitution occurs inside the tiles of the dual substitution as
  illustrated in Figure~\ref{fig:tree-substitution-pruned-tribo} and
  \ref{fig:prunedtribo-iterated}.

\begin{figure}[h]
\includegraphics[scale=.7]{images/tribo_prototile_a_color_label_ratio_07.pdf}
\hspace{-.5cm}
\raisebox{1cm}{$\mapsto$}
\hspace{-.3cm}
\includegraphics[scale=.7]{images/tribo_image_a_color_label_ratio_07.pdf}
\quad
\raisebox{.7cm}{\includegraphics[scale=.7]{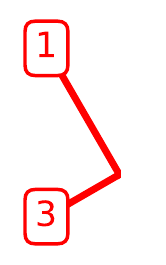}}
\hspace{-.5cm}
\raisebox{1cm}{$\mapsto$}
\hspace{-.3cm}
\includegraphics[scale=.7]{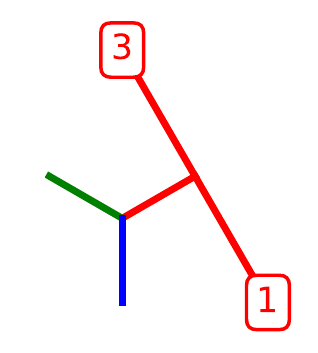}
\quad
\raisebox{.7cm}{\includegraphics[scale=.7]{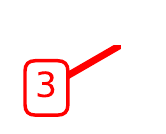}}
\hspace{-.5cm}
\raisebox{1cm}{$\mapsto$}
\hspace{-.3cm}
\includegraphics[scale=.7]{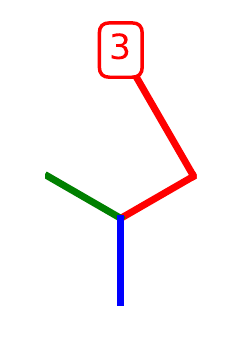}
\quad
\includegraphics[scale=.7]{images/tribo_prototile_b_color_label_ratio_07.pdf}
\raisebox{.5cm}{$\mapsto$}
\includegraphics[scale=.7]{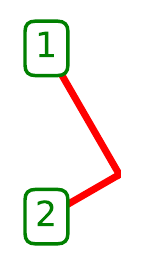}
\includegraphics[scale=.7]{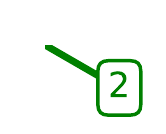}
\raisebox{.5cm}{$\mapsto$}
\includegraphics[scale=.7]{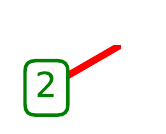}
\includegraphics[scale=.7]{images/tribo_prototile_c_color_label_ratio_07.pdf}
\raisebox{.5cm}{$\mapsto$}
\includegraphics[scale=.7]{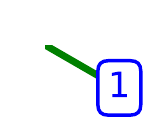}
\caption{\label{fig:tribo-pruned-thierry} The pruned Tribonacci tree substitution. In red the tiles: $W_a$, $W_{a'}$, $W_{a''}$; in green $W_b$ and $W_{b'}$; in blue $W_c$.}
\end{figure}

\begin{figure}[h]
\centering
\begin{tikzpicture}
[x={(-0.6062cm, -0.35cm)}, y={(0.6062cm, -0.35cm)},
z={(0.000000cm,0.700000cm)}]
\definecolor{facecolor}{rgb}{1.000,0.000,0.000}
\fill[fill=facecolor, draw=black, shift={(0,0,0)}]
(0, 0, 0) -- (0, 1, 0) -- (0, 1, 1) -- (0, 0, 1) -- cycle;
\draw[very thick] (0,0,0)--(0,1/2,1/2);
\node[circle,fill=black,draw=black,minimum size=1mm,inner sep=0pt] at
(0,0,0) {};
\end{tikzpicture}
\raisebox{.5cm}{$\mapsto$}
\begin{tikzpicture}
[x={(-0.6062cm, -0.35cm)}, y={(0.6062cm, -0.35cm)},
z={(0.000000cm,0.700000cm)}]
\definecolor{facecolor}{rgb}{0.000,1.000,0.000}
\fill[fill=facecolor, draw=black, shift={(0,0,0)}]
(0, 0, 0) -- (0, 0, 1) -- (1, 0, 1) -- (1, 0, 0) -- cycle;
\definecolor{facecolor}{rgb}{1.000,0.000,0.000}
\fill[fill=facecolor, draw=black, shift={(0,0,0)}]
(0, 0, 0) -- (0, 1, 0) -- (0, 1, 1) -- (0, 0, 1) -- cycle;
\definecolor{facecolor}{rgb}{0.000,0.500,1.000}
\fill[fill=facecolor, draw=black, shift={(0,0,0)}]
(0, 0, 0) -- (1, 0, 0) -- (1, 1, 0) -- (0, 1, 0) -- cycle;
\draw[very thick] (0,0,0)--(0,1/2,1/2)--(0,0,1) (0,0,0)--(1/2,0,1/2) (0,0,0)--(1/2,1/2,0);
\node[circle,fill=black,draw=black,minimum size=1mm,inner sep=0pt] at
(0,0,0) {};
\end{tikzpicture}\qquad
\begin{tikzpicture}
[x={(-0.6062cm, -0.35cm)}, y={(0.6062cm, -0.35cm)},
z={(0.000000cm,0.700000cm)}]
\definecolor{facecolor}{rgb}{1.000,0.000,0.000}
\fill[fill=facecolor, draw=black, shift={(0,0,0)}]
(0, 0, 0) -- (0, 1, 0) -- (0, 1, 1) -- (0, 0, 1) -- cycle;
\draw[very thick] (0,0,0)--(0,1/2,1/2)--(0,0,1);
\node[circle,fill=black,draw=black,minimum size=1mm,inner sep=0pt] at
(0,0,0) {};
\end{tikzpicture}
\raisebox{.5cm}{$\mapsto$}
\begin{tikzpicture}
[x={(-0.6062cm, -0.35cm)}, y={(0.6062cm, -0.35cm)},
z={(0.000000cm,0.700000cm)}]
\definecolor{facecolor}{rgb}{0.000,1.000,0.000}
\fill[fill=facecolor, draw=black, shift={(0,0,0)}]
(0, 0, 0) -- (0, 0, 1) -- (1, 0, 1) -- (1, 0, 0) -- cycle;
\definecolor{facecolor}{rgb}{1.000,0.000,0.000}
\fill[fill=facecolor, draw=black, shift={(0,0,0)}]
(0, 0, 0) -- (0, 1, 0) -- (0, 1, 1) -- (0, 0, 1) -- cycle;
\definecolor{facecolor}{rgb}{0.000,0.500,1.000}
\fill[fill=facecolor, draw=black, shift={(0,0,0)}]
(0, 0, 0) -- (1, 0, 0) -- (1, 1, 0) -- (0, 1, 0) -- cycle;

\draw[very thick] (0,0,0)--(0,1/2,1/2)--(0,0,1) (0,1/2,1/2)--(0,1,0) (0,0,0)--(1/2,0,1/2) (0,0,0)--(1/2,1/2,0);
\node[circle,fill=black,draw=black,minimum size=1mm,inner sep=0pt] at
(0,0,0) {};
\end{tikzpicture}
\qquad
\begin{tikzpicture}
[x={(-0.6062cm, -0.35cm)}, y={(0.6062cm, -0.35cm)},
z={(0.000000cm,0.700000cm)}]
\definecolor{facecolor}{rgb}{1.000,0.000,0.000}
\fill[fill=facecolor, draw=black, shift={(0,0,0)}]
(0, 0, 0) -- (0, 1, 0) -- (0, 1, 1) -- (0, 0, 1) -- cycle;
\draw[very thick] (0,0,0)--(0,1/2,1/2) (0,0,1)--(0,1,0);
\node[circle,fill=black,draw=black,minimum size=1mm,inner sep=0pt] at
(0,0,0) {};
\end{tikzpicture}
\raisebox{.5cm}{$\mapsto$}
\begin{tikzpicture}
[x={(-0.6062cm, -0.35cm)}, y={(0.6062cm, -0.35cm)},
z={(0.000000cm,0.700000cm)}]
\definecolor{facecolor}{rgb}{0.000,1.000,0.000}
\fill[fill=facecolor, draw=black, shift={(0,0,0)}]
(0, 0, 0) -- (0, 0, 1) -- (1, 0, 1) -- (1, 0, 0) -- cycle;
\definecolor{facecolor}{rgb}{1.000,0.000,0.000}
\fill[fill=facecolor, draw=black, shift={(0,0,0)}]
(0, 0, 0) -- (0, 1, 0) -- (0, 1, 1) -- (0, 0, 1) -- cycle;
\definecolor{facecolor}{rgb}{0.000,0.500,1.000}
\fill[fill=facecolor, draw=black, shift={(0,0,0)}]
(0, 0, 0) -- (1, 0, 0) -- (1, 1, 0) -- (0, 1, 0) -- cycle;
\draw[very thick] (0,0,0)--(0,1/2,1/2)--(0,0,1) (0,1/2,1/2)--(0,1,0) (0,0,0)--(1/2,0,1/2) (0,0,0)--(1/2,1/2,0) (1/2,0,1/2)--(1,0,0);
\node[circle,fill=black,draw=black,minimum size=1mm,inner sep=0pt] at
(0,0,0) {};
\end{tikzpicture}

\begin{tikzpicture}
[x={(-0.6062cm, -0.35cm)}, y={(0.6062cm, -0.35cm)},
z={(0.000000cm,0.700000cm)}]
\definecolor{facecolor}{rgb}{0.000,1.000,0.000}
\fill[fill=facecolor, draw=black, shift={(0,0,0)}]
(0, 0, 0) -- (0, 0, 1) -- (1, 0, 1) -- (1, 0, 0) -- cycle;
\draw[very thick] (0,0,0)--(1/2,0,1/2);
\node[circle,fill=black,draw=black,minimum size=1mm,inner sep=0pt] at
(0,0,0) {};
\end{tikzpicture}
\raisebox{.6cm}{$\mapsto$}
\begin{tikzpicture}
[x={(-0.6062cm, -0.35cm)}, y={(0.6062cm, -0.35cm)},
z={(0.000000cm,0.700000cm)}]
\definecolor{facecolor}{rgb}{1.000,0.000,0.000}
\fill[fill=facecolor, draw=black, shift={(0,0,1)}]
(0, 0, 0) -- (0, 1, 0) -- (0, 1, 1) -- (0, 0, 1) -- cycle;
\draw[very thick] (0,0,1)--(0,1/2,1+1/2);
\node[circle,fill=black,draw=black,minimum size=1mm,inner sep=0pt] at
(0,0,0) {};
\draw[dashed] (0,0,0)--(0,0,1);
\end{tikzpicture}
\qquad
\begin{tikzpicture}
[x={(-0.6062cm, -0.35cm)}, y={(0.6062cm, -0.35cm)},
z={(0.000000cm,0.700000cm)}]
\definecolor{facecolor}{rgb}{0.000,1.000,0.000}
\fill[fill=facecolor, draw=black, shift={(0,0,0)}]
(0, 0, 0) -- (0, 0, 1) -- (1, 0, 1) -- (1, 0, 0) -- cycle;
\draw[very thick] (0,0,0)--(1/2,0,1/2)--(1,0,0);
\node[circle,fill=black,draw=black,minimum size=1mm,inner sep=0pt] at
(0,0,0) {};
\end{tikzpicture}
\raisebox{.6cm}{$\mapsto$}
\begin{tikzpicture}
[x={(-0.6062cm, -0.35cm)}, y={(0.6062cm, -0.35cm)},
z={(0.000000cm,0.700000cm)}]
\definecolor{facecolor}{rgb}{1.000,0.000,0.000}
\fill[fill=facecolor, draw=black, shift={(0,0,1)}]
(0, 0, 0) -- (0, 1, 0) -- (0, 1, 1) -- (0, 0, 1) -- cycle;
\draw[very thick] (0,0,1)--(0,1/2,3/2)--(0,0,2);
\node[circle,fill=black,draw=black,minimum size=1mm,inner sep=0pt] at
(0,0,0) {};
\draw[dashed] (0,0,0)--(0,0,1);
\end{tikzpicture}

\begin{tikzpicture}
[x={(-0.6062cm, -0.35cm)}, y={(0.6062cm, -0.35cm)},
z={(0.000000cm,0.700000cm)}]
\definecolor{facecolor}{rgb}{0.000,0.500,1.000}
\fill[fill=facecolor, draw=black, shift={(0,0,0)}]
(0, 0, 0) -- (1, 0, 0) -- (1, 1, 0) -- (0, 1, 0) -- cycle;
\draw[very thick] (0,0,0)--(1/2,1/2,0);
\node[circle,fill=black,draw=black,minimum size=1mm,inner sep=0pt] at
(0,0,0) {};
\end{tikzpicture}
\raisebox{.6cm}{$\mapsto$}
\begin{tikzpicture}
[x={(-0.6062cm, -0.35cm)}, y={(0.6062cm, -0.35cm)},
z={(0.000000cm,0.700000cm)}]
\definecolor{facecolor}{rgb}{0.000,1.000,0.000}
\fill[fill=facecolor, draw=black, shift={(0,0,1)}]
(0, 0, 0) -- (1, 0, 0) -- (1, 0, 1) -- (0, 0, 1) -- cycle;
\draw[very thick] (0,0,1)--(1/2,0,3/2);
\node[circle,fill=black,draw=black,minimum size=1mm,inner sep=0pt] at
(0,0,0) {};
\draw[dashed] (0,0,0)--(0,0,1);
\end{tikzpicture}
\caption{\label{fig:tree-substitution-pruned-tribo} The pruned Tribonacci tree substitution inside the tiles of the dual substitution.}
\end{figure}
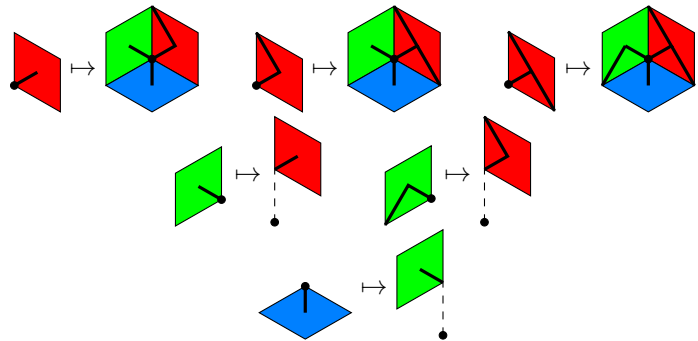

\begin{figure}[h]
\centering
\includegraphics[scale=.15]{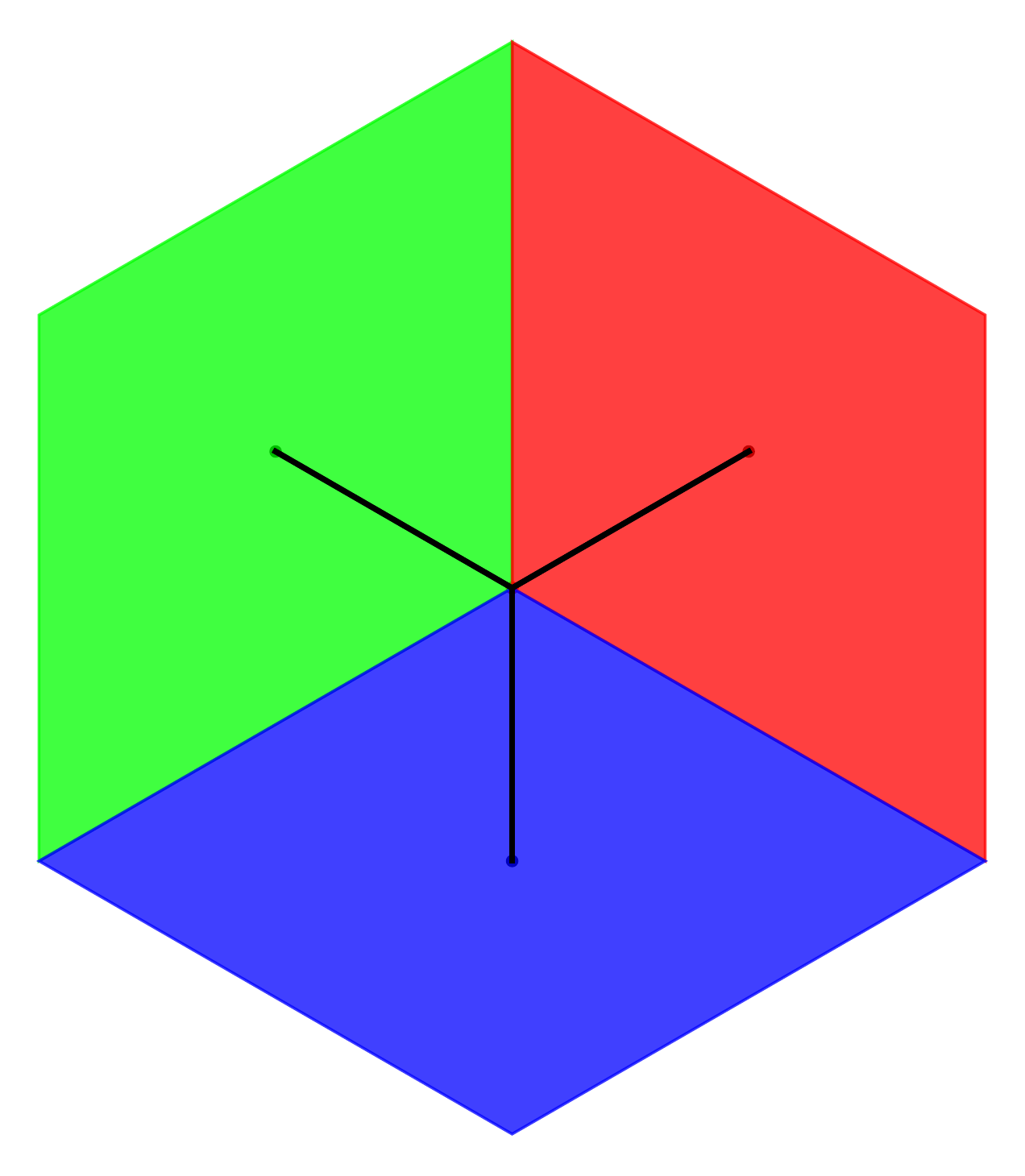}
\qquad 
\includegraphics[scale=.25]{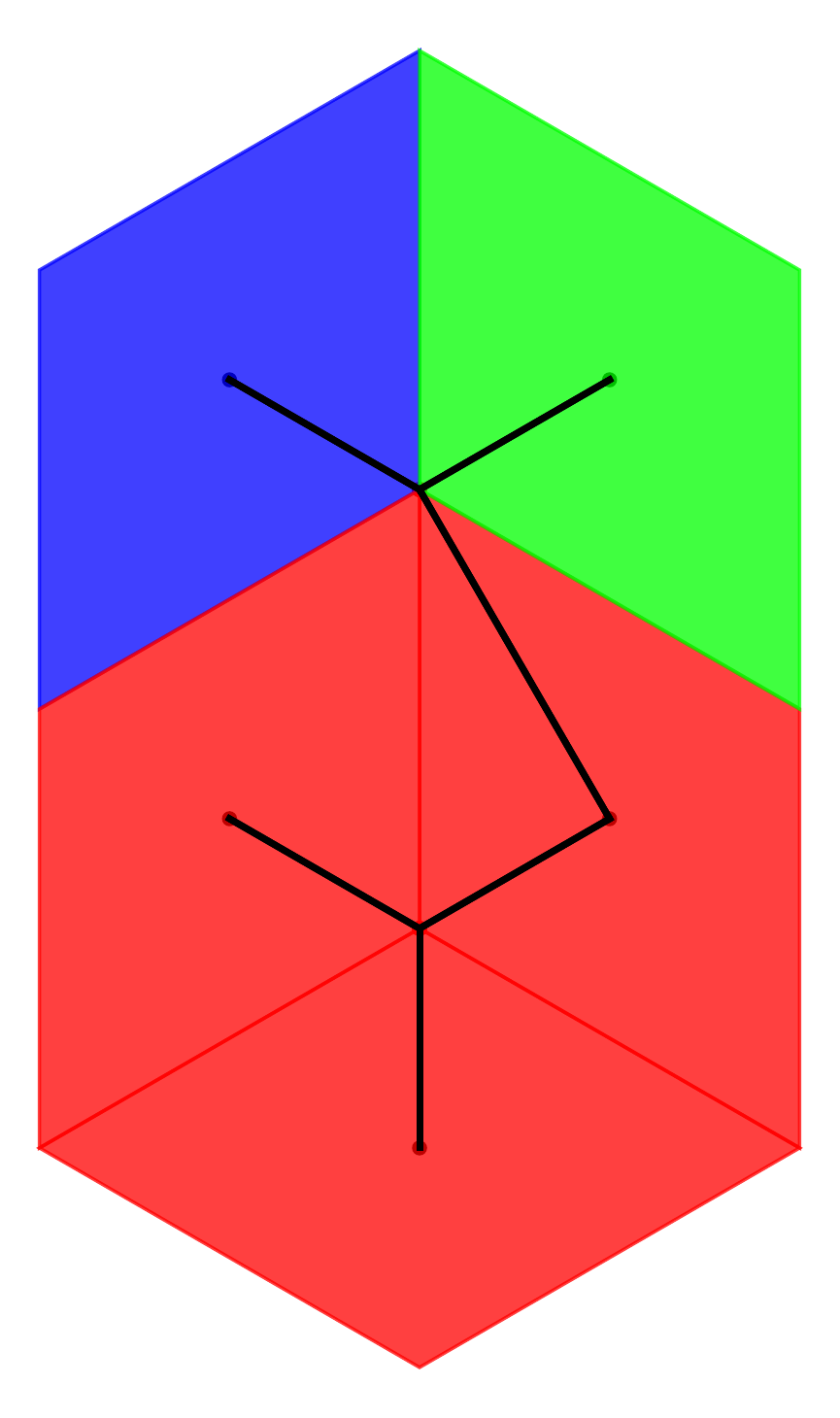}
\qquad 
\includegraphics[scale=.32]{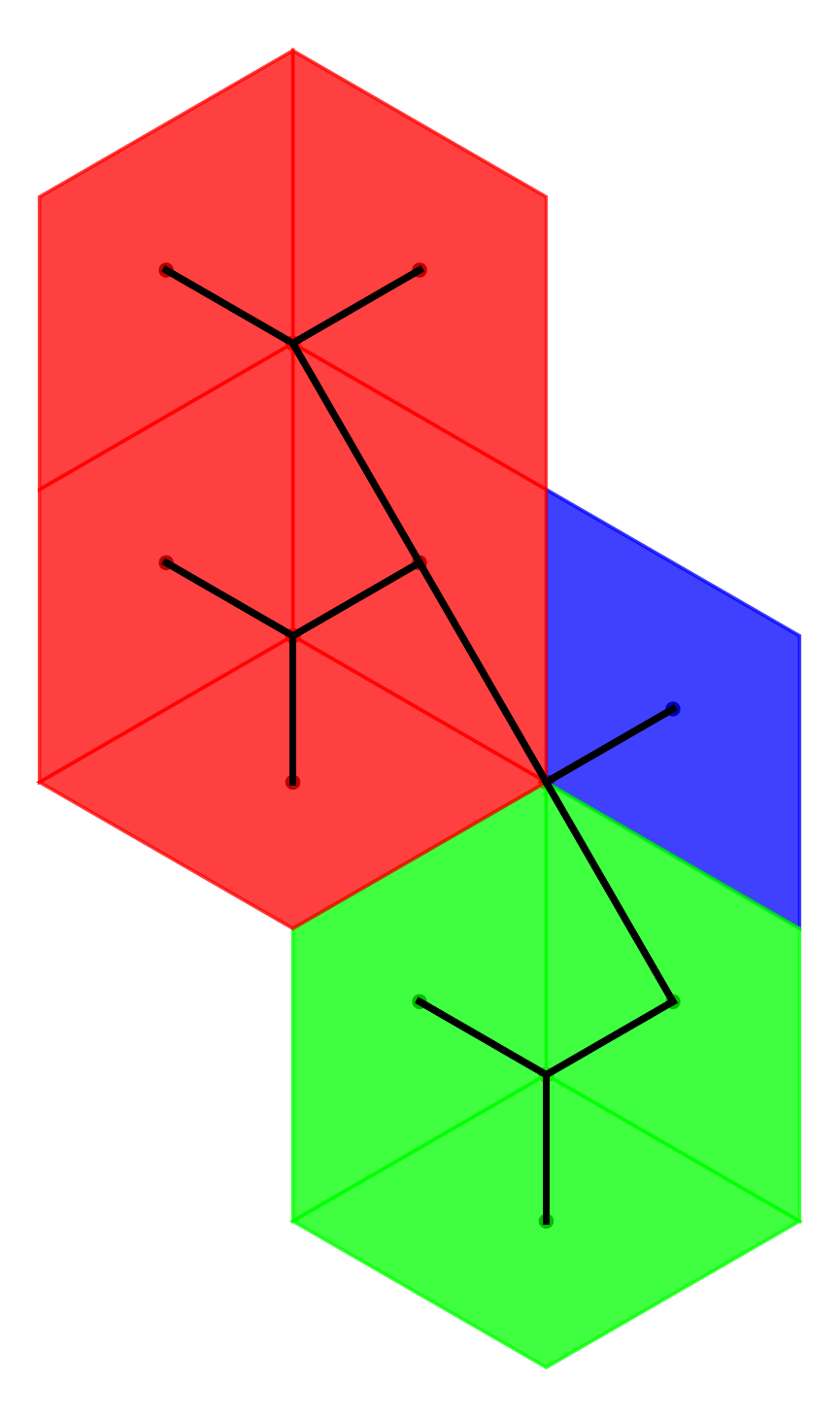}
 
\includegraphics[scale=.3]{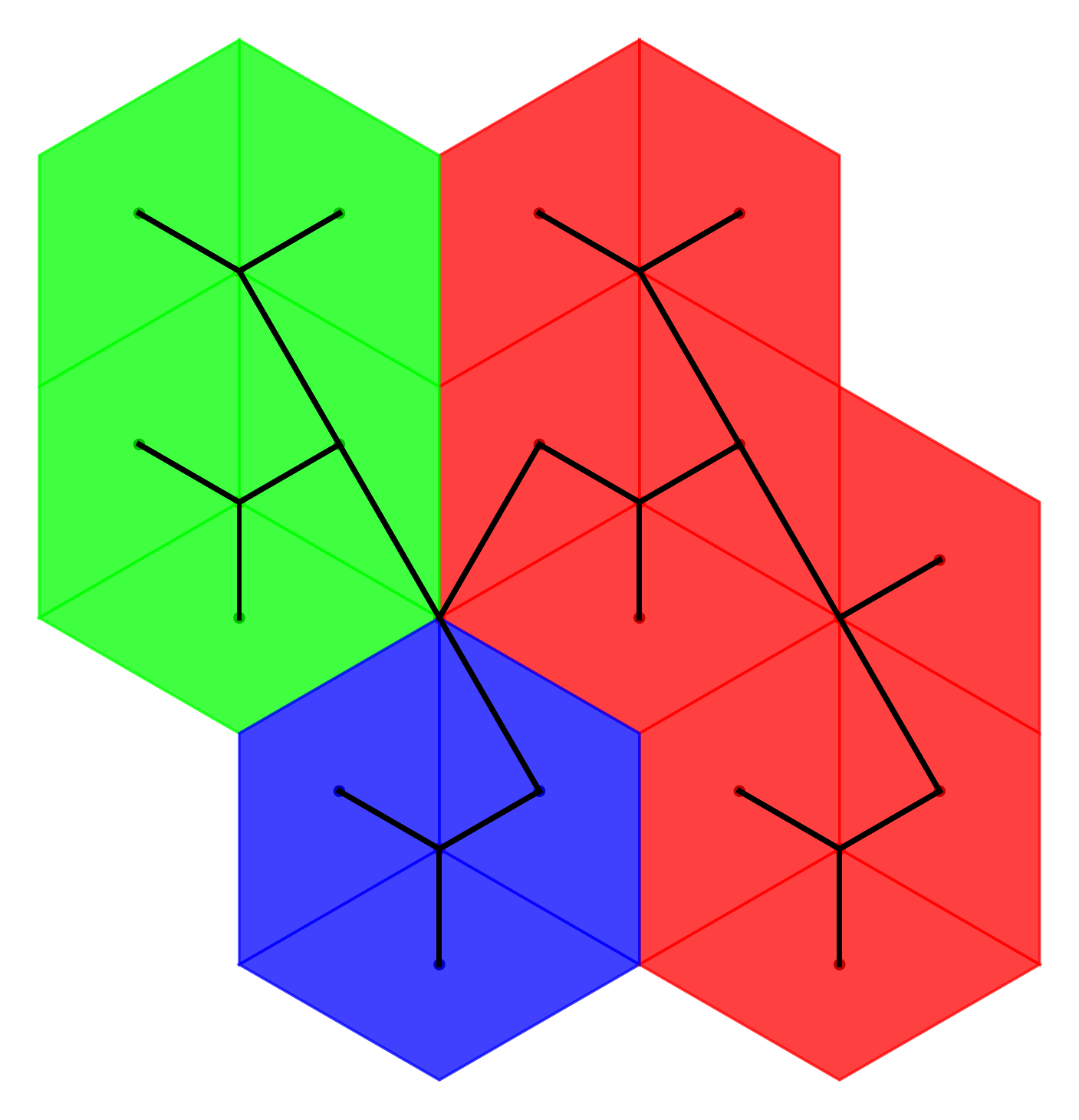}
\qquad 
\includegraphics[scale=.38]{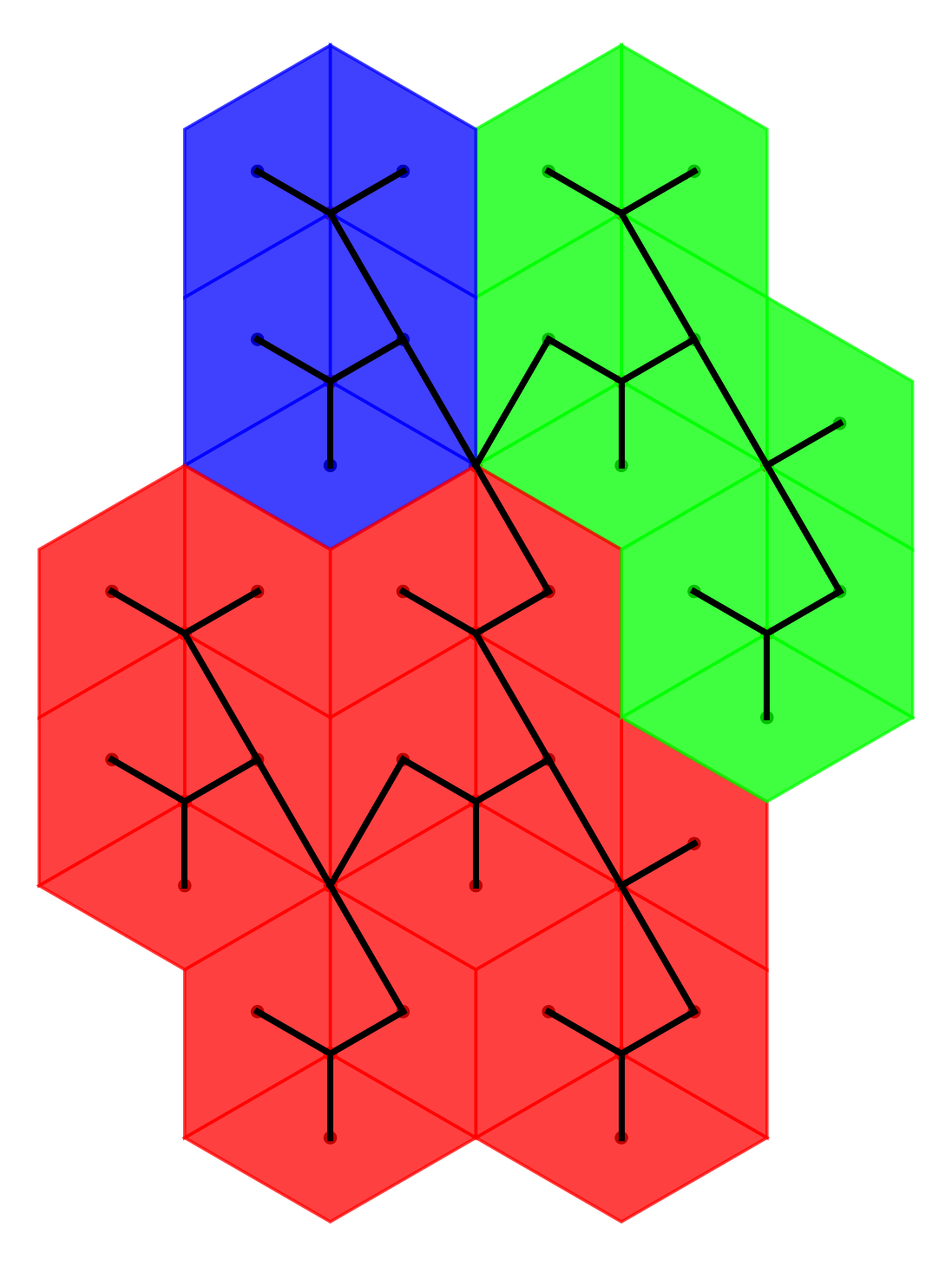}
\caption{Iterates of the pruned Tribonacci tree substitution.\label{fig:prunedtribo-iterated}}
\end{figure}

Turning back to Section~\ref{sec:dual-substitution}, we remark that in
the case of the Tribonacci substitution the usual prototiles
$P_a=\pi_c((0,a)^*)\subseteq E_c$ yield tiles $P_\gamma$ that are
disk-like and satisfy Condition~\ref{item:adjacency-quest-properties-dual-substitution} of
Question~\ref{quest:properties-dual-substitution}. The trees in
Figure~\ref{fig:prunedtribo-iterated} can be viewed as connecting the
tiles of the iterated images of the dual substitution.
\end{example}

Unfortunately pruning is not always enough: in our third example
detailed in Section~\ref{sec:second-example} even the pruned tree
substitution inside the contracting space is not injective and creates
loops.  

Let us describe our problem and our general aim. We will see pruning as a special case of covering.

Let $\sigma$ be an irreducible Pisot substitution over the alphabet
$A$ which is a parageometric iwip automorphism. A \textbf{covering} of
the tree substitution $\tau$ is a forgetful map $f:G\to \PF(\sigma)$
where $G$ is an oriented graph and $\PF(\sigma)$ is the prefix-suffix
automaton of $\sigma$ such that for each vertex $a$ of $G$, the map
$f$ induces a bijection between the edges of $G$ ending in $a$ and the
edges of $\PF(\sigma)$ ending in $f(a)$. We denote by $\tilde A$ the
set of vertices of $G$.

For each $\tilde a\in\tilde A$, let $P_{\tilde a}\subseteq E_c$ be a
prototile in the contracting space $E_c$. Pick a section $s:
A\to\tilde A$ of the forgetful map to select a starting patch:
\[
P_0=\bigcup_{a\in A}P_{s(a)}.
\]
Recall that $\ell:F_A\to\Z^A$ is the abelianization map, the
incidence matrix $M_\sigma$ restricted to the contracting
space $E_c$ is a contraction, and $\pi_c$ is the projection onto $E_c$ along the expanding direction.

We consider the iterated patches
\[P_n=\bigcup_{\gamma} P_{\tilde a_n}+\pi_c(M_\sigma^{-n}\ell(p)) \]
where the union is taken over all reverse paths of length $n$ in $G$
ending at the image of the section $s$:
\[
\gamma=s(a_0)\stackrel{e_1}{\longleftarrow}\tilde
a_1\cdots\stackrel{e_{n-1}}{\longleftarrow}\tilde a_n,
\]
and $p$ is the prefix of the image $f(\gamma)$ in the prefix-suffix
automaton of $\sigma$.

\begin{quest}
For any irreducible Pisot substitution which is an iwip automorphism, do there always exist a covering substitution and a choice of prototiles such that all the iterated patches are trees inside the contracting space?
\end{quest}

In Example~\ref{ex:tribo_pruned} the covering tree substitution was
defined by pruning to give a positive answer to this question. The
covering graph $G$ and the forgetful map $f$ can be recovered
from Figure~\ref{fig:tribo-pruned-thierry}.

\section{Interval exchange on the circle}

In this Section we show how a tree substitution induces an interval
exchange on the circle $\mathbb{S}^1$. Our construction further
extends the work of X.~Bressaud and Y.~Jullian~\cite{BJ12}. Indeed
they interpreted the celebrated construction of P.~Arnoux and
J.-C.~Yoccoz~\cite{arnoux-yoccoz} of an interval exchange
transformation related to the Tribonacci substitution as being the
contour of the compact heart $\Omega_A$ of the repelling tree
$\Trepul$.

\subsection{Contour of the tree substitution}\label{sec:contour}

Again we fix a substitution $\sigma$ that is a parageometric iwip
 automorphism of the free group $F_A$. We consider the tree
substitution $\tau$ defined in
Section~\ref{sec:combinatorial-tree-substitution-parageometric}.  At
each branch point of the finite simplicial tree $W$ we fix a cyclic
order of the outgoing edges. With respect to these cyclic orders
we get a way to turn around the tree, that is to say a map
$c:\mathbb S^1\to W$ which is locally an isometry except at leaves of
$W$ and which is $2$-to-$1$ except at leaves and branch points. Recall
that $V_a$ is the set of gluing vertices of each $W_a$ and that the set
of gluing points contains all leaves of $W_a$. We consider that the
circle is divided in arcs by these gluing points: each arc is labeled
by the letter $a_{ij}$ if it wraps around $W_a$ and is bounded by the gluing points $i$ and $j$. We denote by $\tilde A$ the set of labels $a_{ij}$ of these arcs.

We will consider $c$ alternatively as the cyclic word in $\tilde A$
labeling these arcs.

To define the contour substitution, we first require that:
\begin{enumerate}
  \renewcommand{\theenumi}{(C\arabic{enumi})}
  \renewcommand{\labelenumi}{\theenumi}
  \item\label{condition:simplicial-isomorphism} for each
$a\in A$, the map $\tau$ induces a simplicial
isomorphism between $W_a$ and the convex hull of the points
$\tau(V_a)$ in the simplicial tree $\tau(W_a)$. 
\end{enumerate}
This condition is satisfied as soon as $W_a$ has the same simplicial
structure as $Y_a$ which can be achieved using
Remark~\ref{rem:simplicial-struct-prototiles}.

Furthermore, we require that:
\begin{enumerate}
  \renewcommand{\theenumi}{(C\arabic{enumi})}
  \renewcommand{\labelenumi}{\theenumi}
  \setcounter{enumi}{1}
  \item\label{condition:equal-valence} for each $a\in A$, and each gluing point $x\in V_a$ the
    valence of $x$ in $W_a$ is equal to the valence of $x$ in
    $\tau(W_a)$.
\end{enumerate}
If $\tau$ satisfies Condition~\ref{condition:simplicial-isomorphism},
the valence of $\tau(x)$ in $\tau(W_a)$ is at least the valence of $x$
in $W_a$. If it is stricly bigger, there exists an edge
$a\stackrel{p,s}{\longleftarrow}b$ in the prefix-suffix automaton and
a gluing point $y\in V_b$ such that $y$ is in a direction $d$ outgoing from
$\tau(x)$ in $\tau(W_a)$ and such that $d$ does not intersect the tree
spanned by $\tau(V_a)$. Let $\gamma$ be the prefix-suffix expansion of
$y$ ending at the letter $b$, and let
$\gamma'=(a\stackrel{p,s}{\longleftarrow})\cdot\gamma$. Using
Remark~\ref{rem:comb-tree-subst-adding-gluing-vertices}, we extend
$\tau$ by adding the point associated with $\gamma'$ in the set of
gluing points $V_a$. Repeating this operation, we will get a tree
substitution satisfying
Conditions~\ref{condition:simplicial-isomorphism} and
\ref{condition:equal-valence}.

We now fix a cyclic order at each branch point of the image tree
$\tau(W)$. For a tree substitution $\tau$ satisfying Conditions~\ref{condition:simplicial-isomorphism} and
\ref{condition:equal-valence}, the cyclic orders of $W$ and $\tau(W)$ are \defi{compatible} if moreover:
\begin{enumerate}
  \renewcommand{\theenumi}{(C\arabic{enumi})}
  \renewcommand{\labelenumi}{\theenumi}
  \setcounter{enumi}{2}
\item\label{condition:tile-cyclic-order}the cyclic orders of each tile
  $W_a$ of $\tau(W)$ are the same as the cyclic orders of the initial
  prototile $W_a$.
\item\label{condition:compatible-cyclic-orders}for any gluing points
  $x,y,z$ of $W$, the cyclic order of $x,y,z$ and that of
  $\tau(x),\tau(y),\tau(z)$ are the same.
\end{enumerate}
By convention we say that two triples of points of two
trees have the same cyclic order if they are either aligned in the
same order or if the cyclic order of the branch points of the centers
of the two tripods are the same. In particular 
Condition~\ref{condition:compatible-cyclic-orders} implies Condition~\ref{condition:simplicial-isomorphism}.

From the above paragraph, we get a contour map $c':\mathbb
S^1\to\tau(W)$ where the circle is divided by the gluing points of
$\tau(W)$. Using Condition~\ref{condition:compatible-cyclic-orders},
the contour $c'$ is divided by the gluing points in $\bigcup_{i\in A}\tau(V_i)$, exactly as the contour $c$ was divided by the gluing points in $\bigcup_{i\in A} V_i$. Thus the contour $c'$ of
$\tau(W)$ is divided into arcs labeled by $\tau(a_{ij})$, for $a_{ij}\in\tilde A$. Each of these arcs of $c'$ is further divided by
arcs between gluing vertices of the tiles $W_b$ of $\tau(W)$
and, Condition~\ref{condition:tile-cyclic-order} identifies these
subarcs with letters in $\tilde A$.

Thus, we get a \defi{contour substitution} $\chi:\tilde A\to\tilde A^*$.

\begin{example}\label{ex:contour-tribo} 
  Consider the Tribonacci tree substitution of
  Figure~\ref{fig:tribo-tree-substitution}. We fix the cyclic
  orders of $W$ and $\tau(W)$ as being given by the embedding into
  the contracting plane (for which we fix the usual clockwise
  orientation). We get that the contour $c$ of $W$ is divided in arcs
  $\tilde A=\{a_{31},a_{12},a_{23},c_{11},b_{21},b_{12}\}$ and from
  the contour $c'$ of $\tau(W)$ we get the contour substitution
\[
\chi:\begin{array}[t]{rcl}
a_{12} &\mapsto & a_{23}c_{11}b_{21} \\
a_{23} &\mapsto & b_{12}a_{31} \\
a_{31} &\mapsto & a_{12} \\
b_{12} &\mapsto & a_{12}a_{23} \\
b_{21} &\mapsto & a_{31} \\
c_{11} &\mapsto & b_{21}b_{12}
\end{array}
\]
(Note that we slightly cheated from our construction to get over the
fact that the patch $\Omega_c$ has only one singular point. We
consider the tree $W_c$ not as a single vertex but as a vertex with a
tiny outgrow, and thus allow an arc $c_{11}$ turning around this
outgrow clockwise).

Note that, this substitution $\chi$ has reciprocal characteristic polynomial $x^6 - x^4
- 4x^3 - x^2 +1$ which factors as $(x^3-x^2-x-1)(x^3+x^2+x-1)$, where
the leftmost factor is the Tribonacci polynomial.
\end{example}

From our definitions we get a forgetful map $f:\tilde A\to A$ which
extends to a map $f:\tilde A^*\to A^*$. We remark that each arc
$b_{k\ell}\in\tilde A$ in the contour $c'$ of $\tau(W)$ is turning
around a patch $W_b$ which comes from an edge
$a\stackrel{p,s}{\longleftarrow}b$ of the prefix-suffix automaton of
the original substitution $\sigma$. The arc $b_{k\ell}$ is covered by
the image of exactly one arc $a_{ij}$ of the contour of $W_a$:
$\chi(a_{ij})=\tilde q\cdot b_{k\ell}\cdot \tilde t$.  We consider
the \textbf{dual contour substitution} $\chi^*:\tilde A\to\tilde A^*$ such
that the letter at position $|p|$ of the word $\chi^*(b_{k\ell})$ is
$a_{ij}$. The substitutions $\chi$ and $\chi^*$ come with a
bijection that maps the edge $a_{ij}\stackrel{\tilde p,\tilde
  s}{\longleftarrow}b_{k\ell}$ (with $f(\tilde p)=p$ and $f(\tilde
s)=s$) of the prefix-suffix automaton of $\chi^*$ to the edge
$a_{ij}\stackrel{\tilde q,\tilde t}{\longrightarrow}b_{k\ell}$ of the
prefix-suffix automaton of $\chi$.

Using the forgetful map $f$, we get that 
\[
f(\chi^*(b_{k\ell}))=\sigma(f(b_{k\ell}))=\sigma(b)
\]
which amount to saying that $\chi^*$ is a \defi{covering} of
$\sigma$.  Turning to matrices, the duality is expressed by
\[
^tM_{\chi}=M_{\chi^*}.
\]

\begin{example}\label{ex:tribo-dual-contour}
  For the Tribonacci tree substitution of
  Example~\ref{ex:tribo-tree-substitution} and the contour of
  Example~\ref{ex:contour-tribo}, the dual contour substitution is
\[
\chi^*:\begin{array}[t]{rcl}
a_{12} &\mapsto &a_{31}b_{12} \\
a_{23} &\mapsto &a_{12}b_{12} \\
a_{31} &\mapsto &a_{23}b_{21} \\
b_{12} &\mapsto &a_{23}c_{11} \\
b_{21} &\mapsto &a_{12}c_{11} \\
c_{11} &\mapsto &a_{12}
\end{array}
\]
We can compute the indexes of this free group automorphism and of its
inverse. Both are equal to $10=2\times 6-2$ and maximal. This
substitution is an iwip geometric automorphism: it is induced by a
pseudo-Anosov transformation of a surface of genus $3$ with one
boundary component.
\end{example}

As usual we are interested in iterating the substitution. Thus, we state
\begin{prop}\label{prop:contour-iterate}
  Let $\tau$ be a tree substitution and assume that $W$ and $\tau(W)$
  are equipped with compatible cyclic orders. Let $c$ be the
  corresponding contour of $W$ and $\chi$ be the contour tree
  substitution. Then, for each $n>0$, $\chi^n(c)$ is a contour of
  $\tau^n(W)$.
\end{prop}
\begin{proof}
  Iterating the tree substitution $\tau$, we consider the set $V_n$ of
  gluing points in $\tau^n(W)$.  We first remark that, by definition of
  a tree substitution, if
  Condition~\ref{condition:simplicial-isomorphism} holds, then $\tau$
  induces a simplicial isomorphism between $\tau^{n-1}(W)$ and the
  convex hull of the points $\tau(V_{n-1})$ in $\tau^n(W)$. Similarly,
  if Condition~\ref{condition:equal-valence} holds, then the valence
  of a gluing point $x\in V_{n-1}$ in $\tau^{n-1}(W)$ is equal to the
  valence of $\tau(x)$ in $\tau^n(W)$.
  
  Using the simplicial isomorphism $\tau$ between $\tau^{n-1}(W)$ and
  the convex hull of the images of the gluing points $\tau(V_{n-1})$
  in $\tau^n(W)$, we define by induction the cyclic orders on
  $\tau^n(W)$ by requiring that the cyclic order of each triple of
  gluing points $x,y,z$ of $\tau^{n-1}(W)$ is the same as the cyclic
  order of the gluing points $\tau(x),\tau(y),\tau(z)$ of
  $\tau^n(W)$. Condition~\ref{condition:equal-valence} implies that
  this completely defines the cyclic order at all points in
  $\tau(V_{n-1})$. Let now $x$ be a branch point of $\tau^n(W)$ which
  is not in $\tau(V_{n-1})$. Then, there exists a tile $W_a$ of
  $\tau^{n-1}(W)$ such that $x$ and all the outgoing edges from $x$
  belong to the patch $\tau(W_a)$. We define the cyclic order at $x$
  in $\tau^n(W)$ as being the cyclic order at $x$ in the patch
  $\tau(W_a)$ of $\tau(W)$. Note that if $x$ is also in the convex
  hull of $\tau(V_{n-1})$,
  Condition~\ref{condition:compatible-cyclic-orders} implies that the
  two cyclic orders that we used are compatible.

  The cyclic orders defined on $\tau^n(W)$ give contour maps
  $c_n:\mathbb S^1\to\tau^n(W)$.

  Two consecutive gluing points $x,y$ of $\tau^{n-1}(W)$ belong to the
  same tile $W_a$ and we fixed the cyclic order of $\tau^n(W)$ such
  that the gluing points in $\tau(W_a)\subseteq\tau^n(W)$ are ordered
  between $\tau(x)$ and $\tau(y)$ as they are in $\tau(W)$. This
  proves that the contour $c_n$ is obtain by applying the contour
  substitution $\chi$ to $c_{n-1}$. This concludes the proof by
  induction.
\end{proof}

For the next Proposition we will assume that the number of extremal
tiles in $\tau^n(W)$, i.e.~tiles $W_i$ of $\tau^n(W)$ which are
adjacent to exactly one other tile, is unbounded. We will say in this
case that $\tau^n(W)$ is \defi{sufficiently branching}. We have in mind that
the limit tree $\Omega_A$, which exists if $\sigma$ is an iwip
parageometric automorphism, is not the convex hull of finitely many
points.
  
\begin{prop}\label{prop:Mdeta-primitive}
  Let $\sigma$ be a primitive substitution and let $\tau$ be a tree
  substitution associated to $\sigma$. Assume that $W$ and $\tau(W)$
  are equipped with compatible cyclic orders and that $\tau^n(W)$
  is sufficiently branching.
  Then, the contour substitution $\chi$ is primitive and has the same
  dominant eigenvalue $\lambda_\sigma$ as $\sigma$.
\end{prop}
\begin{proof}
  
  The growth rate of any letter $a$
  in $\tilde A$ under iteration of the dual contour substitution $\chi^*$ 
  is the same as that of $f(a)$ under $\sigma$, where $f$ is the   	
  forgetful map. This proves that there
  exists a positive eigenvector $\vec\ell$ such that
  $M_{\chi^*}\vec\ell=\lambda_\sigma\vec\ell$.

  As $\sigma$ is primitive, for some $n\geq 1$, $\tau^n(W)$ contains a copy of all tiles
  $W_a$ for all $a\in A$. Such a copy is wrapped around by all
  $a_{ij}$ in $\tilde A$ with $f(a_{ij})=a$. Thus all letters of
  $\tilde A$ appear in the image of the contour substitution $\chi$
  and, by duality, all letters of $\tilde A$ appear in the image of
  the dual contour substitution $\chi^*$.

  This proves that the matrix of $\chi^*$ is block-wise
  diagonalizable, each block contains at least one letter of
  $f\inv(a)$ for each $a\in A$ and, has dominant eigenvalue
  $\lambda_\sigma$. Moreover, up to passing to a positive power, each
  diagonal block is positive.

  From our assumption, the number of extremal tiles in $\tau^n(W)$
  goes to infinity.  As the circle is originally divided into finitely
  many arcs, there exists an arc $a_{ij}\in\tilde A$ such that
  $\chi^n(a_{ij})$ wraps around a whole extremal patch $W_b$: all
  letters $b_{kl}\in f\inv (b)$ appear in $\chi^n(a_{ij})$.

  By duality, one the blocks of the block-wise diagonal matrix
  $M_{\chi^*}^n$ spans all of $f\inv(b)$. We proved that
  $M_{\chi^*}$ has a unique diagonal block, hence $\chi^*$ is
  primitive. And by duality $\chi$ is primitive as well.
\end{proof}

From Proposition~\ref{prop:Mdeta-primitive} we get that the matrix
$M_\chi$ has a unique (projective) positive left eigenvector
$\vec\ell_{\mathbb S^1}$ associated with the dominant eigenvalue $\lambda_\sigma$.

We use this eigenvector $\vec\ell_{\mathbb S^1}$ to describe the points of the unit
circle $\mathbb S^1$ by infinite paths in the prefix-suffix
automaton. First divide the unit circle $\mathbb S^1$ by the arcs of
$\tilde A$ each with the length given by $\vec\ell_{\mathbb S^1}$.  Let
$\upor{\mathcal P}$ be the set of infinite paths
$\gamma=a_0\stackrel{p_1,s_1}{\longrightarrow}a_1\stackrel{p_2,s_2}{\longrightarrow}\ldots$,
with $a_i\in\tilde A$, $\chi(a_i)=p_{i+1}\cdot a_{i+1}\cdot
s_{i+1}$. We insist that here the infinite path in the prefix-suffix
automaton is in the positive direction of edges which is reverse to
all the paths (or prefix-suffix expansions) we considered up to
here. This is why we use the funny notation $\upor{\mathcal
  P}$. Consider the map $\QSone:\upor{\mathcal P}\to\mathbb S^1$ such that
\[
\QSone(\gamma)=x_0+\sum_{n=1}^{+\infty}\frac 1{\lambda_\sigma^n}\ell_{\mathbb S^1}(p_{n}),
\]
where $x_0$ is the left endpoint of the arc $a_0$ and $\ell_{\mathbb S^1}(p_n)$ is the scalar product between the abelianization of the word $p_n$ and the eigenvector $\vec\ell_{\mathbb S^1}$. The map $\QSone$ is obviously continuous.

For any finite path
$\gamma=a_0\stackrel{p_1,s_1}{\longrightarrow}a_1\stackrel{p_2,s_2}{\longrightarrow}\ldots
\stackrel{p_n,s_n}{\longrightarrow}a_n$, the cylinder $[\gamma]$ of all infinite paths
starting with $\gamma$ is mapped by $\QSone$ to an arc $\QSone([\gamma])$ of
$\mathbb S^1$ of length $(\frac 1{\lambda_\sigma})^n\ell(a_n)$.

The bijection between the edges of the prefix-suffix automata of
$\chi$ and $\chi^*$ yields a homeomorphism
$\upor{\mathcal P}_\chi\simeq{\mathcal P}_{\chi^*}$ and by composition a
continuous map ${\mathcal P}_{\chi^*}\to\mathbb S^1$ which, by abuse of notation, we also denote by $\QSone$.

\subsection{Piecewise exchange and self-similarity on the
  circle}\label{sec:vershik-map}

The contour we constructed in the previous Section carries a self-similar
piecewise exchange exactly as the attracting shift $X_\sigma$ and
the limit set $\Omega_A$ do. In this Section we describe this
self-similar piecewise exchange on the circle.

Our description of the circle comes from the map $\QSone:\mathcal
P_{\chi^*}\to\mathbb S^1$. 
The piecewise exchange comes primarily from the shift map $S$ on the
attracting shift $X_\sigma$, which is carried to the Vershik map on the space $\mathcal{P}_{\chi^*}$ of infinite prefix-suffix expansions of $\chi^*$.  
We will obtain the piecewise exchange on the circle by pushing forward the Vershik map through $\QSone$. 

Let $\sigma$ be a substitution. 
Recall that in Section~\ref{sec:subst-attr-shift} we introduced the notation $p(\gamma)$ for the prefix of a prefix-suffix expansion $\gamma$. Here we rather deal with
suffixes and we use the similar notation $s(\gamma)=s_0\sigma(s_1)\cdots\sigma^{n-1}(s_{n-1})$.   
We denote by $\Pmax$ the set of infinite prefix-suffix expansions with empty suffix (and by $\Pmin$ those with empty prefix).

The \textbf{Vershik map}~\cite[Section~6.3.3]{durand-cant6}
\[
V:\mathcal P\ssm\Pmax\to\mathcal P\ssm\Pmin
\]
sends a prefix-suffix expansion
\[
  \gamma=a_0\stackrel{p_{0},\epsilon}{\longleftarrow}a_{1}\stackrel{p_{1},\epsilon}{\longleftarrow}
  \cdots a_r\stackrel{p_{r},s_{r}}{\longleftarrow}a_{r+1}\cdots
  a_{n-1}\stackrel{p_{n-1},s_{n-1}}{\longleftarrow}a_n\cdots
\]
where $r$ is the smallest index such that $s_r=b_r t_r\neq\epsilon$, with $b_r\in A$ and $t_r\in A^*$, to
\[
  V(\gamma)=b_0\stackrel{\epsilon,t_{0}}{\longleftarrow}b_{1}\stackrel{\epsilon,t_{1}}{\longleftarrow}
  \cdots b_r\stackrel{p_{r}a_r,t_{r}}{\longleftarrow}a_{r+1}\cdots
  a_{n-1}\stackrel{p_{n-1},s_{n-1}}{\longleftarrow}a_n\cdots,
\]
with $\sigma(a_{r+1})=p_r a_r s_r=p_r a_r b_r
t_r$ and, for $i=1,\ldots,r-1$, $\sigma(b_{i+1})=b_it_i$.
As it is defined using only the head of a prefix-suffix expansion
path, the Vershik map is continuous.

In the setting of Section~\ref{sec:global-picture}, the Vershik map is
simply the push-forward of the shift-map $S$ through the map $\Gamma$
which associates to any bi-infinite word in the attracting shift
$X_\sigma$ of $\sigma$ its prefix-suffix expansion:
\[
\forall Z\in X_\sigma,\ \Gamma(Z)\not\in\Pmax\Rightarrow V(\Gamma(Z))=\Gamma(SZ),
\]
and pushing forward to the limit set $\Omega_A$ of the repelling tree
$\Trepul$ we get:
\begin{equation}\label{eq:ver1}
\forall \gamma\in\mathcal P\ssm\Pmax,\ Q(V(\gamma))=a_0\inv Q(\gamma),
\end{equation}
where $a_0$ is the ending letter of the infinite prefix-suffix
expansion $\gamma$. And, finally, for a finite path $\gamma$ in the
prefix-suffix automaton ending in $a_0$ with non-empty suffix
\begin{equation}\label{eq:ver2}
a_0\inv\Omega_\gamma=\Omega_{V(\gamma)}.
\end{equation}
We can partially push-forward the Vershik map to the iterates of the
abstract tree substitution. 
A gluing point $x$ in $\tau^n(W)$ is
described as a finite path $\gamma$ of length $n$ in the prefix-suffix
automaton, together with a gluing point $y\in V_{a_n}$, where $a_n$ is
the starting letter of $\gamma$. If $\gamma$ has non-empty suffix then
$\gamma$ and $V(\gamma)$ both start at the same letter $a_n$. 
For sake of simplicity, we abuse again of notation and we let the Vershik map act on gluing points, instead of writing $V$ on their prefix-suffix expansions. Then, 
the image $V(x)$ is described by the same gluing point $y$ of the tile
$W_{a_n}$ associated to the path $V(\gamma)$ in $\tau^n(W)$.

To push-forward the Vershik map to the contour circle, we need that the cyclic orders of $W$
and $\tau(W)$ are compatible with the translations defined by \eqref{eq:ver1}-\eqref{eq:ver2}. Thus we
strengthen the compatibility condition of the previous section.

The cyclic orders of $W$ and $\tau(W)$ are \textbf{strongly
  compatible} if, in addition to
Conditions~\ref{condition:simplicial-isomorphism},
\ref{condition:equal-valence}, \ref{condition:tile-cyclic-order} and
\ref{condition:compatible-cyclic-orders}, the following holds:
\begin{enumerate}
  \renewcommand{\theenumi}{(C\arabic{enumi})}
  \renewcommand{\labelenumi}{\theenumi}
  \setcounter{enumi}{4}
\item\label{condition:Vershik-compatible} For $a\in A$ let $x$, $y$,
  $z$ be gluing points of $\tau(W_a)\subseteq\tau(W)$ where the map
  $V$ is defined. Then, the cyclic orders of $(x,y,z)$ and
  $(V(x),V(y),V(z))$ are equal in $\tau(W)$.
\end{enumerate}
We remark that using the cyclic order defined in the proof of
Proposition~\ref{prop:contour-iterate},
Condition~\ref{condition:Vershik-compatible} propagates to gluing points
in $\tau^n(W)$: for any gluing points $x$, $y$ and $z$ in $\tau^n(W_a)$
where the map $V$ is defined, the cyclic orders of $(x,y,z)$ and
$(V(x),V(y),V(z))$ are equal in $\tau^n(W)$.

\begin{rem}\label{rem:order_from_Trepul}
  Conditions~\ref{condition:simplicial-isomorphism}~-~\ref{condition:Vershik-compatible}
  are finite combinatorial conditions that are easily checked.

  We already explained how to satisfy
  Conditions~\ref{condition:simplicial-isomorphism} and
  \ref{condition:equal-valence}. We now prove that cyclic orders
  satisfying
  Conditions~\ref{condition:tile-cyclic-order}~-~\ref{condition:Vershik-compatible}
  exist.

  With Condition \ref{condition:simplicial-isomorphism},
  $\tau^n(W)$ is isomorphic to the subtree $\sigma_T^n(Y)$ of
  $\Omega_A$, for each $n$. Recall that the free group $F_A$ acts freely 
  on the tree
  $\Trepul$ and, that there are finitely many orbits of branch points,
  each with finite valence.  The contracting homothety $H$ permutes
  the finitely many directions at the finitely many orbits of branch
  points. Up to replacing $H$ by a suitable power (as well as the
  initial substitution $\sigma$), any cyclic orders on directions at
  orbits of branch points extend to cyclic orders at branch points of
  $\Trepul$ which are preserved by the action of $F_A$ and
  $H$. Finally, through the isomorphisms with substrees of $\Trepul$
  we get cyclic orders on $W$ and $\tau(W)$.

  As the expanding homothety $H\inv$ corresponds to the map $\tau$ and
  as the translation by $a\inv$ for $a\in A$ corresponds to the
  Vershik map, we conclude that this choice of cyclic orders satisfies
  Conditions~\ref{condition:tile-cyclic-order}~-~\ref{condition:Vershik-compatible}.
\end{rem}

\begin{prop}\label{prop:vershik-on-circle}
  Let $\sigma$ be a primitive substitution and a parageometric
  iwip automorphism. Let $\tau$ be a tree substitution associated to
  $\sigma$. Assume that $W$ and $\tau(W)$ are equipped with strongly
  compatible cyclic orders.

  Then, the Vershik map is pushed forward by $\QSone$ to a piecewise rotation
  of the unit circle $\mathbb S^1$.
\end{prop}
\begin{proof}
  We use the notation of the previous Section. For any finite path $\gamma$
  in the prefix-suffix automaton of $\chi^*$, the image
  $\QSone([\gamma])$ of the cylinder $[\gamma]$ is an arc of the unit
  circle of length $(\frac 1{\lambda_\sigma})^n\ell_{\mathbb S^1}(a_n)$. If we
  assume that $\gamma$ as non-empty suffix, then $V(\gamma)$ has the
  same length $n$ as $\gamma$ and the same starting letter $a_n$. Thus
  we get that $V$ is pushed-forward by $\QSone$ to a rotation from the arc
  $\QSone([\gamma])$ to the arc $\QSone([V(\gamma)])$.

  Remark that the above paragraph is abusive as the map $\QSone$ is
  non-injective and thus the map $V$ cannot be pushed-forward to
  $\mathbb S^1$. However for us, a piecewise rotation is not a map
  $\mathbb S^1\to\mathbb S^1$ but rather a division of the circle into
  finitely many arcs and on each of these arcs a rotation. In
  particular, boundary points of the arcs can be rotated to two
  different points according to whether they are considered as
  boundary points of one or the other of the two arcs they belong to.

  The division of the circle into arcs with non-empty suffixes is not
  finite. In order to get a piecewise rotation of the circle, we are
  left with proving that all but finitely many adjacent such arcs are rotated
  by the same angle. This is the purpose of the next lemma. 

\begin{lem}\label{lem:vershik-adjacent-cylinders}
  Let $\gamma$ and $\gamma'$ be two prefix-suffix expansion paths of
  $\chi^*$ with same length, non-empty suffixes and ending at the
  same letter such that the associated arcs $\QSone([\gamma])$ and
  $\QSone([\gamma'])$ are adjacent on $\mathbb S^1$.  Then the arcs
  $\QSone([V(\gamma)])$ and $\QSone([V(\gamma')])$ are adjacent on the circle
  (and in the same order).
\end{lem}
\begin{proof}
  Let $n$ be the length of $\gamma$ and $\gamma'$. From
  Proposition~\ref{prop:contour-iterate}, $\chi^n(c)$ is the contour
  of $\tau^n(W)$. The forgetful map $f$ maps a prefix-suffix expansion
  of $\chi^*$ to a prefix-suffix expansion of $\sigma$.  Thus, the
  tiles $W_{f(\gamma)}$ and $W_{f(\gamma')}$ are equal or adjacent in
  $\tau^n(W)$. If they are equal, then $f(\gamma)=f(\gamma')$ and thus
  $V(f(\gamma))=V(f(\gamma'))$. Else, in the repelling tree the tiles
  $\Omega_{f(\gamma)}$ and $\Omega_{f(\gamma')}$ have a singular point
  $P$ in common. Let $a_0\in A$ be the common ending letter of both
  $f(\gamma)$ and $f(\gamma')$. Then
  $a_0^{-1}\Omega_{f(\gamma)}=\Omega_{V(f(\gamma))}$ and
  $a_0^{-1}\Omega_{f(\gamma')}=\Omega_{V(f(\gamma'))}$. These
  translated tiles have the singular point $a_0^{-1}P$ in common. The
  singular point $P$ gives rise to gluing points $y\in
  W_{f(\gamma)}$ and $y'\in W_{f(\gamma')}$ which are identified in
  $\tau^n(W)$. The Vershik map is defined on these two points and
  $V(y)$ and $V(y')$ are identified in $\tau^n(W)$: $W_{V(f(\gamma))}$
  and $W_{V(f(\gamma'))}$ have a point in common.

  Let us denote by $[x,y]=\QSone([\gamma])$ and
  $[y',z]=\QSone([\gamma'])$ the arcs of the contour of
  $\tau^n(W)$. The Vershik map is defined on these four points and the
  gluing points $y$ and $y'$ are identified in $\tau^n(W)$ as well as
  the gluing points $V(y)$ and $V(y')$. By
  Condition~\ref{condition:tile-cyclic-order}, $[V(x),V(y)]$ and
  $[V(y'),V(z)]$ are arcs of the contour of $\tau^n(W)$ with a common
  point. Assume by contradiction that, in the contour of $\tau^n(W)$,
  the arc $[V(x),V(y)]$ is followed by the arc $[y'',t]$ which is
  different from $[V(y'),V(z)]$. This arc turns around a tile
  $W_{\gamma''}$ and if $\gamma''$ is equal to $V(f(\gamma))$ or
  $V(f(\gamma'))$ then it is the image by $V$ of an arc of
  $W_{f(\gamma)}$ or $W_{f(\gamma')}$. Then
  Condition~\ref{condition:Vershik-compatible} implies that the arc
  $[x,y]$ is not followed by the arc $[y',z]$ in the contour of
  $\tau^n(W)$, which is a contradiction. We now assume that $\gamma''$ is
  different from both $f(\gamma)$ and $f(\gamma')$. Let $\alpha$ be a
  finite path of the length $m$ in the prefix-suffix automaton of $\sigma$, ending at
  the starting letter of $\gamma''$ with non-empty prefix. Then
  $\gamma''\cdot\alpha$ has non-empy prefix and is the image by $V$ of
  a path $\gamma'''$. Let $t'$ be a gluing point in the tile
  $W_{\gamma'''}$ of $\tau^{n+m}(W)$ such that $V(t')$ is distinct
  from $\tau^m(y'')$. From
  Condition~\ref{condition:Vershik-compatible}, the cyclic order of
  $(\tau^m(x),\tau^m(z),t')$ in $\tau^{n+m}(W)$ is equal to that of
  $(V(\tau^m(x)),V(\tau^m(z)),V(t'))$. In particular, $\tau^m(x)$,
  $\tau^m(z)$ and, $t'$ lie in three different directions outgoing
  from $y$. From Condition~\ref{condition:equal-valence}, there exist
  a point $t''$ in $\tau^n(W)$ such that $\tau^m(t'')$ lies in the
  same direction from $y$ as
  $t'$. Condition~\ref{condition:compatible-cyclic-orders} implies
  that the cyclic order of $(x,z,t'')$ in $\tau^n(W)$ is the same as
  that of $(\tau^m(x),\tau^m(z),\tau^m(t''))$ which we proved to be
  equal to that of $(V(x),V(z),t)$. This contradicts the fact that
  $[x,y]$ is followed by $[y',z]$ in the contour of $\tau^n(W)$.

  By contradiction we proved that the arc
  $[V(x),V(y)]=\QSone([V(\gamma)])$ is followed by the arc
  $[V(y'),z]=\QSone([\gamma'])$ in the contour of $\tau^n(W)$.
\end{proof}
The set $\Pmax$ of infinite prefix-suffix expansions with empty
suffixes is finite. We divide the circle $\mathbb S^1$ into finitely
many subarcs of the arcs $\QSone([a])$, for $a\in\tilde A$, bounded by 
points in $\QSone(\Pmax)$. This is a finite subdivision. According to
Lemma~\ref{lem:vershik-adjacent-cylinders}, on each of these arcs the
Vershik map $V$ is pushed-forward by $\QSone$ to a rotation.
This concludes the proof of Proposition~\ref{prop:vershik-on-circle}.
\end{proof}

\begin{example}\label{ex:piec-exch-circle-tribo}

  Consider the Tribonacci substitution $\sigma$, the tree substitution
  $\tau$ of Example~\ref{ex:tribo-tree-substitution}, the contour
  substitution $\chi$ of Example~\ref{ex:contour-tribo} and the dual
  contour substitution $\chi^*$ of
  Example~\ref{ex:tribo-dual-contour}. 

  The set $\Pmax$ of infinite prefix-suffix expansions with empty
  suffixes consists of the three infinite paths $\Gamma(S\inv(w_a))$,
  $\Gamma(S\inv(w_b))$ and, $\Gamma(S\inv(w_c))$ which appear in
  Example~\ref{ex:tribo-fix-points}:
  \begin{center}
    \begin{tikzpicture}[->,>=stealth',shorten >=1pt,auto,node
      distance=1.8cm, thick,main
      node/.style={circle,draw,font=\sffamily\bfseries}]

      \node[main node] (1) {\large $a$}; 
      \node[main node] (2) [above right of=1] {\large $b$}; 
      \node[main node] (3) [below right of=2] {\large $c$};
  
      \path[every node/.style={font=\sffamily\small}]
      (1) edge [] node [above left] {$a,\epsilon$} (2)    
      (2) edge [] node[above right] {$a,\epsilon$} (3)
      (3) edge [] node[below] {$\epsilon,\epsilon$} (1);

    \end{tikzpicture}
  \end{center}
  We now have two options: either we add these three points
  $Q(\Gamma(S\inv(w_a))$, $Q(\Gamma(S\inv(w_b))$ and,
  $Q(\Gamma(S\inv(w_c))$ as gluing vertices in the tree subsitution or
  alternatively, we add these points to the contour of the circle. As
  the prefix-suffix expansions are periodic, using
  Remark~\ref{rem:comb-tree-subst-adding-gluing-vertices} the first
  option can be achieved. However, in order to fix compatible cyclic
  orders this will require to further describe the branch points of
  $W_a$ which now has four gluing vertices.

  The second option requires to consider the infinite prefix-suffix
  expansions with empty suffixes $\widetilde{\mathcal
    P}_{\text{max}}=\Pmax(\chi^*)$ of the dual contour substitution
  $\chi^*$ rather than those of $\sigma$: the three infinite paths
  \begin{center}
    \begin{tikzpicture}[->,>=stealth',shorten >=1pt,auto,node
      distance=1.8cm, thick,main
      node/.style={circle,draw,font=\sffamily\bfseries}]

      \node[main node] (1) {\large $a_{12}$};
      \node[main node] (2) [above right of=1] {\large $b_{12}$};
      \node[main node] (3) [below right of=2] {\large $c_{11}$};
  
      \path[every node/.style={font=\sffamily\small}]
      (1) edge [] node [above left] {$a_{31},\epsilon$} (2)    
      (2) edge [] node[above right] {$a_{23},\epsilon$} (3)
      (3) edge [] node[below] {$\epsilon,\epsilon$} (1);
    \end{tikzpicture}
  \end{center}
  We thus introduce a new subdivision by the three points
  $\QSone(\widetilde{\mathcal P}_{\text{max}})$ of the arcs
  $a_{12}=a_{14}a_{42}$, $b_{12}=b_{13}b_{31}$ and
  $c_{11}=c_{12}c_{21}$. We get an extended  contour substitution
  \[
    \begin{array}[t]{rcl}
      a_{14} &\mapsto & a_{23}c_{12}\\
      a_{42} &\mapsto & c_{21}b_{21} \\
      a_{23} &\mapsto & b_{12}a_{31} \\
      a_{31} &\mapsto & a_{14}a_{42} \\
      b_{13} &\mapsto & a_{14}\\
      b_{32} &\mapsto & a_{42}a_{23} \\
      b_{21} &\mapsto & a_{31} \\
      c_{12} &\mapsto & b_{21}b_{13}\\
      c_{21} &\mapsto & b_{32}
    \end{array}
  \]
  From Lemma~\ref{lem:vershik-adjacent-cylinders}, we remark that the
  adjacent arcs $Q([a_{31}])$ and $Q([a_{14}])$ are rotated by the
  same angle as are the adjacent arcs $Q([a_{42}])$ and $Q([a_{23}])$
  and, the adjacent arcs $Q([b_{21}])$ and $Q([b_{13}])$. We thus get
  a division of the circle into six arcs.

  To describe the piecewise rotation on each of these six arcs, we
  give the image of some of the boundary points.

  The left endpoint of the arc $a_{31}$ comes from the singular
  point $P=Q(w'_a)$ of $\Omega_a$. We already computed (see
  Examples~\ref{ex:tribo-right-left-special} and
  \ref{ex:tribo-periodic-desubstitution}) that
  \[
    \Gamma(w'_a)=\gamma_a,\ V(\gamma_a)=V(\Gamma(w'_a))=\Gamma(S(w'_a)),\
    Q(w'_a)=P,
  \]
  \[
    Q(V(\Gamma(w'_a)))=Q(S(w'_a))=a\inv P,
  \]
  and $a\inv P$ is a singularity of $\Omega_b$.  Thus we get that the
  left endpoint of the arc $a_{31}$ is rotated to the left endpoint of the arc $b_{13}$.  The right endpoint of the arc
  $a_{23}$ comes from the same singularity $P$ of $\Omega_a$ and is
  also rotated to the same point of the circle.

  Similarly, we compute:
  \[
    \Gamma(w'_b)=\gamma_b,\ V(\gamma_b)=V(\Gamma(w'_b))=\Gamma(S(w'_b)),\
    Q(w'_b)=P,
  \]
  \[
    Q(V(\Gamma(w'_b)))=Q(S(w'_b))=b\inv P,
  \]
  \[
    \Gamma(w'_c)=\gamma_c,\ V(\gamma_c)=V(\Gamma(w'_c))=\Gamma(S(w'_c)),\
    Q(w'_c)=P,
  \] 
  \[
    Q(V(\Gamma(w'_c)))=Q(S(w'_c))=c\inv P.
  \]
  Thus the left endpoint of the arc $b_{21}$ and the right endpoint of the arc $b_{32}$ are both rotated (by different
  rotations) to the left endpoint of $a_{23}$.
  And, the left endpoint of $c_{12}$ and the right endpoint of
  $c_{21}$ are both rotated (by different rotations) to the point
  which is the left endpoint of $a_{14}$.

  The piecewise rotation of the circle for Tribonacci substitution is
  pictured in Figure~\ref{fig:tribo-circle}. Notice that this is the Arnoux-Yoccoz interval exchange on the six intervals $a_{31}a_{14}$, $a_{42}a_{23}$, $c_{12}$, $c_{21}$, $b_{21}b_{13}$ and $b_{32}$.

\begin{figure}[h]
\centering
\includegraphics[scale=1]{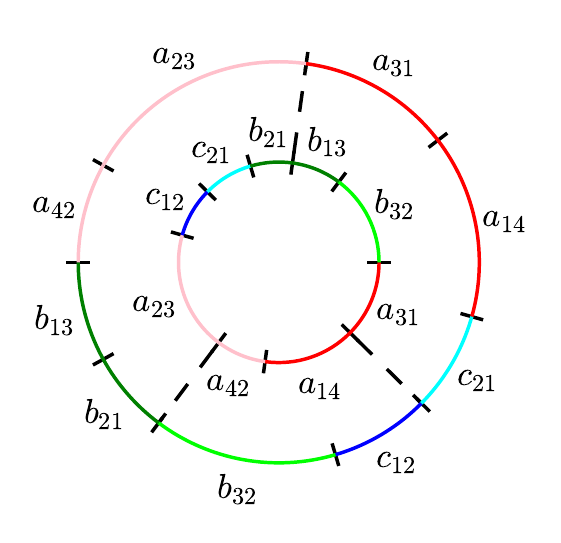}
\caption{\label{fig:tribo-circle}Interval exchange on the circle for the Tribonacci
  substitution: the innermost circle is the contour of $W$ while the
  outermost circle is the result of the piecewise exchange.}
\end{figure}

The suspension of this piecewise rotation of the circle with six
pieces is the surface of genus $3$ with one boundary component that we
referred to at the end of Example~\ref{ex:tribo-dual-contour}. 
\end{example}

\section{Examples}

We survey in this Section two other examples.

\subsection{First example}

We consider the substitution
\[
\sigma:\begin{array}[t]{rcl}a&\mapsto&ac\\b&\mapsto&ab\\c&\mapsto&b\end{array}
\]
which is an iwip automorphism of the free group $F_{\{a,b,c\}}$.  As
before we denote by $X_\sigma$ its attracting shift. The topological
representative of $\sigma$ on the rose with three petals is an
irreducible train-track and has three periodic Nielsen paths (of
period $4$):

\[
\begin{tikzpicture}
\draw (0,0) to [out=0,in=180] node[near end,above,sloped] {$a$} (.6,.3);
\draw (0,0) to [out=0,in=180] node[near end,below,sloped] {$b$} (.6,-.3);

\draw (1,0) to [out=0,in=180] node[near end,above,sloped] {$b$} ++(.6,.3);
\draw (1,0) to [out=0,in=180] node[near end,below,sloped] {$c$} ++(.6,-.3);

\draw (3,0) to [out=180,in=0] node[midway,above,sloped] {$a$} ++(-.6,.3);
\draw (3,0) to [out=180,in=0] node[midway,below,sloped] {$c$} ++(-.6,-.3);

\draw (4,0) to [out=180,in=0] node[midway,above,sloped] {$b$} ++(-.6,.3);
\draw (4,0) to [out=180,in=0] node[midway,below,sloped] {$ac$} ++(-.6,-.3);
\end{tikzpicture}
\]
Thus $X_\sigma$ contains four singular leaves:

\begin{tikzpicture}
\node (gauche) at (0,0) {$\cdots acbacabacbabac\cdot$}; 
\draw (gauche.east) to [out=0,in=180] ++(.6,.3) node[anchor=west] {$abacbabacbacab\cdots\quad : Z_1$};
\draw (gauche.east) to [out=0,in=180] ++(.6,-.3) node[anchor=west]{$bacabacbabacab\cdots\quad: Z_2$};
\end{tikzpicture}

\begin{tikzpicture}
\node (gauche) at (0,0) {$\cdots acbabacab acba\cdot$}; 
\draw (gauche.east) to [out=0,in=180] ++(.6,.3) node[anchor=west] {$bacbacab acbabacab\cdots\quad : Z_3$};
\draw (gauche.east) to [out=0,in=180] ++(.6,-.3) node[anchor=west]{$cab acbabacab\cdots\quad : Z_4$};
\end{tikzpicture}

\begin{tikzpicture}
\node[anchor=east] at (0,.3) {$Z_5:\quad\cdots acbabacbacab  acba$}; 
\node[anchor=east] at (0,-.3) {$Z_6:\quad\cdots  acbabacab ac$};
\draw (0,.3) to [out=0,in=180] ++(.6,-.3) node[anchor=west] {$\cdot bacbab acbabacbacab \cdots$};
\draw (0,-.3) to [out=0,in=180] ++(.6,.3);
\end{tikzpicture}

\begin{tikzpicture}
\node[anchor=east] at (0,.3) {$Z_7:\quad\cdots  acbabacbacab acb$}; 
\node[anchor=east] at (0,-.3) {$Z_8:\quad\cdots  acbabacbacab acbabac$};
\draw (0,.3) to [out=0,in=180] ++(.6,-.3) node[anchor=west] {$\cdot abacbacab acbabacab\cdots$};
\draw (0,-.3) to [out=0,in=180] ++(.6,.3);
\end{tikzpicture}

An easy calculation shows that the index is maximal, thus $\sigma$ is parageometric.

The map $Q$ identifies these left and right special infinite words, defining the points 
\begin{align*}
P_1=Q(Z_1)=Q(Z_2),\quad P_2=Q(Z_3)=Q(Z_4), \\
P_3=Q(Z_5)=Q(Z_6),\quad P_4=Q(Z_7)=Q(Z_8).
\end{align*}

The infinite desubstitution paths of the singular bi-infinite words are:
\begin{center}
\begin{tikzpicture}[->,>=stealth',shorten >=1pt,auto,node distance=1.4cm,thick,main node/.style={circle,draw,font=\sffamily\bfseries}]

  \node (0) {$\Gamma(Z_1)$};
  \node[main node] (1) [right of=0] {$a$};
  \node[main node] (2) [right of=1] {$b$};
  \node[main node] (3) [below of=2] {$b$};
  \node[main node] (4) [below of=1] {$c$};
  \node (5) [right of=2] {$\Gamma(Z_3)$};
  \node (6) [right of=3] {$\Gamma(Z_2)$};
  \node (7) [below of=0] {$\Gamma(Z_4)$};

  \path[every node/.style={font=\sffamily\small}]
    (1) edge (0)
    (2) edge (5)
    (3) edge (6)
    (4) edge (7)
    (2) edge node [above] {$\epsilon,b$} (1)   
    (3) edge node [right] {$a,\epsilon$} (2)
    (4) edge node [above] {$\epsilon,\epsilon$} (3)
    (1) edge node [left] {$a,\epsilon$} (4);
\end{tikzpicture}

\medskip

\begin{tikzpicture}[->,>=stealth',shorten >=1pt,auto,node distance=1.4cm,thick,main node/.style={circle,draw,font=\sffamily\bfseries}]

  \node (0) {$\Gamma(Z_8)$};
  \node[main node] (7) [above of=0] {$a$};
  \node[main node] (1) [left of=7] {$b$};
  \node[main node] (2) [left of=1] {$b$};
  \node[main node] (3) [below of=2] {$c$};
  \node[main node] (4) [left of=3] {$a$};
  \node[main node] (5) [above of=4] {$a$};
  \node[main node] (6) [left of=5] {$c$};
  \node[main node] (8) [left of=6] {$b$};
  \node[main node] (9) [left of=8] {$a$};
  \node (10) [below of=9] {$\Gamma(Z_7)$};
  \node (11) [below of=8] {$\Gamma(Z_6)$};
  \node (12) [below of=1] {$\Gamma(Z_5)$};

  \path[every node/.style={font=\sffamily\small}]
    (7) edge (0)
    (9) edge (10)
    (8) edge (11)
    (1) edge (12)
    (8) edge node [above] {$\epsilon,b$} (9)
    (6) edge node [above] {$\epsilon,\epsilon$} (8)
    (5) edge node [above] {$a,\epsilon$} (6)
    (1) edge node [above] {$\epsilon,b$} (7)
    (2) edge node [above] {$a,\epsilon$} (1)   
    (3) edge node [right] {$\epsilon,\epsilon$} (2)
    (4) edge node [above] {$a,\epsilon$} (3)
    (5) edge node [left] {$\epsilon,c$} (4)
    (2) edge node [above] {$\epsilon,b$} (5);
\end{tikzpicture}

\end{center}

Singular points in $\Omega_a$ are given by infinite tails ending
at occurrences of the letter $a$ as a state in the prefix-suffix expansion
of the above singular prefix-suffix expansions. However, for $Z_5$, $Z_6$ and, $Z_7$, $Z_8$ we
only need to consider desubstitution paths with different first vertex,
and thus $P_3$ is not a singular point of $\Omega_b$ and, $P_4$ is not a singular point of $\Omega_a$. 
Thus, for $\Omega_a$ we have to consider: $\Gamma(Z_1)$, $B^2(\Gamma(Z_6))$ and, $B^3(\Gamma(Z_5))$ and we get
\[
V_a=\Sing(\Omega_a)=\{P_1, aP_3, acP_4\}.
\]
Similarly, singular points in $\Omega_b$ and $\Omega_c$ are given by occurrences of the state $b$
in desubstitution paths.
\[
V_b=\Sing(\Omega_b)=\{P_1,P_2,bP_4\},\quad V_c=\Sing(\Omega_c)=\{P_2,cP_3,cP_4\}.
\]

Finally, to understand the tree substitution we need to know the
images of the singular points under the map $\tau$. Recall that the
image by $\tau$ of a singular point given by a singular prefix-suffix
expansion $\gamma$ is given by the beheaded prefix-suffix expansion
$B(\gamma)$ in the copy of the prototile corresponding to the heading
edge $e_1(\gamma)$ of $\gamma$.

Using the map $Q$ we also get descriptions of the singular points in
the repelling tree as well as their images by the expanding homothety
$H\inv$. We sum-up our calculations in the following table:

\[
\begin{array}{|r||c|c|c||c|c|c||}
\cline{2-7}
\multicolumn{1}{c||}{}&\multicolumn{3}{c||}{V_a}&\multicolumn{3}{c||}{V_b}\\
\cline{2-7}
\multicolumn{1}{c||}{}&\circlebox{red}{1}&\circlebox{red}{2}&\circlebox{red}{3}&\circlebox{green}{1}&\circlebox{green}{2}&\circlebox{green}{3}\\
\hline
\gamma&B^2(\Gamma(Z_6))&\Gamma(Z_1)&B^3(\Gamma(Z_5))&\Gamma(Z_2)&B(\Gamma(Z_5))&\Gamma(Z_3)\\
\hline
Q(\gamma)&aP_3    &P_1     &acP_4&P_1     &bP_4    &P_2\\
\hline
\hline
e_1(\gamma)&a\stackrel{\epsilon,b}{\longleftarrow}b&a\stackrel{\epsilon,b}{\longleftarrow}b&a\stackrel{\epsilon,c}{\longleftarrow}a&b\stackrel{\epsilon,\epsilon}{\longleftarrow}c&b\stackrel{\epsilon,\epsilon}{\longleftarrow}c&b\stackrel{a,\epsilon}{\longleftarrow}b\\
\hline
B(\gamma)&B(\Gamma(Z_5))&\Gamma(Z_3)&B^2(\Gamma(Z_6))&\Gamma(Z_4)&B^2(\Gamma(Z_5))&\Gamma(Z_2)\\
\hline
\tau (Q(\gamma))&bP_4&P_2&aP_3&P_2&cP_3&cb\inv P_1\\
\hline
\end{array}
\]

\[
\begin{array}{|r||c|c|c||}
\cline{2-4}
\multicolumn{1}{c||}{}&\multicolumn{3}{c||}{V_c}\\
\cline{2-4}
\multicolumn{1}{c||}{}&\circlebox{blue}{1}&\circlebox{blue}{2}&\circlebox{blue}{3}\\
\hline
\gamma&\Gamma(Z_4)&B^2(\Gamma(Z_5))&B(\Gamma(Z_6))\\
\hline
Q(\gamma)&P_2       &cP_3        &cP_4\\
\hline
\hline
e_1(\gamma)&c\stackrel{a,\epsilon}{\longleftarrow}a&c\stackrel{a,\epsilon}{\longleftarrow}a&c\stackrel{a,\epsilon}{\longleftarrow}a\\
\hline 
B(\gamma)&\Gamma(Z_1)&B^3(\Gamma(Z_5))&B^2(\Gamma(Z_6))\\
\hline
\tau(Q(\gamma))&cb\inv P_1&cb\inv acP_4&cb\inv aP_3\\
\hline
\end{array}
\]

We get the abstract tree substitution depicted in Figure~\ref{fig:ex1-tree-substitution}: 
\begin{align*}
\tau : W_a &\mapsto (W_a \sqcup W_b)/\sim  \\
W_b &\mapsto (W_b \sqcup W_c)/\sim \\
W_c &\mapsto W_a
\end{align*}
From the pairs of singular prefix-suffix expansions we get that the
gluing points $\Gamma(Z_1)$ in $W_a$ and $\Gamma(Z_2)$ in $W_b$ are
identified in $W$ as well as the gluing points $\Gamma(Z_3)$ in $W_b$
and $\Gamma(Z_4)$ in $W_c$. The same gluing points are identified in
$\tau(W_a)$ and $\tau(W_b)$.

\begin{figure}[h]
\centering
\includegraphics[scale=.7]{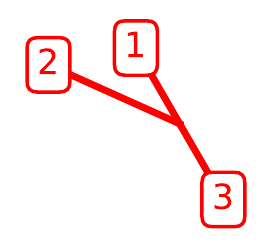}
\raisebox{.7cm}{$\mapsto$}
\includegraphics[scale=.7]{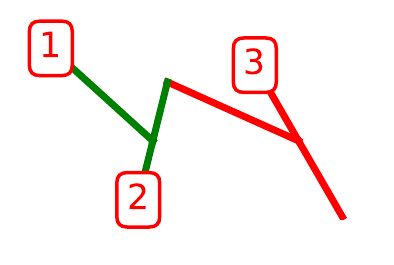}
\qquad
\includegraphics[scale=.7]{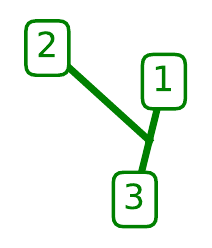}
\raisebox{.7cm}{$\mapsto$}
\includegraphics[scale=.7]{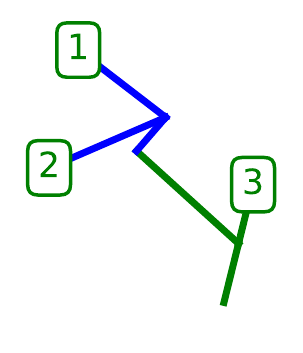}
\newline
\includegraphics[scale=.7]{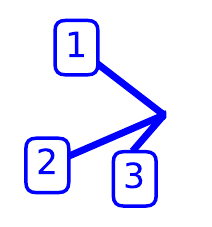}
\raisebox{.7cm}{$\mapsto$}\hspace{-.1cm}
\includegraphics[scale=.7]{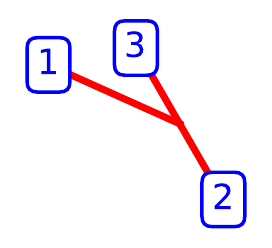}
\caption{\label{fig:ex1-tree-substitution} Tree substitution for $a\mapsto ac,b\mapsto ab,c\mapsto b$. From left to right we have the substitution of $W_a$ (red), $W_b$ (green) and $W_c$ (blue) respectively. }
\end{figure}

\begin{figure}[h]
\centering
\includegraphics[scale=.4]{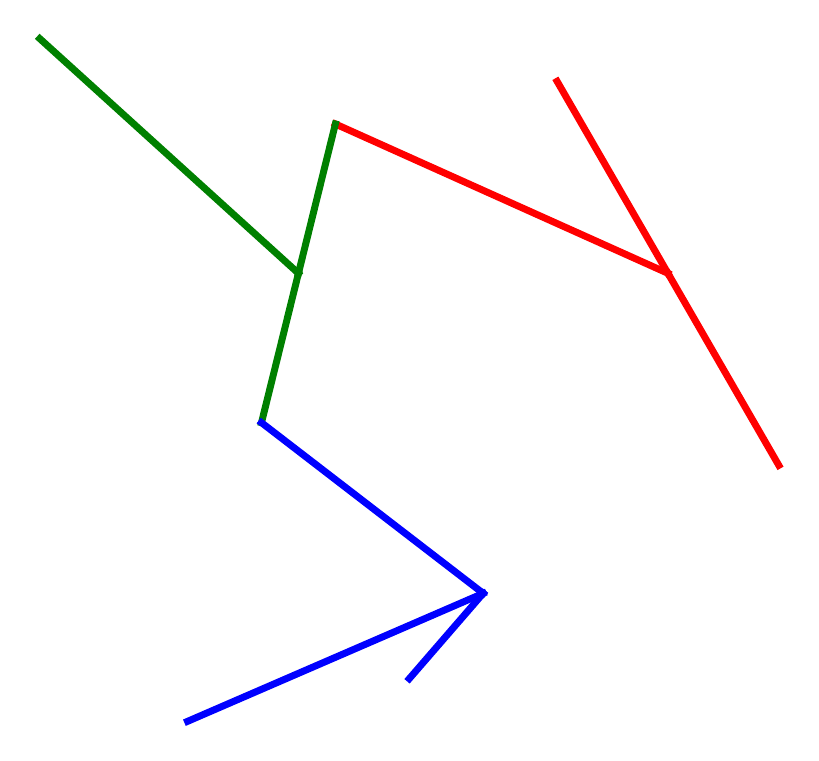}
\includegraphics[scale=.4]{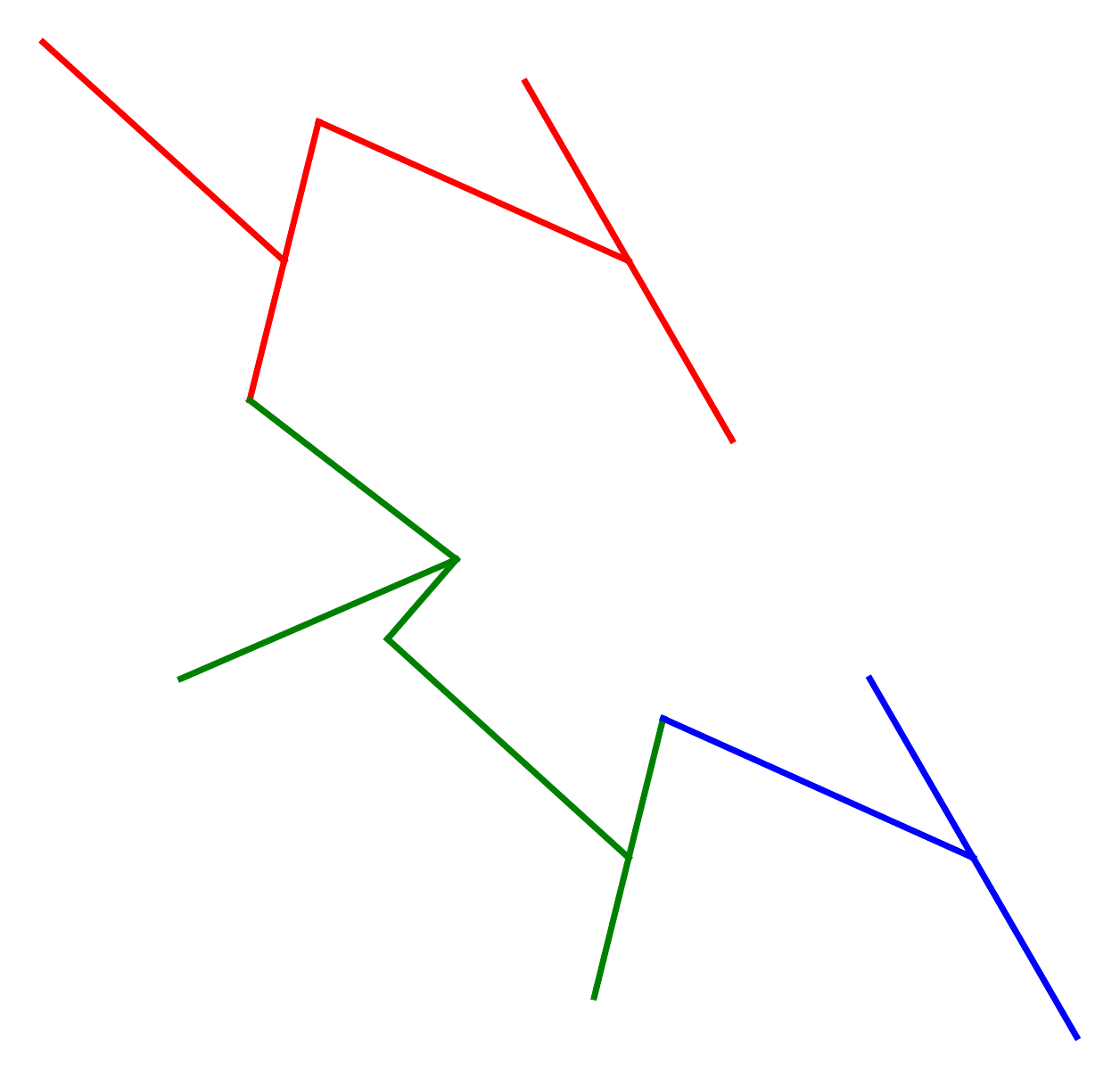}
\caption{\label{fig:ex1-arbres-0-et-1}The trees $W$ and $\tau(W)$ for  $a\mapsto ac,b\mapsto ab,c\mapsto b$.}
\end{figure}

We proceed to get a piecewise rotation of the contour circle. We fix
the cyclic orders given by the embeddings in the contracting plane
oriented clockwise. The contour substitution $\chi$ and its dual are
\[
\chi:\begin{array}[t]{rcl}
a_{13} &\mapsto& b_{21}a_{21}\\
a_{21} &\mapsto& b_{32}\\
a_{32} &\mapsto& a_{13}a_{32}b_{13}\\
b_{13} &\mapsto& c_{13}b_{21}\\
b_{21} &\mapsto& c_{21}\\
b_{32} &\mapsto& b_{13}b_{32}c_{32}\\
c_{13} &\mapsto& a_{21}\\
c_{21} &\mapsto& a_{32}\\
c_{32} &\mapsto& a_{13}
\end{array}
\quad
\chi^*:\begin{array}[t]{rcl}
a_{13} &\mapsto& a_{32}c_{32}\\
a_{21} &\mapsto& a_{13}c_{13}\\
a_{32} &\mapsto& a_{32}c_{21}\\
b_{13} &\mapsto& a_{32}b_{32}\\
b_{21} &\mapsto& a_{13}b_{13}\\
b_{32} &\mapsto& a_{21}b_{32}\\
c_{13} &\mapsto& b_{13}\\
c_{21} &\mapsto& b_{21}\\
c_{32} &\mapsto& b_{32}
\end{array}
\]
There is a unique infinite prefix-suffix expansion with empy suffix both for $\sigma$ and $\chi^*$:

\begin{center}
\begin{tikzpicture}[->,>=stealth',shorten >=1pt,auto,node distance=2cm,
                    thick,main node/.style={circle,draw,font=\sffamily\bfseries}]

  \node[main node] (1) [right of=0] {$b$};
  
  \path[every node/.style={font=\sffamily\small}]
    (1) edge (0)
    (1) edge [loop right] node [right] {$a,\epsilon$} (1);   
\end{tikzpicture}
\begin{tikzpicture}[->,>=stealth',shorten >=1pt,auto,node distance=2cm,
                    thick,main node/.style={circle,draw,font=\sffamily\bfseries}]

  \node (0) {};                  
  \node[main node] (1) [right of=0] {$b_{32}$};
  
  \path[every node/.style={font=\sffamily\small}]
    (1) edge (0)
    (1) edge [loop right] node [right] {$a_{21},\epsilon$} (1);   
\end{tikzpicture}
\end{center}

Thus we subdivide the arc $b_{32}$ into two consecutive arcs $b_{34}$
and $b_{42}$. Using the Vershik map, we compute the images of the arcs
$Q([\gamma])$ for finite prefix-suffix expansion paths with non-empty
suffix. We observe that the consecutive arcs $a_{13}$ and $a_{32}$ are
rotated by the same angle and, similarly the consecutive arcs $b_{42}$
and $b_{21}$. We get the piecewise rotation of the circle with eight
pieces described in Figure~\ref{fig:ex1-circle}
\begin{figure}[h]
\centering
\includegraphics[scale=1]{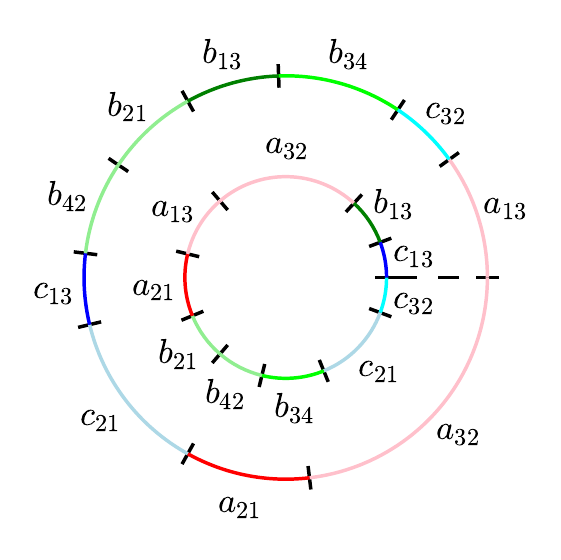}
\caption{\label{fig:ex1-circle} Piecewise rotation of the circle for
  the substitution $a\mapsto ac,b\mapsto ab,c\mapsto b$.  The
  innermost circle is the contour while the outermost circle is the
  result of the piecewise exchange.}
\end{figure}

The tree substitution fails to depict a tree inside the
contracting plane, as illustrated in Figure~\ref{fig:ex1-fourth-iterate}. Thus we also need in this example to prune the tree
substitution (see Section~\ref{sec:pruning-cheating}) to get a tree inside the contracting plane. After some iterations we get the tree pictured in Figure~\ref{fig:ex1-eigth-iterate}.

\begin{figure}[h]
\centering
\begin{overpic}[scale=.6]{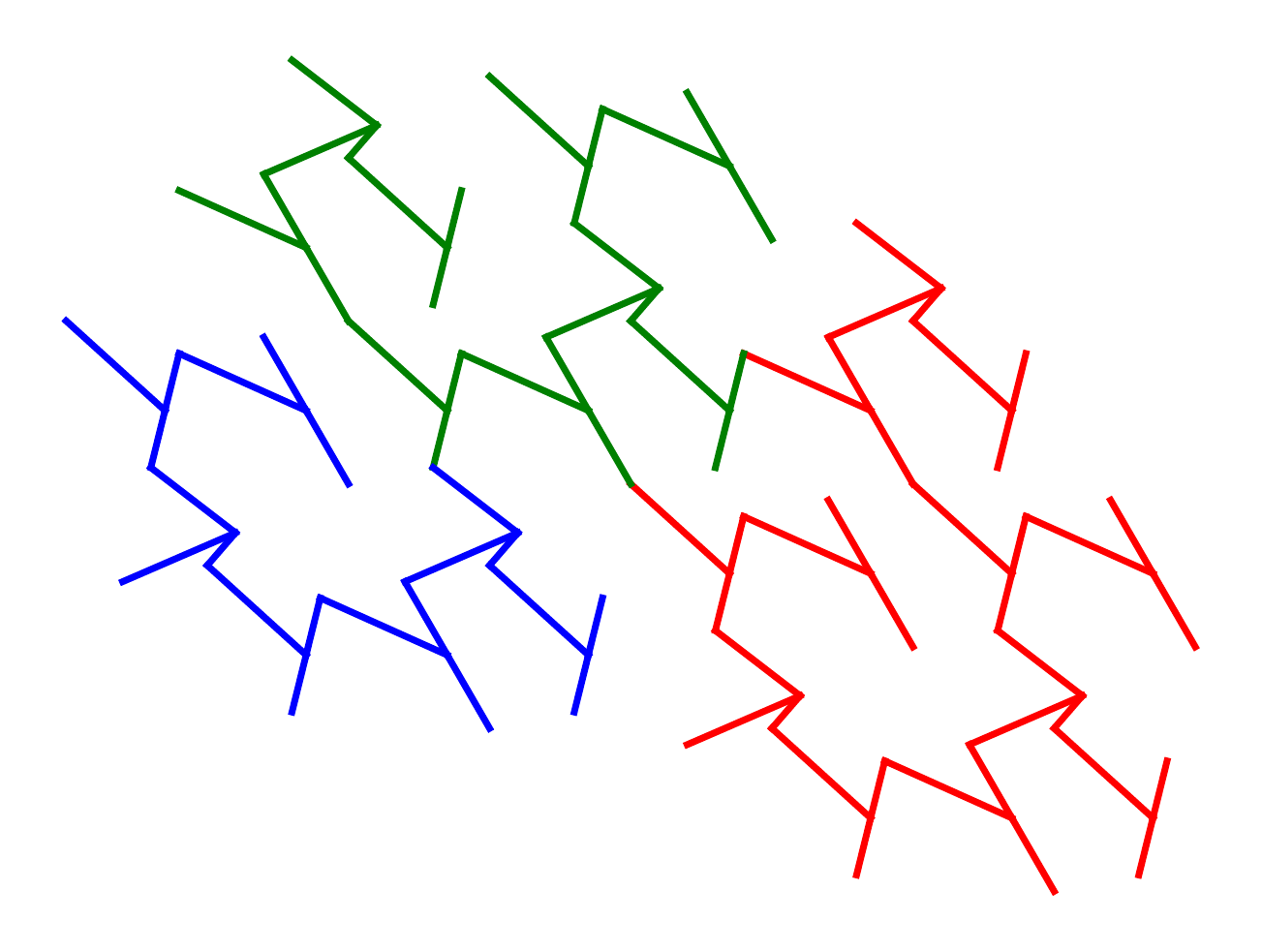}
\put (48,36) {$\bigcirc$}
\put (50,10) {\rule{1pt}{20mm}}
\put (0,10) {\parbox{3.9cm}{\scriptsize These two edges are not adjacent in the abstract tree.}}
\end{overpic}
\caption{\label{fig:ex1-fourth-iterate}Projection of the tree $\tau^4(W)$ for the tree substitution of Figure~\ref{fig:ex1-tree-substitution}. This is not a tree inside the contracting plane.}
\end{figure}

\begin{figure}[h]
\centering
\includegraphics[scale=.2]{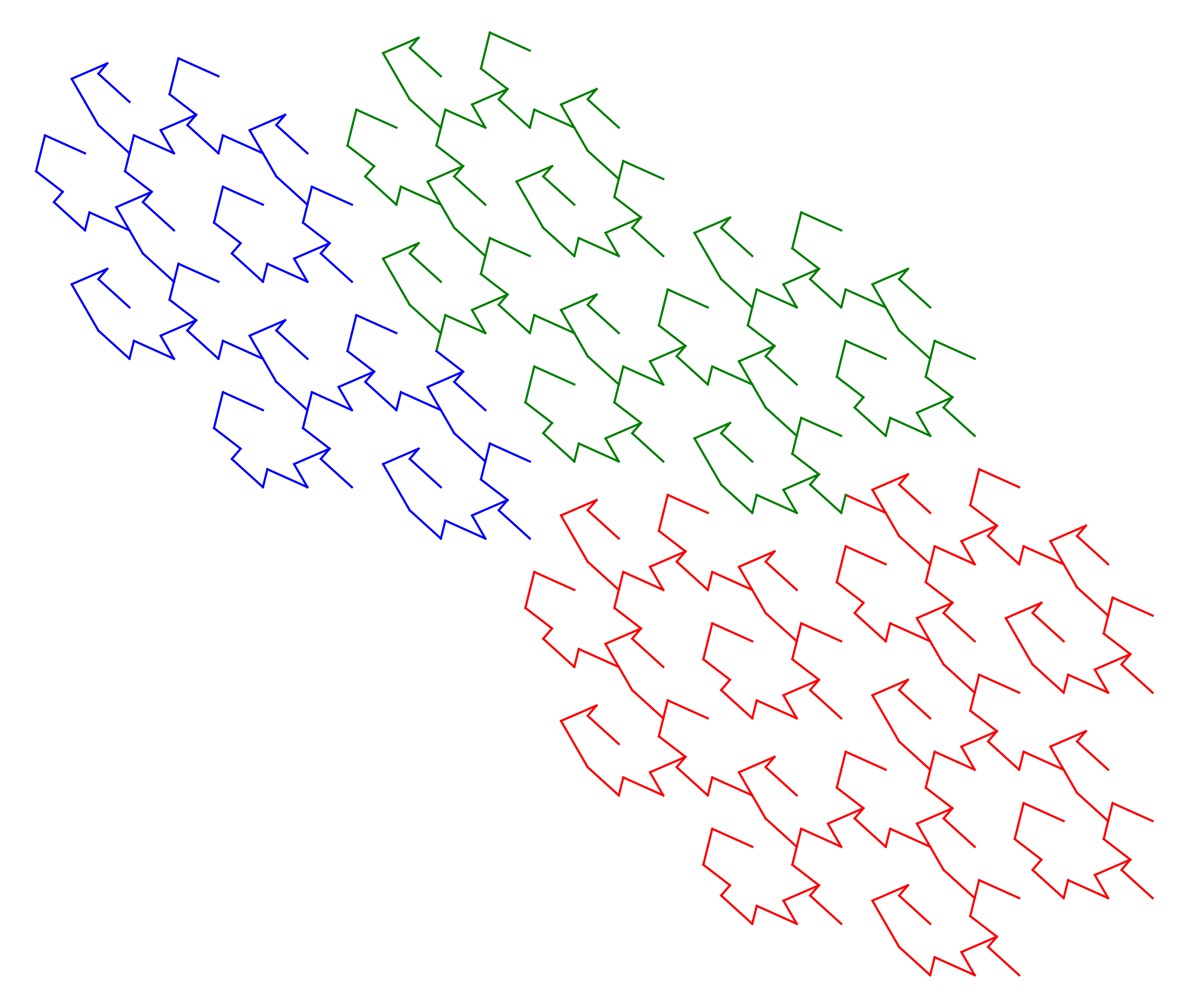}
\caption{\label{fig:ex1-eigth-iterate}Projection in the contracting plane of the eighth iterate of the pruned tree substitution for $\sigma:a\mapsto ac$, $b\mapsto ab$, $c\mapsto b$.}
\end{figure}

\subsection{Second example}\label{sec:second-example}

We consider the following substitution which is an iwip automorphism of the free group $F_{\{a,b,c\}}$:
\[
\sigma:\begin{array}[t]{rcl}a&\mapsto&abc\\b&\mapsto&bcabc\\c&\mapsto&cbcabc\end{array}
\]
This example comes from a family of irreducible substitutions which are parageometric iwip automorphisms studied by Leroy~\cite{leroy-Sadic}.

As the computations are more complicated we will not detail them here. 

There are six attracting infinite fixed  words (three positive, one
negative and two more coming from two indivisible Nielsen paths).
Computing the singular pairs of prefix-suffix expansions and their
tails we get the prototiles of the tree substitution and their images
described in Figure~\ref{fig:leroy-ex-0-tree-substitution}.

\begin{figure}[h]
\centering
\includegraphics[scale=.5]{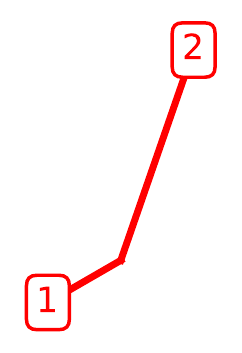}
\raisebox{1cm}{$\mapsto$}
\includegraphics[scale=.5]{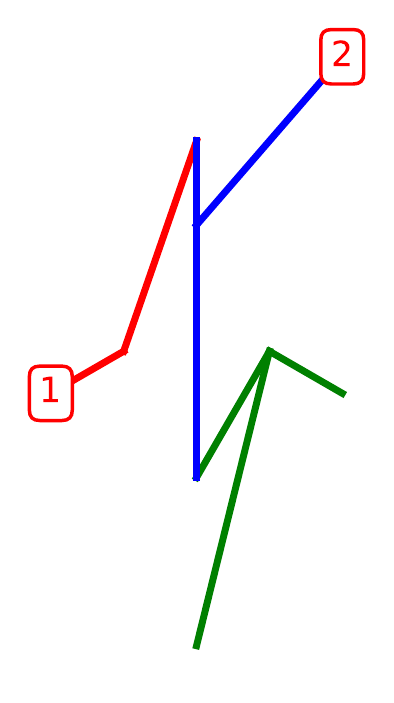}
\qquad
\includegraphics[scale=.5]{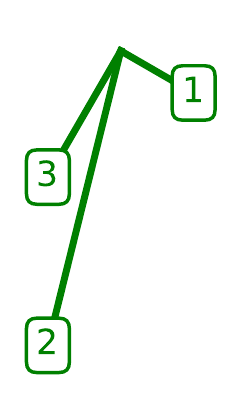}
\raisebox{1cm}{$\mapsto$}
\begin{overpic}[scale=.5]{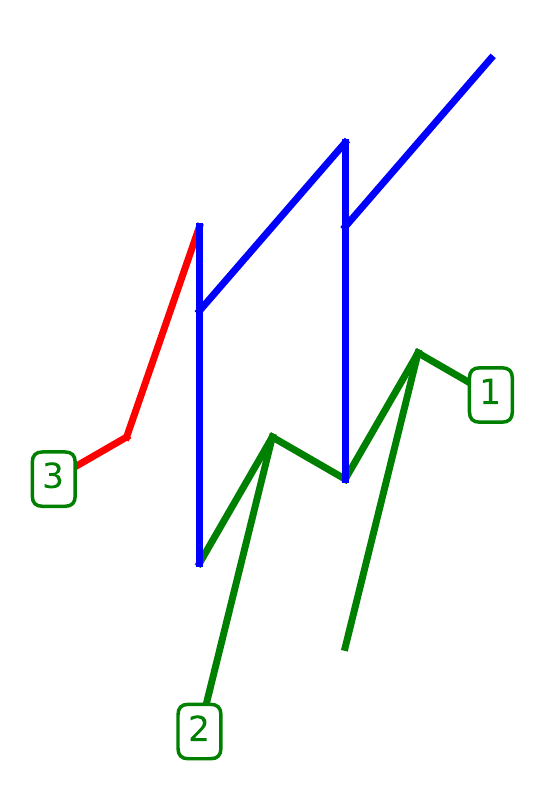}
\put (39,38) {$\bigcirc$}
\end{overpic}
\newline
\includegraphics[scale=.5]{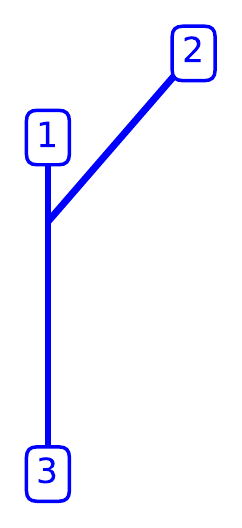}
\raisebox{1.5cm}{$\mapsto$}\hspace{-.1cm}
\begin{overpic}[scale=.5]{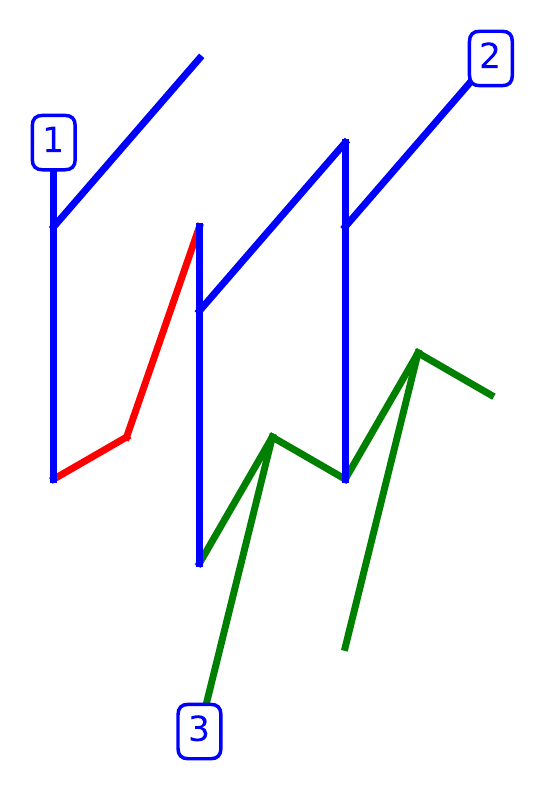}
\put (39,38) {$\bigcirc$}
\end{overpic}
\qquad
\raisebox{2cm}{\parbox{5cm}{In the abstract tree substitution the adjacency inside the circles should rather be:
\begin{center}
\includegraphics[scale=.5]{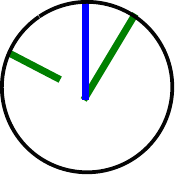}
\end{center}
}}
\caption{\label{fig:leroy-ex-0-tree-substitution}Tree substitution for the substitution $a\mapsto abc,\ b\mapsto bcabc,\ c\mapsto cbcabc$. The images of the prototiles $\tau(W_b)$ and $\tau(W_c)$ in the contracting plane are not trees.}
\end{figure}

We remark that the images $\tau(W_b)$ and $\tau(W_c)$ are not trees when embedded in the contracting plane. Thus we turn to the pruned tree substitution
as before. But again, as shown in Figure~\ref{fig:leroy-ex-0-first-iterate-pruned}, already the first image $\tau(W)$ is not a tree.

\begin{figure}[h]
\centering
\begin{overpic}[scale=.5]{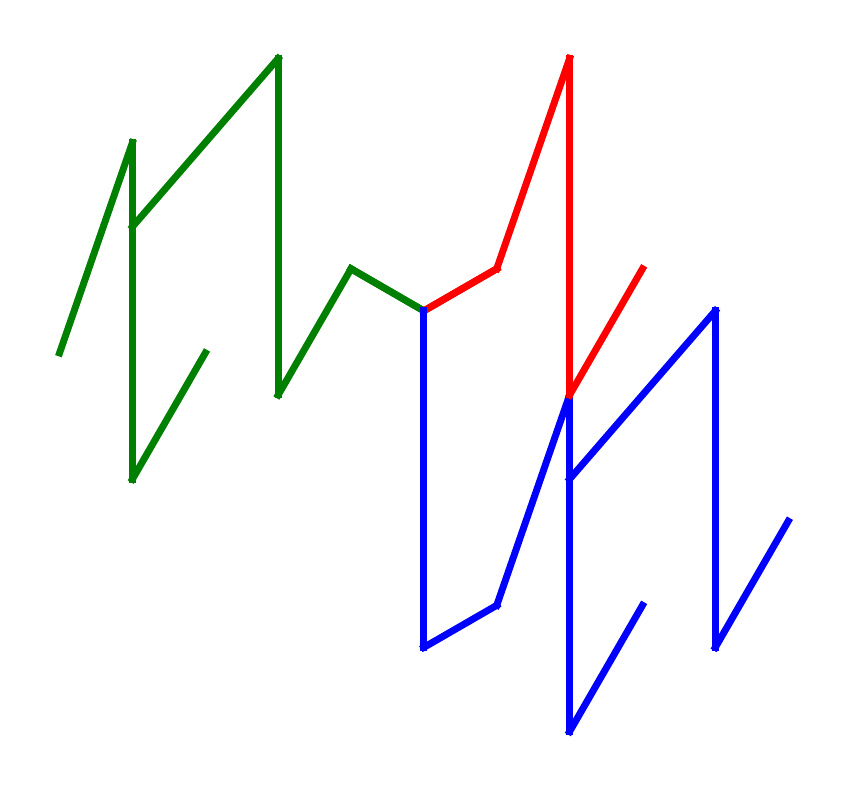}
\put (63,44) {$\bigcirc$}
\end{overpic}
\caption{\label{fig:leroy-ex-0-first-iterate-pruned}Projection in the contracting plane of the first iterate of the pruned tree substitution of $a\mapsto abc,\ b\mapsto bcabc,\ c\mapsto cbcabc$. In the abstract tree, the red and blue patches do not touch inside the circle.}
\end{figure}

We thus need to take a higher cover. This cover is provided by the
observation that the dual substitution satisfies
Condition~\ref{item:adjacency-quest-properties-dual-substitution} of
Question~\ref{quest:properties-dual-substitution}. Thus we can create
a tree substitution inside the tiles of the dual substitution as
illustrated in
Figures~\ref{fig:leroy_0_tree_substitution_pruned_and_dual} and \ref{fig:leroy_0_iterated_image_pruned_and_dual}.

\begin{figure}
  \centering
\begin{tabular}{rccccccc}
$W_i$:&
\raisebox{-2mm}{\includegraphics[scale=.22]{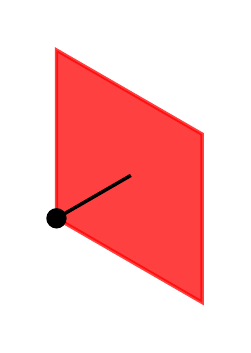}}&
\raisebox{-2mm}{\includegraphics[scale=.22]{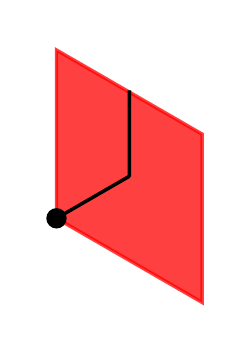}}&
\raisebox{-2mm}{\includegraphics[scale=.22]{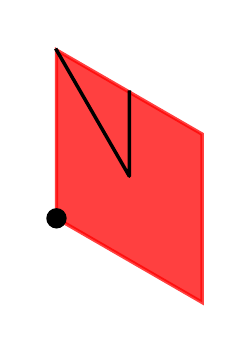}}&
\raisebox{-2mm}{\includegraphics[scale=.22]{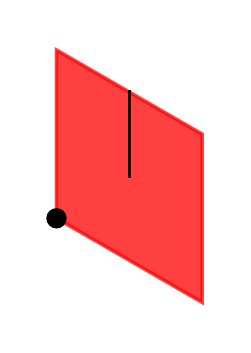}}&
\raisebox{-2mm}{\includegraphics[scale=.22]{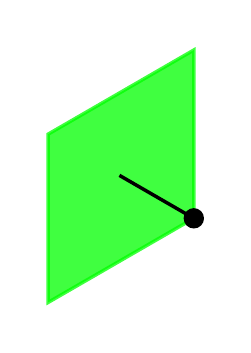}}&
\raisebox{-2mm}{\includegraphics[scale=.22]{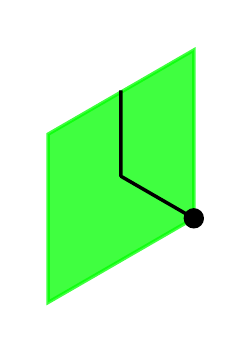}}&
\raisebox{-2mm}{\includegraphics[scale=.22]{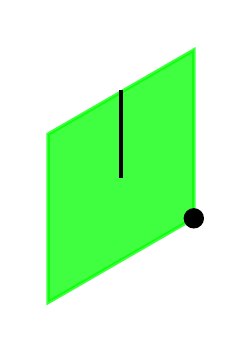}}\\
$\tau(W_i)$:&
\raisebox{-4mm}{\includegraphics[scale=.22]{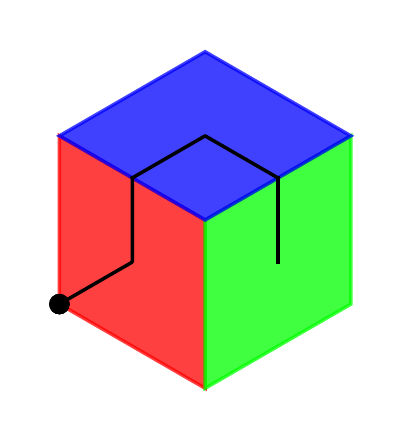}}&
\raisebox{-4mm}{\includegraphics[scale=.22]{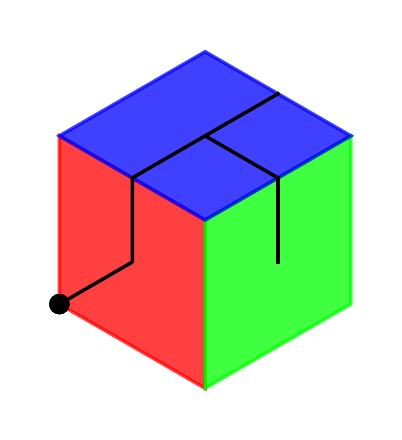}}&
\raisebox{-4mm}{\includegraphics[scale=.22]{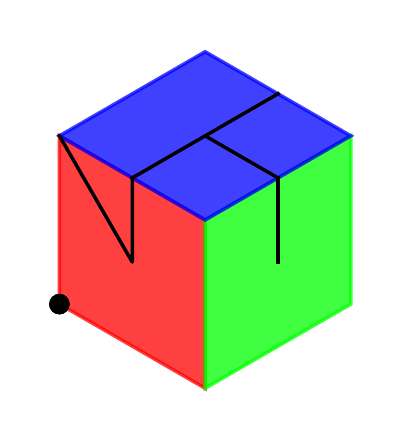}}&
\raisebox{-4mm}{\includegraphics[scale=.22]{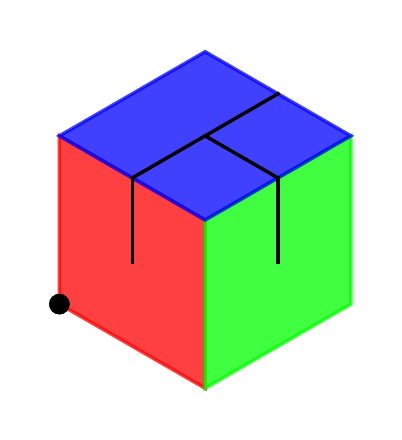}}&
\raisebox{-4mm}{\includegraphics[scale=.22]{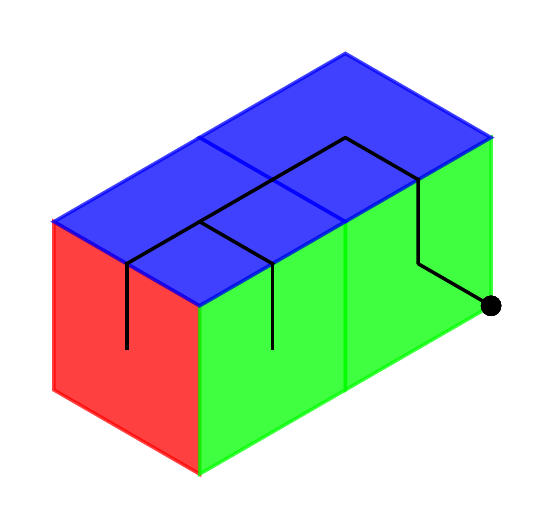}}&
\raisebox{-4mm}{\includegraphics[scale=.22]{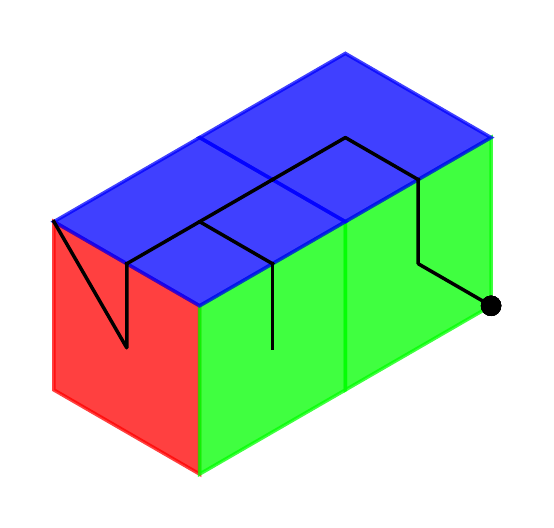}}&
\raisebox{-4mm}{\includegraphics[scale=.22]{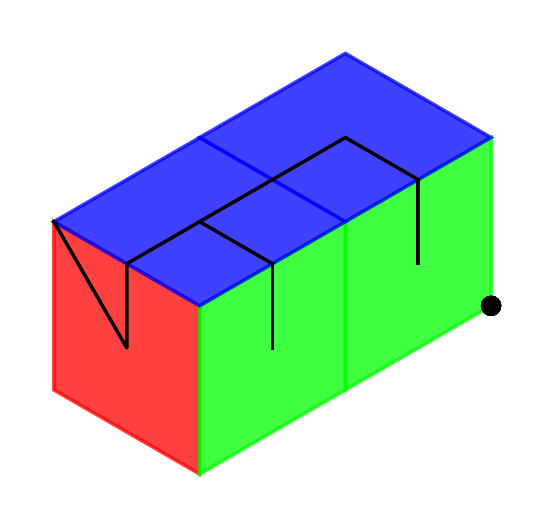}}
\end{tabular}

\begin{tabular}{rccccccc}
$W_i$:&
\raisebox{-2mm}{\includegraphics[scale=.22]{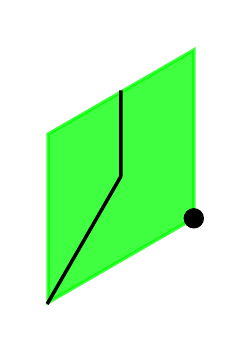}}&
\raisebox{-2mm}{\includegraphics[scale=.22]{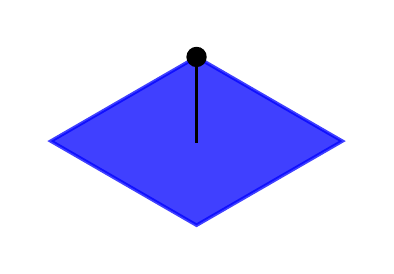}}&
\raisebox{-2mm}{\includegraphics[scale=.22]{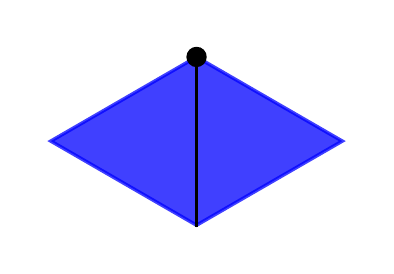}}&
\raisebox{-2mm}{\includegraphics[scale=.22]{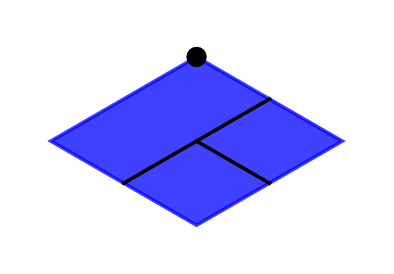}}&
\raisebox{-2mm}{\includegraphics[scale=.22]{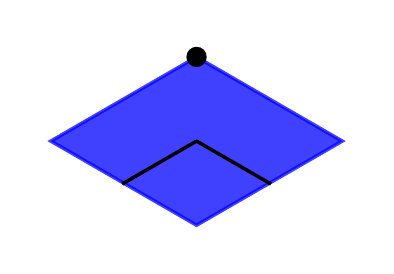}}&
\raisebox{-2mm}{\includegraphics[scale=.22]{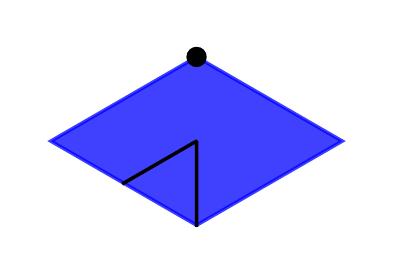}}\\
$\tau(W_i)$:&
\raisebox{-6mm}{\includegraphics[scale=.22]{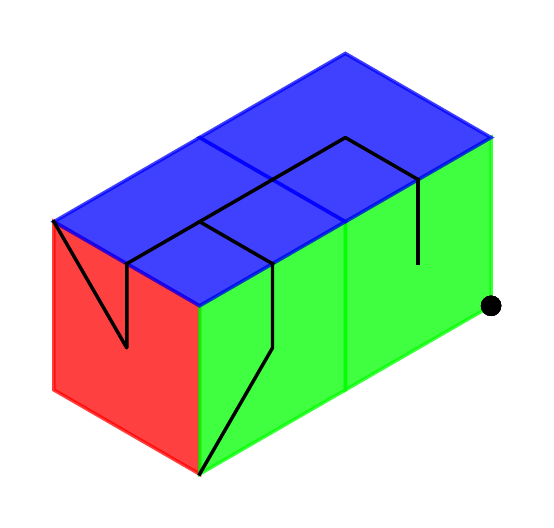}}&
\raisebox{-6mm}{\includegraphics[scale=.22]{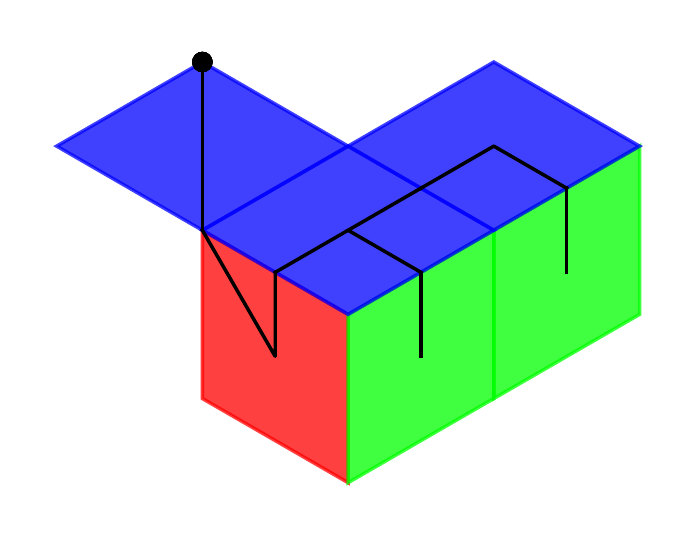}}&
\raisebox{-6mm}{\includegraphics[scale=.22]{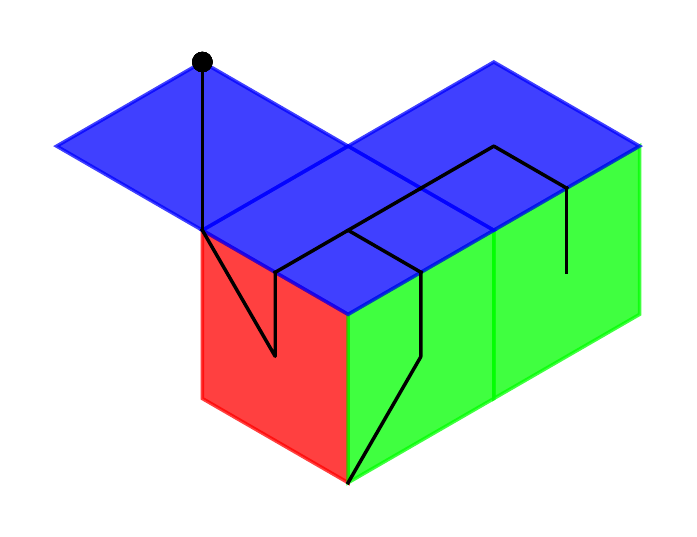}}&
\raisebox{-6mm}{\includegraphics[scale=.22]{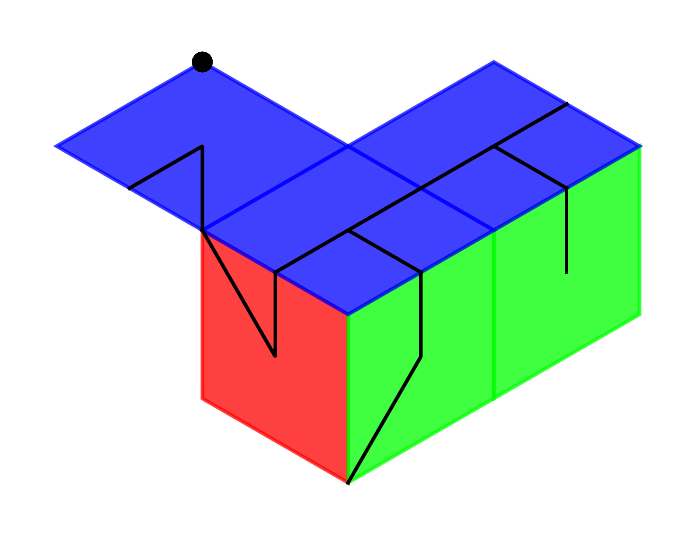}}&
\raisebox{-6mm}{\includegraphics[scale=.22]{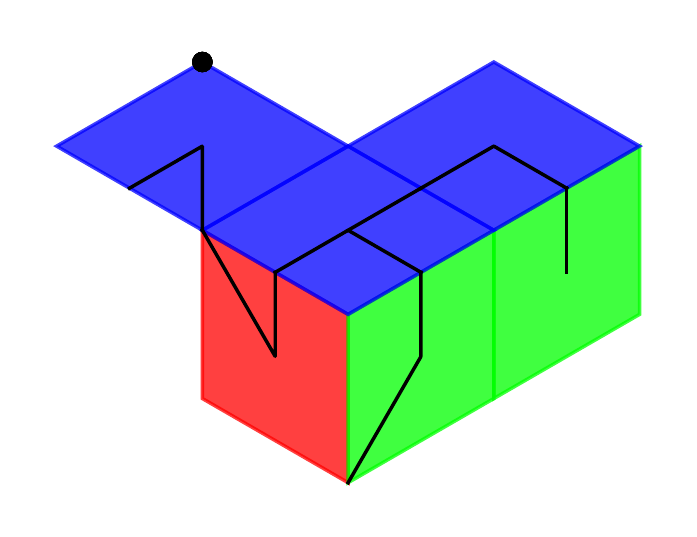}}&
\raisebox{-6mm}{\includegraphics[scale=.22]{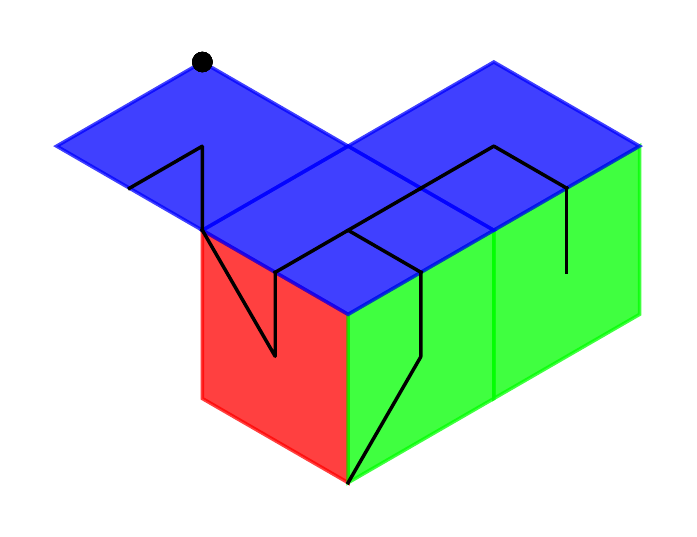}}
\end{tabular}
\caption{\label{fig:leroy_0_tree_substitution_pruned_and_dual}Tree
  substitution obtained by adjacency of tiles of the dual substitution
  for the substitution $a\mapsto abc,\ b\mapsto bcabc,\ c\mapsto
  cbcabc$. The black dot indicates the position of the origin.}\end{figure}

\begin{figure}[h]
\centering
\includegraphics[scale=.13]{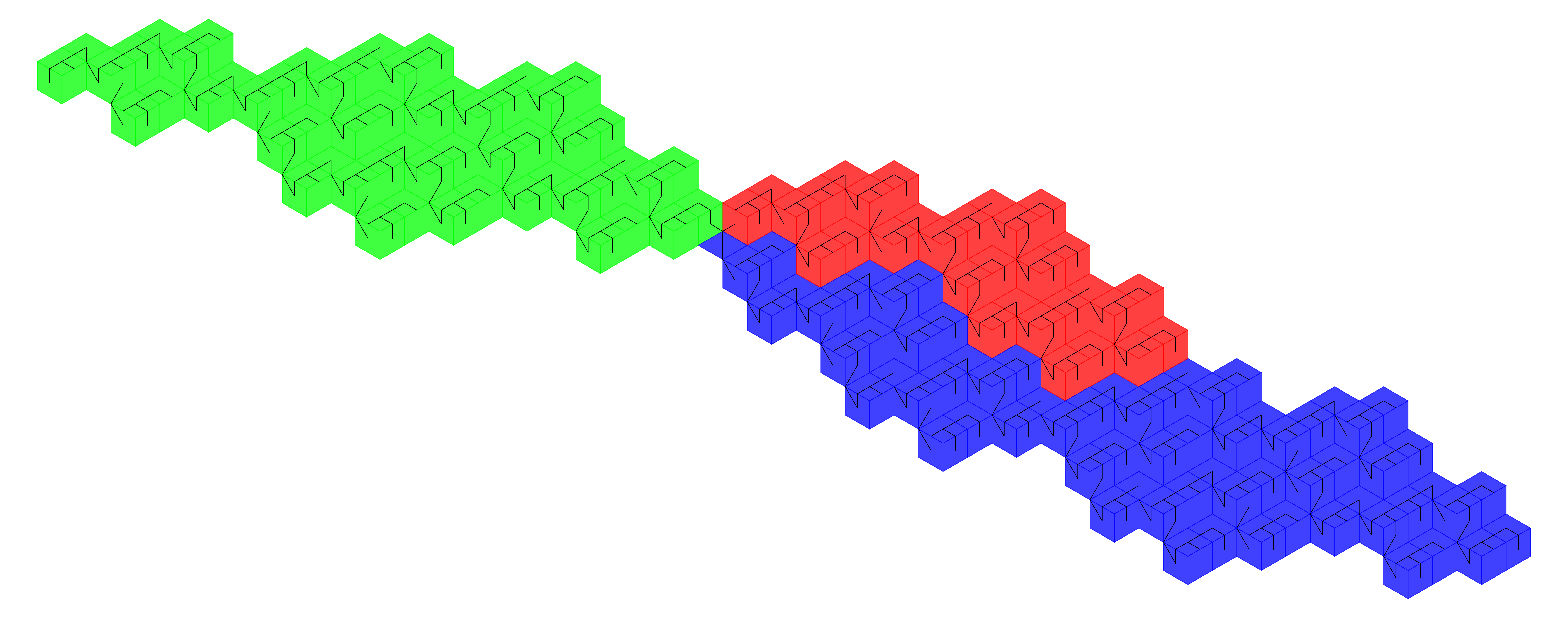}
\caption{\label{fig:leroy_0_iterated_image_pruned_and_dual}Third iterate of the tree substitution obtained by using Condition~\ref{item:adjacency-quest-properties-dual-substitution} of Question~\ref{quest:properties-dual-substitution} for the substitution $a\mapsto abc,\ b\mapsto bcabc,\ c\mapsto cbcabc$.}
\end{figure}

\bigskip

\noindent\textbf{Acknowledgements:} The authors are grateful to Pierre Arnoux for many discussions and his interest in our work.


\begin{thebibliography}{ABB{\etalchar{+}}15}

\bibitem[ABB{\etalchar{+}}15]{ABBLS15}
Shigeki Akiyama, Marcy Barge, Valérie Berth{\'e}, Jeong-Yup Lee, and Anne
  Siegel.
\newblock On the {P}isot substitution conjecture.
\newblock In {\em Mathematics of aperiodic order}, volume 309 of {\em Progr.
  Math.}, pages 33--72. Birkh\"auser/Springer, Basel, 2015.

\bibitem[AI01]{arnoux-ito}
Pierre Arnoux and Shunji Ito.
\newblock Pisot substitutions and {R}auzy fractals.
\newblock {\em Bull. Belg. Math. Soc. Simon Stevin}, 8(2):181--207, 2001.
\newblock Journ\'ees Montoises d'Informatique Th\'eorique (Marne-la-Vall\'ee,
  2000).

\bibitem[AR91]{arnoux-rauzy}
Pierre Arnoux and G\'erard Rauzy.
\newblock Repr\'esentation g\'eom\'etrique de suites de complexit\'e {$2n+1$}.
\newblock {\em Bull. Soc. Math. France}, 119(2):199--215, 1991.

\bibitem[Arn88]{A88}
Pierre Arnoux.
\newblock Un exemple de semi-conjugaison entre un \'echange d'intervalles et
  une translation sur le tore.
\newblock {\em Bull. Soc. Math. France}, 116(4):489--500 (1989), 1988.

\bibitem[AY81]{arnoux-yoccoz}
Pierre Arnoux and Jean-Christophe Yoccoz.
\newblock Construction de diff\'eomorphismes pseudo-{A}nosov.
\newblock {\em C. R. Acad. Sci. Paris S\'er. I Math.}, 292(1):75--78, 1981.

\bibitem[BD01]{bd-complete-invariant}
Marcy Barge and Beverly Diamond.
\newblock A complete invariant for the topology of one-dimensional substitution
  tiling spaces.
\newblock {\em Ergodic Theory Dynam. Systems}, 21(5):1333--1358, 2001.

\bibitem[BD07]{bd-proximality}
Marcy Barge and Beverly Diamond.
\newblock Proximality in {P}isot tiling spaces.
\newblock {\em Fund. Math.}, 194(3):191--238, 2007.

\bibitem[BDH03]{bdh-asymptotic-orbits}
Marcy Barge, Beverly Diamond, and Charles Holton.
\newblock Asymptotic orbits of primitive substitutions.
\newblock {\em Theoret. Comput. Sci.}, 301(1-3):439--450, 2003.

\bibitem[BDJP14]{BDJP}
Val\'erie Berth\'e, \'Eric Domenjoud, Damien Jamet, and Xavier Proven\c{c}al.
\newblock Fully subtractive algorithm, tribonacci numeration and connectedness
  of discrete planes.
\newblock In {\em Numeration and substitution 2012}, RIMS K\^oky\^uroku
  Bessatsu, B46, pages 159--174. Res. Inst. Math. Sci. (RIMS), Kyoto, 2014.

\bibitem[BH92]{bh-traintrack}
Mladen Bestvina and Michael Handel.
\newblock Train tracks and automorphisms of free groups.
\newblock {\em Ann. of Math. (2)}, 135(1):1--51, 1992.

\bibitem[BJ12]{BJ12}
Xavier Bressaud and Yann Jullian.
\newblock Interval exchange transformation extension of a substitution
  dynamical system.
\newblock {\em Confluentes Math.}, 4(4):1250005, 54, 2012.

\bibitem[BST10]{cant-chapter7}
Val\'erie Berth\'e, Anne Siegel, and J\"org Thuswaldner.
\newblock Substitutions, {R}auzy fractals and tilings.
\newblock In {\em Combinatorics, automata and number theory}, volume 135 of
  {\em Encyclopedia Math. Appl.}, pages 248--323. Cambridge Univ. Press,
  Cambridge, 2010.

\bibitem[Cas97]{cassaigne-speciaux}
Julien Cassaigne.
\newblock Complexit\'e et facteurs sp\'eciaux.
\newblock {\em Bull. Belg. Math. Soc. Simon Stevin}, 4(1):67--88, 1997.
\newblock Journ\'ees Montoises (Mons, 1994).

\bibitem[CH12]{ch-b}
Thierry Coulbois and Arnaud Hilion.
\newblock Botany of irreducible automorphisms of free groups.
\newblock {\em Pacific J. Math.}, 256(2):291--307, 2012.

\bibitem[CH14]{ch-a}
Thierry Coulbois and Arnaud Hilion.
\newblock Rips induction: index of the dual lamination of an
  {$\mathbb{R}$}-tree.
\newblock {\em Groups Geom. Dyn.}, 8(1):97--134, 2014.

\bibitem[Chi76]{chiswell-length-function}
Ian~M. Chiswell.
\newblock Abstract length functions in groups.
\newblock {\em Math. Proc. Cambridge Philos. Soc.}, 80(3):451--463, 1976.

\bibitem[Chi01]{chiswell-lambda-tree}
Ian~M. Chiswell.
\newblock {\em Introduction to {$\Lambda$}-trees}.
\newblock World Scientific Publishing Co., Inc., River Edge, NJ, 2001.

\bibitem[CHL09]{chl4}
Thierry Coulbois, Arnaud Hilion, and Martin Lustig.
\newblock {$\mathbb R$}-trees, dual laminations, and compact systems of partial
  isometries.
\newblock {\em Math. Proc. Cambridge Phil. Soc.}, 147:345--368, 2009.

\bibitem[CHR15]{chr}
Thierry Coulbois, Arnaud Hilion, and Patrick Reynolds.
\newblock Indecomposable {$F_N$}-trees and minimal laminations.
\newblock {\em Groups Geom. Dyn.}, 9(2):567--597, 2015.

\bibitem[CL95]{cl-verysmall}
Marshall~M. Cohen and Martin Lustig.
\newblock Very small group actions on {${\bf R}$}-trees and {D}ehn twist
  automorphisms.
\newblock {\em Topology}, 34(3):575--617, 1995.

\bibitem[CL15]{cl-long-turns}
Thierry Coulbois and Martin Lustig.
\newblock Long turns, {INP}'s and indices for free group automorphisms.
\newblock {\em Illinois J. Math.}, 59(4):1087--1109, 2015.

\bibitem[Cou15]{coulbois-sage}
Thierry Coulbois.
\newblock Train-tracks for sage, 2015.

\bibitem[CS01a]{cs-pref-suff}
Vincent Canterini and Anne Siegel.
\newblock Automate des pr\'efixes-suffixes associ\'e \`a une substitution
  primitive.
\newblock {\em J. Th\'eor. Nombres Bordeaux}, 13(2):353--369, 2001.

\bibitem[CS01b]{cs}
Vincent Canterini and Anne Siegel.
\newblock Geometric representation of substitutions of {P}isot type.
\newblock {\em Trans. Amer. Math. Soc.}, 353(12):5121--5144 (electronic), 2001.


\bibitem[DPV16]{DPV}
Eric Domenjoud, Xavier Proven\c{c}al, and Laurent Vuillon.
\newblock Palindromic language of thin discrete planes.
\newblock {\em Theoret. Comput. Sci.}, 624:101--108, 2016.

\bibitem[Dur10]{durand-cant6}
Fabien Durand.
\newblock Combinatorics on {B}ratteli diagrams and dynamical systems.
\newblock In {\em Combinatorics, automata and number theory}, volume 135 of
  {\em Encyclopedia Math. Appl.}, pages 324--372. Cambridge Univ. Press,
  Cambridge, 2010.

\bibitem[GJLL98]{gjll}
Damien Gaboriau, Andre Jaeger, Gilbert Levitt, and Martin Lustig.
\newblock An index for counting fixed points of automorphisms of free groups.
\newblock {\em Duke Math. J.}, 93(3):425--452, 1998.

\bibitem[HZ01]{hz-directed}
Charles Holton and Luca~Q. Zamboni.
\newblock Directed graphs and substitutions.
\newblock {\em Theory Comput. Syst.}, 34(6):545--564, 2001.

\bibitem[Jul11]{jullian-coeur}
Yann Jullian.
\newblock Construction du c\oe ur compact d'un arbre r\'eel par substitution
  d'arbre.
\newblock {\em Ann. Inst. Fourier (Grenoble)}, 61(3):851--904, 2011.

\bibitem[KL14]{kl-diagonal-closure}
Ilya Kapovich and Martin Lustig.
\newblock Invariant laminations for irreducible automorphisms of free groups.
\newblock {\em Q. J. Math.}, 65(4):1241--1275, 2014.

\bibitem[Ler14]{leroy-Sadic}
Julien Leroy.
\newblock An {$S$}-adic characterization of minimal subshifts with first
  difference of complexity {$1\le p(n+1)-p(n)\le2$}.
\newblock {\em Discrete Math. Theor. Comput. Sci.}, 16(1):233--286, 2014.

\bibitem[LL03]{ll-north-south}
Gilbert Levitt and Martin Lustig.
\newblock Irreducible automorphisms of {$F\sb n$} have north-south dynamics on
  compactified outer space.
\newblock {\em J. Inst. Math. Jussieu}, 2(1):59--72, 2003.

\bibitem[Lyn63]{lyndon}
Roger~C. Lyndon.
\newblock Length functions in groups.
\newblock {\em Math. Scand.}, 12:209--234, 1963.

\bibitem[Mos92]{mosse-puissances}
Brigitte Moss\'e.
\newblock Puissances de mots et reconnaissabilit\'e des points fixes d'une
  substitution.
\newblock {\em Theoret. Comput. Sci.}, 99(2):327--334, 1992.

\bibitem[Mos96]{mosse-reconnaissabilite}
Brigitte Moss{\'e}.
\newblock Reconnaissabilit\'e des substitutions et complexit\'e des suites
  automatiques.
\newblock {\em Bull. Soc. Math. France}, 124(2):329--346, 1996.

\bibitem[MW88]{mw}
R.~Daniel Mauldin and Stanley~C. Williams.
\newblock Hausdorff dimension in graph directed constructions.
\newblock {\em Trans. Amer. Math. Soc.}, 309(2):811--829, 1988.

\bibitem[Que87]{quef-book}
Martine Queff{\'e}lec.
\newblock {\em Substitution dynamical systems---spectral analysis}, volume 1294
  of {\em Lecture Notes in Mathematics}.
\newblock Springer-Verlag, Berlin, 1987.

\bibitem[Rau82]{rauzy}
Gérard Rauzy.
\newblock Nombres alg\'ebriques et substitutions.
\newblock {\em Bull. Soc. Math. France}, 110(2):147--178, 1982.

\bibitem[Sage]{sagemath}
The~Sage Developers.
\newblock {\em {S}ageMath, the {S}age {M}athematics {S}oftware {S}ystem
  ({V}ersion 8.0.0)}, 2017.
\newblock {\tt http://www.sagemath.org}.

\bibitem[Sir00]{Sir1}
V\'\i ctor~F. Sirvent.
\newblock Geodesic laminations as geometric realizations of {P}isot
  substitutions.
\newblock {\em Ergodic Theory Dynam. Systems}, 20(4):1253--1266, 2000.

\bibitem[Sir03]{Sir2}
V\'\i ctor~F. Sirvent.
\newblock Geodesic laminations as geometric realizations of {A}rnoux-{R}auzy
  sequences.
\newblock {\em Bull. Belg. Math. Soc. Simon Stevin}, 10(2):221--229, 2003.

\bibitem[Via06]{viana-survey}
Marcelo Viana.
\newblock Ergodic theory of interval exchange maps.
\newblock {\em Rev. Mat. Complut.}, 19(1):7--100, 2006.

\bibitem[Yoc05]{yocc-2005}
Jean-Christophe Yoccoz.
\newblock Échanges d'intervalles.
\newblock Cours au Collège de France, 2005.

\end{thebibliography}

\newcommand{\etalchar}[1]{$^{#1}$}

\end{document}